\documentclass[11pt]{amsart}
\usepackage{amsfonts, bm}
\usepackage{mathrsfs}
\allowdisplaybreaks

\usepackage{makecell}
\usepackage{hyperref}
\usepackage{enumitem}
\makeatletter

\renewcommand{\tocsection}[3]{%
	\indentlabel{\@ifnotempty{#2}{\bfseries\ignorespaces#1 #2\quad}}\bfseries#3}

\usepackage{amsmath}
\usepackage{amssymb}
\usepackage{multicol}
\usepackage{stmaryrd}
\usepackage{cite}
\usepackage{epsfig}
\usepackage{color}
\usepackage{graphics}
\usepackage{graphicx}
\usepackage{multicol,graphics}
\newcommand\bes{\begin{eqnarray}}
	\newcommand\ees{\end{eqnarray}}

\newtheorem{theorem}{Theorem}[section]
\newtheorem{lemma}[theorem]{Lemma}
\newtheorem{corollary}[theorem]{Corollary}

\newtheorem{remark}[theorem]{Remark}
\newtheorem{proposition}[theorem]{Proposition}
\numberwithin{equation}{section}

\theoremstyle{plain}
\newtheorem*{theorem*}{Theorem A}

\newcommand\bess{\begin{eqnarray*}}
	\newcommand\eess{\end{eqnarray*}}

\newcommand\yy{\infty}

\begin{document}
	
	\title[Trichotomy dynamics for biological invasion ]{Trichotomy dynamics of a free boundary model\\ for biological invasion}
	\author[H. Cao, Y. Du, W. Ni, X. Zhang ]{Hongkai Cao$^\dag$,  Yihong Du$^\ddag$, Wenjie Ni$^{\ddag}$ and Xiaoyan Zhang$^\dag$}
	\thanks{\hspace{-.5cm}
		\mbox{\ \ $^{\dag}$} School of Mathematics, Shandong University, Jinan, China.
		\\
		\mbox{\ \ $^\ddag$} School of Science and Technology, University of New England, Armidale, NSW 2351, Australia.
		\\
		\mbox{\small \ \ \ \  Emails:} chk@mail.sdu.edu.cn (Cao),\ ydu@une.edu.au (Du),\ wni2@une.edu.au (Ni), \ zxysd@sdu.edu.cn (Zhang)}
	\thanks{ The research of Du and Ni was supported by the Australian Research Council.
		The research of Cao and Zhang was supported by Natural Science Foundation of China (No.11571200) and the Natural Science Foundation of Shandong Province (No. ZR2021MA062)
	}
	
	\date{\today}

	\begin{abstract} It is well known that the reaction-diffusion equation $u_t=du_{xx}+f(u)$ with compactly supported nonnegative initial functions exhibits  trichotomy dynamics for bistable and combustion type $f(u)$ \cite{DM, zlatos}.	The same is true for the corresponding Stefan type free boundary problem \cite{DL}.	 In this paper, we reveal a rather different type of trichotomy for this reaction-diffusion equation under a new set of (free) boundary conditions,  arising as a model for biological invasion with $u(t,x)$ representing the density of an invading species over the one dimensional spatial regin  $[0, h(t)]$. 
	The evolution of the invading front $x=h(t)$ is governed by  $h'(t)=-\frac d\delta u_x(t, h(t))$ and $u(t, h(t))=\delta\in (\hat\theta_f, 1)$, with $\hat\theta_f \in [0, 1)$  uniquely determined by $f$; they allow $h(t)$ to advance as well as to retreat when time increases. At the fixed boundary $x=0$, the density is controlled by $u(t,0)=\delta_0\geq 0$.  We completely classify the long-time dynamics of the model when $f(u)$ is a monostable, or bistable, or combustion type nonlinear function.  In the biologically interesting case that $\delta_0<\delta$,  we show that there are exactly three scenarios: (i) successful spreading, (ii) finite-time vanishing, (iii) a transition state characterized by $h(t)\to l_*\in (0, \infty)$ and $u(t,x)\to w_*(x)$ as $t\to\infty$, where $(u(t,x), h(t))\equiv (w_*(x), l_*)$ is the unique stationary solution of the free boundary problem. The model here does not have the usual order-preserving property enjoyed by those considered in \cite{DM, zlatos, DL} and elsewhere (i.e., $u(0,x)\leq v(0,x)$ implies $u(t,x)\leq v(t,x)$ for all $t>0$ if $u$ and $v$ are two solutions of the problem), which  is intrinsically linked to the many novel features of the model.

		\bigskip
		
		\medskip
		
		
	\end{abstract}
	
	\maketitle
	\tableofcontents

	\section{Introduction}
	In this paper, we investigate the long-time behaviour of the following free boundary problem:
	\begin{equation}\label{1.1}
		\left\{\begin{array}{ll}
			u_{t}-d u_{x x}=f(u),\, & t>0,\, 0<x<h(t), \\
			u(t, 0)=\delta_{0},\,u(t, h(t))=\delta, & t>0, \\
			h^{\prime}(t)=-\frac{d}{\delta} u_{x}(t, h(t)),\, & t>0, \\
			h(0)=h_{0}, \,u(0, x)=u_{0}(x),\, & 0 \leq x \leq h_{0}.
		\end{array}\right.
	\end{equation}
	Here $f$ is a monostable, or bistable, or combustion type nonlinear function (to be made precise below), $d$, $h_0$  and $\delta$ are given positive constants with $\delta\in (\hat\theta_f,1)$,  where $\hat \theta_f\in [0, 1)$ is uniquely determined by $f$, and $\delta_0$ is a nonnegative constant.
	The initial function $u_{0}(x)$ is assumed to belong to $\mathcal{X}(h_{0})$, defined by
	\[\mathcal{X}(h_{0}):=\left\{\phi \in C^{2}\left(\left[0, h_{0}\right]\right): \phi(x)>0 \text { in }\left(0, h_{0}\right],\,\phi\left(0\right)=\delta_{0},\, \phi\left( h_{0}\right)=\delta\right\}.\]
	
	Problem \eqref{1.1} arises as a model for the spreading of an invasive species subject to an intervention, which turns out to be an intriguing mathematical problem, partly due to the lack of the usual order-preserving property (see \eqref{order} below for a precise description). The lack of this property is intrinsically linked to the fact that the population range $[0, h(t)]$ can expand as well as  shrink, which is a key new feature of the model, and is also the cause of numerous technical difficulties in understanding its behaviour. By making use of several new techniques, we obtain a  complete classification of the long-time dynamics of \eqref{1.1}. In particular, we reveal a trichotomy dynamics  in the biologically interesting case $\delta_0<\delta$; namely, as the initial data $(u_0, h_0)$ vary, the long-time dynamics has exactly three scenarios: 
	\begin{itemize}
		\item[(i)] successful spreading, 
		\item[(ii)] finite-time vanishing, 
		\item[(iii)] a transition state characterized by $h(t)\to l_*\in (0, \infty)$ and $u(t,x)\to w_*(x)$ as $t\to\infty$. 
	\end{itemize}
	The precise statement of our results is given in subsection 1.5 below, after some explanations on how \eqref{1.1} arises from an  invasion biology background.
	\medskip
	
	Trichotomy dynamics of reaction-diffusion equations of the type
	\begin{equation}\label{eq-kpp}
		u_t - d u_{xx} = f(u) \mbox{ for }  t>0,\ x\in\mathbb R,
	\end{equation}
has attracted considerable attention in the last a few decades. Such dynamics was conjectured by Kanel \cite{Ka} for \eqref{eq-kpp}  with compactly supported nonnegative initial function $u(0,x)$  when $f$ is of combustion type, namely for any such initial function, only three types of long-time behaviour can occur: As $t\to\infty$, for small initial function $u(0,x)$, $u(t,x)\to 0$; for large $u(0,x)$, $u(t,x)\to 1$; and for intermittent $u(0,x)$, 	$u(t,x)\to \theta\in (0,1)$, where $\theta\in (0,1)$ is the 
ignition temperature given by $f$. For bistable type of $f$, Aronson and Weinberger \cite{AW78} had a similar conjecture. These conjectures have been positively answered by Zlatos \cite{zlatos} for a special class of initial functions, and by Du and Matano \cite{DM} for rather general classes of initial functions. For example, if $u(0,x)=\phi_0(x)$ leads to vanishing ($u(t,x)\to 0$) and $u(0,x)=\phi_1(x)\geq \phi_0(x)$ leads to spreading ($u(t,x)\to 1$), then it follows from \cite{DM} that there exists a unique $\sigma^*\in (0,1)$ such that
 \begin{equation}\label{u_sigma}
u(0,x)=\phi_\sigma(x):=(1-\sigma)\phi_0(x)+\sigma\phi_1(x)
\end{equation}
 leads to vanishing when $\sigma\in [0, \sigma^*)$, leads to spreading when $\sigma\in (\sigma^*, 1]$, and leads to a  transition state when $\sigma=\sigma^*$. For the corresponding Stefan type free boundary problem (see \eqref{A} below), a similar trichotomy was obtained by Du and Lou \cite{DL}. As we will see in subsection 1.6 below, the trichotomy for \eqref{1.1} is rather different in nature, with a deep reason linked to the lack of the usual order-preserving property of \eqref{1.1} mentioned above. Some other  works on trichotomy dynamics can be found in  \cite{ADF, MW, P, WZZ} and the references therein. To the best of our knowledge,  this paper seems the first to prove a trichotomy for a model without the usual order-preserving property.\medskip
	
	Let us now briefly recall several models for propagation dynamics and see how \eqref{1.1} arises from an invasion biology consideration.
	
	\subsection{The Fisher-KPP model}
	
	Reaction-diffusion models of the form \eqref{eq-kpp} for species propagation can be traced back to the classical works of Fisher \cite{Fisher} and Kolmogorov-Petrovskii-Piskunov (KPP) \cite{KPP}, who in 1937 independently used traveling wave solutions to understand the propagation properties of what is now known as the Fisher-KPP equation, which is \eqref{eq-kpp} with
	$f$ satisfying 
	\[
	(\mathbf{f_m}):\ \ 
	f(0)=f(1)=0,~ f'(0)>0>f'(1),\ f(u)>0 \text{ for } u\in(0,1), f(u)<0 \text{ for } u\in(1,\infty),
	\]
	and 
	\[
	(\mathbf{f_{KPP}}):\ \ \ \ f(u) \leq f'(0)u \text{ for } u\in [0,1].
	\]
	It was shown  in \cite{Fisher, KPP} that \eqref{eq-kpp} admits traveling wave solutions of the form \( u(t,x) = \Phi(x - ct) \) if and only if \( c \geq c^* \), where \( c^* =2\sqrt{d\,f'(0)}\) is called the minimal speed. 
	Moreover, \( c^* \) coincides with the asymptotic spreading speed of solutions with compactly supported initial data, namely (see Aronson-Weinberger \cite{AW78}), for any $0<\epsilon\ll 1$,
	\[\begin{cases}
		\lim_{t\to\infty} \sup_{|x|\geq (c^*+\epsilon)t} u(t,x)=0,\\
		\lim_{t\to\infty}\sup_{|x|\leq (c^*-\epsilon)t}|u(t,x)-1|=0.
	\end{cases}
	\]
	These conclusions on $u(t,x)$ imply that for any $\sigma\in (0,1)$, the set
	\[
	\Omega_\sigma(t):=\{x\in\mathbb R: u(t,x)>\sigma\}
	\]
	spreads to the entire space $\mathbb R$ with asymptotic speed $c^*$\footnote{If only the monostable condition ${\bf (f_m)}$ holds, the above conclusions remain valid except for the formula $c^*=2\sqrt{d\ f'(0)}$, which holds if additionally  the KPP condition ${\bf (f_{KPP})}$ is satisfied.}. The set $\Omega_\sigma(t)$ can be interpreted  biologically  as the population range of the species at time $t$
	modelled by \eqref{eq-kpp}, where $u(t,x)$ stands for the density of the propagating species at time $t$ and spatial location $x$. The assumption that $u(0,x)$ is compactly supported means that the population range is a bounded set initially. 
	
	A more natural choice of the population range at time $t$ is $\Omega_0(t):=\{x\in\mathbb R: u(t,x)>0\}$, but  $\Omega_0(t)\equiv \mathbb R$ for $t>0$ even though $\Omega_0(0)$ is bounded by assumption. Therefore to capture the propagation dynamics via \eqref{eq-kpp}, it is necessary to use $\Omega_\sigma(t)$ rather than $\Omega_0(t)$ to stand for the population range.
	However, it is easily seen that  for any fixed $t>0$, $\Omega_\sigma(t)\to\mathbb R$ as $\sigma\to 0$. Therefore \eqref{eq-kpp} does not give the precise population range for positive time, although it captures the crucial propagation speed $c^*$, which is independent of $\sigma$.
	
	The exhibition of a propagation speed is a fundamental feature of \eqref{eq-kpp}, representing a phenomenon observed in numerous real world examples of species invasion via their range expansion  (after Skellam \cite{Skellam} in 1951), with  the spreading of muskrats in Europe at the beginning of the 20th century  perhaps the most well-known (see, e.g., \cite{SK, Skellam}). Such a propagating behaviour also arises in other areas, such as spreading of flames in combustion theory, etc. 
	
	These pioneering works  have inspired extensive further research along many lines which continues to this day, including, but not limited to, research on a broad range of problems with various nonlinearities and inhomogeneous media; see, as a small sample, \cite{AW78, BH03, BHM, Bramson, DM, dgm, FZ, FM,  HNRR16,  LZ, N09, P2, RRR, shen10, W82, W02, Xin, zlatos} and the references therein. This paper continues the research along the line of free boundary models.
	
\subsection{The Stefan type free boundary model}	It is possible to modify \eqref{eq-kpp} into a free boundary problem so that the population range is precisely described in the model.	In 2010, Du and Lin \cite{D} considered such a model, which has the following form:
	\begin{equation}\label{A}
		\left\{\begin{array}{ll}
			u_{t}-d u_{x x}=f(u), & t>0,\, g(t)<x<h(t), \\
			u(t, g(t))=u(t, h(t))=0, & t>0, \\
			g'(t)=-\mu u_x(t, g(t)),\ h^{\prime}(t)=-\mu u_{x}(t, h(t)), & t>0, \\
			g(0)=-h_0, \ h(0)=h_{0}, \, u(0, x)=u_{0}(x), & -h_0 \leq x \leq h_{0},
		\end{array}\right.
	\end{equation}
	where $[g(t), h(t)]$ stands for the population range at time $t$, and $\mu$ is a positive constant. The evolution of the free boundaries $x=g(t)$ and $x=h(t)$ is governed by the equations
	\[
	\begin{cases} u(t, g(t))=0,\ g'(t)=-\mu u_x(t, g(t)), \\
		u(t, h(t))=0,\ h^{\prime}(t)=-\mu u_{x}(t, h(t)),  \\
	\end{cases}
	\]
	which coincide with the Stefan conditions used in classical one-phase free boundary problems describing the melting of ice in contact with water (see \cite{Rub}). In a biological context, where very few first principles are available to help with the modelling, one may deduce these equations from some reasonable biological assumptions (see, for example, \cite{BDK}).

	In \cite{D}, only the special logistic function $f(u)=u(a-bu)$ was considered, where $a$ and $b$ are positive constants.	In 2015, Du and Lou \cite{DL} extended the model to include rather general $f(u)$, including monostable, bistable, and combustion type nonlinearities.  
	
	Let us recall that a $C^1$ function $f(u)$ is called \underline{monostable} if ${\bf (f_m)}$ holds, and it is called \underline{bistable} if it satisfies
	\[
	{\bf (f_b):} \ \ \ \begin{cases} f(0)=f(\theta)=f(1)=0 \mbox{ for some } \theta\in (0,1),\ f'(0)<0,\ f'(1)<0, \\
		f(u)<0 \mbox{ for } u\in (0,\theta)\cup (1,\infty),\ f(u)>0 \mbox{ for } u\in (\theta, 1),\
		\int_0^1f(u)du>0.
	\end{cases}
	\]
	A \underline{combustion type} of $f(u)$ is a $C^1$ function with the following properties,
	\[
	{\bf (f_c):} \ \ \ \begin{cases}  
		f(u)=0 \mbox{ for } u\in [0,\theta] \mbox{ and some } \theta\in (0,1),\\ f(1)=0>f'(1),\ f(u)>0 \mbox{ for } u\in (\theta, 1),\ f(u)<0 \mbox{ for } u>1.
	\end{cases}
	\]
	
	A bistable $f(u)$ is often used to reflect the so called Allee effect in biology, namely low population density may negatively impact the population growth due to reduced chances to reproduce, etc. In this context, the constant $\theta\in (0,1)$ in ${\bf (f_b)}$ is also known as the threshold Allee density. A combustion type $f(u)$ 
	can be used to reflect weak Allee effect  ($\theta$ is known traditionally as the ignition temperature). 
	
	It is easily seen that if $f(u)$ satisfies ${\bf (f_b)}$ or ${\bf (f_c)}$, then there exists a unique $\theta^{*} = \theta_{f}^{*} \in [\theta, 1)$ such that
\begin{equation}\label{theta*}
	\int_{0}^{\theta^{*}} f(s) \,ds = 0, \qquad \int_{0}^{u} f(s) \,ds > 0 \quad \text{for } u \in (\theta^{*}, 1].
\end{equation}
For convenience of presentation, we define
\[
\hat\theta_f:=\begin{cases} 0 &\mbox{ if $f$ satisfies ${\bf (f_m)}$},\\
\theta^*_f&\mbox{ if $f$ satisfies ${\bf (f_b)}$ or ${\bf (f_c)}$}.
\end{cases}
\]
Moreover, we say $f(u)$ satisfies ${\bf (F)}$ if it is a $C^1$ function satisfying ${\bf (f_m)}$, or ${\bf (f_b)}$, or ${\bf (f_c)}$.
	
	With these three types of nonlinearities,  as time goes to infinity, generically the species modelled by
	\eqref{A} either spreads successfully, or vanishes; namely, as $t\to\infty$, either
	\begin{itemize}
		\item {\bf Spreading:}
		$\begin{cases} (g(t), h(t))\to (-\infty, \infty),\\
			u(t,x)\to 1 \mbox{ locally uniformly in } x\in\mathbb R,
		\end{cases}$\\
		or 
		
		\item {\bf Vanishing:} $\begin{cases} (g(t), h(t))\to (g_\infty, h_\infty), \mbox{ which is a finite interval},\\
			u(t,x)\to 0 \mbox{  uniformly for } x\in [g(t), h(t)].
		\end{cases}$
	\end{itemize}
	\noindent
	To be precise, for bistable and combustion type of $f(u)$, a non-generic transition case also occurs, revealing a trichotomy dynamics of \eqref{A}; see \cite{DL} for details.
	
	It was further shown in \cite{DL} that for the  three types of $f(u)$ defined above, for every $\mu>0$, there exists a unique pair \( (c, q)=\left(c_{*}, q_{*}\right) \), with \( c_{*}=c_{*}(\mu)>0 \), satisfying
	\begin{equation}\label{sw}
		\left\{\begin{array}{l}
			d q^{\prime \prime}-c q^{\prime}+f(q)=0,\, q>0 \, \text{ in }\,(0,\, \infty), \\
			q(0)=0,\, q(\infty)=1,\, q^{\prime}\left(0\right)=\frac c\mu.
		\end{array}\right.
	\end{equation}
	Moreover, \( q_{*}^{\prime}(z)>0 \) for all \( z \geq 0 \), \( c_{*}(\mu)<c_{0} \), and \( \lim _{\mu \rightarrow \infty} c_{*}(\mu)=c_{0} \), where \( c_{0} \) denotes the asymptotic spreading speed of the corresponding Cauchy problem \eqref{eq-kpp}.
	Furthermore, $c_*(\mu)$ is the asymptotic spreading speed of \eqref{A}: if successful spreading happens, then
	\[
	\lim_{t\to\infty} \frac{h(t)}t=\lim_{t\to\infty}\frac{g(t)}{-t}=c_*(\mu).
	\]

	Sharp propagation profile of \eqref{A} was  obtained in \cite{DMZ}, namely the following conclusions hold when spreading is successful: As $t\to\infty$,
	\[
	\begin{cases}
		h(t)-c_*t\to h^0,\ g(t)+c_*t\to g^0,\\
		u(t,x)-q_*(h(t)-x)\to 0 \mbox{ uniformly for } x\in [0, h(t)],\\
		u(t,x)-q_*(x-g(t))\to 0 \mbox{ uniformly for } x\in [g(t), 0],
	\end{cases}
	\]
	where $g^0$ and $h^0$ are finite numbers (depending on the initial function $u_0$).
	\medskip
	
	\subsection{A new  free boundary model} Recently,  Du \cite{Du2023a} considered a modified version of \eqref{A}, which has the following form:
	\begin{equation}\label{B}
		\left\{\begin{array}{ll}
			u_{t}-d u_{x x}=f(u), & t>0, \,g(t)<x<h(t), \\
			u(t, g(t))=u(t, h(t))=\delta, & t>0, \\
			g^{\prime}(t)=-\frac{d}{\delta} u_{x}(t, g(t)),  \
			h^{\prime}(t)=-\frac{d}{\delta} u_{x}(t, h(t)), & t>0, \\
			-g(0)=h(0)=h_{0}, \,u(0, x)=u_{0}(x), \,& -h_{0} \leq x \leq h_{0},
		\end{array}\right.
	\end{equation}
	where  $f$ satisfies ${\bf (f_m)}$ and $\delta\in (0,1)$.
	In \eqref{B}, the evolution of the free boundaries $x=g(t)$ and $x=h(t)$ is determined by
	\[\begin{cases}
		u(t, h(t))=\delta,\ h^{\prime}(t)=-\frac{d}{\delta} u_{x}(t, h(t)), \\
		u(t, g(t))=\delta,\ g^{\prime}(t)=-\frac{d}{\delta} u_{x}(t, g(t)).
	\end{cases}\]
	These equations can be deduced from the biological assumption that  the positions of the fronts $x=g(t)$ and $x=h(t)$ are balanced  to keep the population density there at some preferred level $\delta$; see \cite{Du2023a} for details. 
	\medskip
	
	Although \eqref{B} looks similar to \eqref{A}, it has several very distinct features.
	The first striking difference  is that the population range $[g(t), h(t)]$ in \eqref{B} may expand as well as shrink when time increases, while the population range of \eqref{A} can only expand due to the strong maximum principle and Hopf boundary lemma: For $t>0$,
	\[
	\begin{cases}\mbox{$u(t,x)>0$ for $g(t)<x<h(t)$}\\
		\mbox{and $u(t,g(t))=u(t, h(t))=0$}
	\end{cases} \implies \ \ \  \begin{cases} \mbox{$h'(t)=-\mu u_x(t, h(t))>0,$}\\
		\mbox{$g'(t)=-\mu u_x(t, g(t))<0$.}\end{cases}
	\]
	
	Secondly, both \eqref{A} and \eqref{eq-kpp} have the following order-preserving property:
	\begin{equation}\label{order}
		u_0(x)\leq \tilde u_0(x)\implies u(t,x)\leq \tilde u(t,x) \mbox{ for all } t>0,
	\end{equation}
	namely, if the initial size of one population $u$ is no bigger than that of another population $\tilde u$, then this order is preserved for all later time $t>0$.
	This property is lost for \eqref{B}, which induces new dynamical behaviour, and also causes considerable technical difficulties for the analysis of \eqref{B}.
	
	Thirdly, it was shown in \cite{Du2023a} that for \eqref{B}, spreading is always successful when ${\bf (f_m)}$ holds, while for \eqref{A}, a spreading-vanishing dichotomy holds for monostable $f$ (see \cite{DL}). Moreover, it was shown in \cite{DLNS} that \eqref{B} exhibits the behaviour of some super invaders: The Allee effect is eliminated if the preferred density level $\delta$ at the range boundary is chosen properly.  More precisely, if the Allee effect is included by assuming $f(u)$ satisfying ${\bf (f_b)}$ (for strong Allee effect) or satisfying ${\bf (f_c)}$ (for weak Allee effect), then for every admissible initial population $u_0\in\mathcal X(h_0)$, \eqref{B} always exhibits successful spreading as for the case of monostable $f(u)$ where no Allee effect is included by assumption, provided that $\delta\in (\theta_f^*, 1)$, where $\theta_f^*$ is given by \eqref{theta*}. Let us recall that, in contrast,   the Affee effect is reflected by the dynamics of both the earlier models \eqref{eq-kpp} and \eqref{A}: small initial population leads to vanishing when ${\bf (f_b)}$ or ${\bf (f_c)}$ holds.

	Let us now recall the main results in \cite{Du2023a} and \cite{DLNS} for \eqref{B};  the model \eqref{1.1} will arise naturally from these results. The existence of a unique global solution of \eqref{B} under rather general assumptions has been established in \cite{DLNS}. Consistent spreading is proved in Theorem 1.1 of \cite{DLNS}, which  is an extension of the corresponding result in \cite{Du2023a}, where only monostable $f$ was considered. We restate the conclusions on the long-time dynamics of \eqref{B} below. \medskip
	
	\noindent
	{\bf Theorem A.} (\underline{Successful invasion}) {\it
		Suppose that $f$ satisfies ${\bf (f_m)}$, or ${\bf (f_b)}$, or ${\bf (f_c)}$,  and $\delta\in (\hat \theta_f, 1)$.
		Then for every initial function $u_0 \in \mathcal X(h_0)$, the unique solution $(u(t,x), g(t), h(t))$ of \eqref{B} satisfies, as $t \to \infty$,
		\begin{align*}
			(g(t), h(t)) \to (-\infty, \infty), \quad u(t, x) \to 1 \quad \text{locally uniformly in } x \in \mathbb{R}.
		\end{align*}
		Furthermore, there exist some constants $\hat{h}, \hat{g} \in \mathbb{R}$ $(\mbox{depending on $u_0$})$ such that
		\begin{align*}\begin{cases}
				\lim_{t \to \infty}[h(t) - c_* t] = \hat{h}, \quad \lim_{t \to \infty} h'(t) = c_*,\\
				\lim_{t \to \infty}[g(t) + c_* t] = \hat{g}, \quad \lim_{t \to \infty} g'(t) = -c_*,\\
				\lim_{t \to \infty} \sup_{x \in [0, h(t)]} |u(t, x) - q_*( h(t) - x)| = 0,\\
				\lim_{t \to \infty} \sup_{x \in [g(t), 0]} |u(t, x) - q_*( x - g(t))| = 0,
			\end{cases}
		\end{align*}
		where $(c_*, q_*(x))$ is the unique solution pair $(c,q(x))$ of 
		\begin{equation}\label{semi1}
			\begin{cases}
				dq'' - cq' + f(q) = 0, \quad q > 0 \quad \text{in } (0,\infty), \\
				q(0) = \delta, \quad q(\infty) = 1, \quad q'(0) = \frac{c \delta}{d},\quad q'>0 \mbox{ in } [0,\infty).
			\end{cases}
		\end{equation}
	}
	
	It follows from Theorem 3.4 of \cite{DLNS} that vanishing always happens if $\delta>1$ in Theorem A. More precisely we have the following conclusion:\medskip
	
	\noindent {\bf Theorem B.} (\underline{Vanishing})
	{ In Theorem A, if the assumption  $\delta\in (\hat\theta_f,1)$ is replaced by $\delta>1$, 
		then for every  $u_0 \in \mathcal X(h_0)$, vanishing happens to \eqref{B}, namely, there exists $\xi^*\in\mathbb R$ depending on $f$ and $u_0$ such that 
		\[
		\lim_{t\to\infty} g(t)=\lim_{t\to\infty} h(t)=\xi^*, \ \lim_{t\to\infty} u(t,x)=\delta \mbox{ uniformly for } x\in [g(t), h(t)].
		\]
	}
	
	When $\delta=1$, the dynamics of \eqref{B} exhibits a  behaviour between successful spreading (as in Theorem A) and vanishing (as in Theorem B), which we call a surviving state and  is described by the following theorem (which is a consequence of Theorem 3.5 in \cite{DLNS}).  
	\medskip
	
	\noindent
	{\bf Theorem C.} (\underline{Surviving})
	{\it If $\delta=1$ in Theorem A, then for every initial function $u_0 \in \mathcal X(h_0)$, the unique solution $(u(t,x), g(t), h(t))$ of \eqref{B}
		satisfies, as $t \to \infty$,
		\[\begin{cases}
			[g(t), h(t)] \to [g_\infty, h_\infty] \mbox{ is a finite interval with $h_\infty>g_\infty$, }\smallskip\\
			\   u(t, x) \to 1 \quad \text{ uniformly for } x \in [g(t), h(t)].
		\end{cases}	
		\]}
	
	\medskip
	
	\noindent {\bf Remark:}	
The constant {$\hat\theta_f$} in Theorem A is sharp:

 ${\bf (f_c)}$ hods \& {$\delta\in (0, \hat\theta_f]$} $\implies$ $(u, g, h)\equiv (\delta, -h_0, h_0)$ is a {stationary solution} of \eqref{B}.
 
  ${\bf (f_b)}$ holds \& {$\delta\in (\theta, \hat\theta_f)$} $\implies$ \eqref{B}  has infinitely many stationary solutions (see \cite{DLNS} for details).
\medskip

Problem \eqref{B} for $\delta\in (0, \hat\theta_f)$ with $f$ satisfying ${\bf (f_b)}$ or ${\bf (f_c)}$ is still not completely understood for general  initial data. 
On the other hand, \eqref{B} in some heterogeneous media has been investigated theoretically in \cite{DMW} (for the time periodic case with monostable nonlinearity) and in \cite{DHL} (for the case of shifting environment), and numerically in \cite{SSD} (when the environment has an invasion barrier).

	\subsection{Intervention and the resulting model \eqref{1.1}}
	
	In this paper, we explore a variation of \eqref{B} in a different direction from \cite{DMW, DHL, SSD}. Our starting point is the scenario of Theorem A, where successful invasion always happens. We are interested in  how intervention might change the long-time dynamics. 
	
	We consider a very simple  intervention (or control) from some time $t_0\geq 0$,  over a fixed region $[\tilde g_0, \tilde h_0]\subset (g(t_0),  h(t_0))$, where  the population density $u(t,x)$ is kept at a fixed level $\delta_0\geq 0$  for all $t\geq t_0$:
	\[
	u(t,x)=\delta_0 \mbox{ for } t\geq t_0,\ x\in [\tilde g_0, \tilde h_0].
	\]	
	Under this intervention, the dynamics of \eqref{B} for $t>t_0$  is determined by the following two independent problems:
	\[
	\left\{\begin{array}{ll}
		U_{t}-d U_{x x}=f(U),\, & t>t_0,\, \tilde h_0<x<H(t), \\
		u(t, \tilde h_0)=\delta_{0},\,U(t, H(t))=\delta, & t>t_0, \\
		H^{\prime}(t)=-\frac{d}{\delta} U_{x}(t, H(t)),\, & t>t_0, \\
		H(t_0)=h(t_0),\ U(t_0,x)=u(t_0,x), & x\in [\tilde h_0, h(t_0)],
	\end{array}\right.
	\]
	and 
	\[
	\left\{\begin{array}{ll}
		V_{t}-d V_{x x}=f(V),\, & t>t_0,\, G(t)<x<\tilde g_0, \\
		V(t, \tilde g_0)=\delta_{0},\,V(t, G(t))=\delta, & t>t_0, \\
		G^{\prime}(t)=-\frac{d}{\delta} V_{x}(t, G(t)),\, & t>t_0, \\
		G(t_0)=g(t_0),\ V(t_0,x)=u(t_0,x), & x\in [g(t_0), \tilde g_0],
	\end{array}\right.
	\]
	
	It is easy to see that by a suitable change of variables, the second system can be transformed into a system of the same form as the first. Therefore we only need to investigate the first. 
	
	To simplify notations, we rewrite the first intervention system in the following standard form:
	\begin{equation*}
		\left\{\begin{array}{ll}
			u_{t}-d u_{x x}=f(u),\, & t>0,\, 0<x<h(t), \\
			u(t, 0)=\delta_{0},\,u(t, h(t))=\delta, & t>0, \\
			h^{\prime}(t)=-\frac{d}{\delta} u_{x}(t, h(t)),\, & t>0, \\
			h(0)=h_{0}, \,u(0, x)=u_{0}(x),\, & 0 \leq x \leq h_{0}.
		\end{array}\right.
	\end{equation*}
	Clearly this is exactly \eqref{1.1}, except that the initial function inherited from the problem for $(U, H)$ may not satisfy $u_0(0)=\delta_0$, but this does not cause any problem for the mathematical analysis; for example,  we can simply replace $t_0$ by any $\tilde t_0>t_0$ to avoid this inconvenience. It is also natural from the biological point of view to allow some transition time (from $t_0$ to $\tilde t_0$) for the intervention to take effect. 	
	
	From now on,	we will treat \eqref{1.1} as a new model and call it a free boundary model for invasion with intervention. We want to determine how the choice of $\delta_0$ affects the dynamics of \eqref{1.1}. It turns out that this is a highly nontrivial task.

	\subsection{Main results for \eqref{1.1}}
	
	Our main results for problem \eqref{1.1} are listed below. 	
	
	\begin{theorem}[\underline{Existence and uniqueness}]\label{th1.1}
		Suppose that $f$ satisfies ${\bf (f_m)}$, or ${\bf (f_b)}$, or ${\bf (f_c)}$, $u_{0} \in \mathcal{X}(h_{0})$,   $\delta_0\geq 0$ and $\delta>0$. Then \eqref{1.1} admits a unique solution $(u,h)$ for $t\in (0, T_{\mathrm{max}})$, and
		\begin{align*}
			(u,  h) \in C^{1+\frac{\alpha}{2}, 2+\alpha}(\Omega_{T_{\mathrm{max}}}) \times C^{1+\frac{\alpha}{2}}((0, T_{\mathrm{max}})),
		\end{align*}
		where $ \alpha \in (0, 1)$, $\Omega_{T_{\mathrm{max}}}:=\left\{(t, x) \in \mathbb{R}^{2}: t \in(0, T_{\mathrm{max}}), \,x \in[0, h(t)]\right\},$
		and  $T_{\mathrm{max}}\in (0,+\infty]$ is  the maximal existence time of the solution.
		
	\end{theorem}

	\begin{theorem}[\underline{Long-time dynamics}]\label{th1.2b}
		Assume the conditions of Theorem \ref{th1.1} are satisfied with $\delta\in(\hat\theta_f,1)$. 
		\begin{enumerate}
			\item[\textup{\bf (a)}] If $\delta_{0} \geq \delta>0$, then successful spreading always happens, i.e.
			\[
			\mbox{$T_{\mathrm{max}} = \infty$ and } \begin{cases}  \lim_{t \to \infty} h(t) = \infty,\\
				\lim_{t \to \infty} u(t, x) = v(x) \mbox{ locally uniformly in $[0,\infty)$,}\end{cases}
			\]
			where $v(x)$ is the unique solution  of  
			\begin{equation}\label{1.3}
				\begin{cases}
					dv^{\prime \prime} + f(v) = 0, & x \in (0, \infty), \\
					v(0) = \delta_{0},\  v(\infty)=1.
				\end{cases}
			\end{equation}
			
			\item[\textup{\bf (b)}] If $0\leq \delta_{0} <\delta<1$, then one of the following
			must happen:\smallskip
			
			\begin{enumerate}
				\item[\textup{(i)}] \underline{Finite-time vanishing}: 
				\[\mbox{$T_{\mathrm{max}} < \infty$ and $\displaystyle \lim_{t \to T_{\mathrm{max}}^{-}} h(t) = 0$.}
				\]
				
				\item[\textup{(ii)}] \underline{Transition}: 				\[
				\mbox{$T_{\mathrm{max}} = \infty$ and }	\begin{cases}\lim_{t \to \infty} h(t) = l_*\in (0, \infty),\\
					\lim_{t \to \infty} \sup_{x \in [0, h(t)]} |u(t, x) - w_*(x)| = 0,
				\end{cases}
				\]
				where 
				\begin{equation}\label{l*}
					l_*:= \displaystyle\int_{\delta_0}^\delta \Big[\frac{2}{d} \int_x^\delta f(s)  ds\Big]^{-1}dx,\end{equation}
				and $w_*$ is the unique solution of the following problem
				\begin{equation}\label{1.7}
					\left\{
					\begin{array}{ll}
						d w_*'' + f(w_*) = 0, & x \in (0, l_*), \\
						w_*(0) = \delta_0, \, w_*(l_*) = \delta, \, w_*'(l_*) = 0.
					\end{array}
					\right.
				\end{equation}
				
				\item[\textup{(iii)}] \underline{Successful spreading}: 
				\[
				\mbox{$T_{\mathrm{max}} = \infty$ and } \begin{cases} \displaystyle \lim_{t \to \infty} h(t) = \infty,\\
					\lim_{t \to \infty} u(t, x) = v(x) \mbox{ locally uniformly in $[0,\infty)$,}\end{cases}
				\]
				where $v(x)$ is the unique solution  of  \eqref{1.3}.
				
			\end{enumerate}
			
		\end{enumerate}
	\end{theorem}
	
	It is easily seen that $(u(t,x), h(t))\equiv (w_*(x), l_*)$ is a stationary solution of \eqref{1.1}.
	
	Similar to \eqref{B}, the solution to problem \eqref{1.1}  does not have the usual order-preserving property \eqref{order}, which makes it particularly difficult to determine exactly when each of the three possibilities in part {\rm (b)} of Theorem \ref{th1.2b} occurs. Nevertheless, we have the following result.
	
	\begin{theorem}\label{th1.3} 
		In part {\rm (b)} of Theorem \ref{th1.2b}, the following conclusions hold.
		\begin{itemize}
			\item[{\rm (i)}] Conditions for vanishing:
			\[
			T_{\mathrm{max}} < \infty \quad \iff \quad 
			\int_{0}^{h(t_{0})} x u(t_{0}, x)  dx < \frac{d(\delta - \delta_{0})}{C_{f}}\quad\text{for some }  t_{0} \in [0, T_{\mathrm{max}}),
			\]
			where $\displaystyle C_{f} := \sup_{u \in (0,1]} \frac{f(u)}{u}$.
			\item[{\rm (ii)}] Sufficient conditions for vanishing and spreading: Let $l_*$ and $w_*$ be given in \eqref{l*} and \eqref{1.7}, respectively. If  $h_0=l_*$ and $u_0\in\mathcal X(h_0)$, then 
			\begin{equation*}
				\begin{cases}
					{u}_0(x)\geq w_*(x), {u}_0(x)\not\equiv w_*(x) \implies {\rm spreading\ occurs},\\
					{u}_0(x)\leq  w_*(x), {u}_0(x)\not\equiv w_*(x) \implies {\rm vanishing\ occurs}.
				\end{cases}
			\end{equation*}	
		\end{itemize}		
	\end{theorem}
	
	Let us observe that part (i) of Theorem \ref{th1.3} implies that if the initial population distribution $u_0\in\mathcal X(h_0)$ satisfies
	\[
	h_0 \int_0^{h_0}u_0(x)dx<\frac{d(\delta - \delta_{0})}{C_{f}},
	\]
	then vanishing must happen. 
	
	When spreading occurs, it is possible to obtain a precise description of the propagation profile as in Theorem A.
	More precisely, we can extend the techniques in \cite{Du2023a, DLNS} to show that
	if $f$ satisfies  ${\bf (f_m)}$, or ${\bf (f_b)}$, or ${\bf (f_c)}$, $\delta\in (\hat\theta_f,1)$,    and spreading happens to the solution  $(u,  h)$   of \eqref{1.1}, then 	\begin{equation}\label{precise}	\left\{\begin{array}{l}
			\lim _{t \rightarrow \infty}\left[h(t)-c_{*} t\right]=h^\infty, \,\lim _{t \rightarrow \infty} h^{\prime}(t)=c_{*}, \\
			\lim _{t \rightarrow \infty} \sup _{x \in[0, h(t)]}\left|u(t, x)-v(x)q_{*}\left(h(t)-x\right)\right|=0, \\
		\end{array}\right.\end{equation}
	where $\left(c_{*}, q_{*}\right) $ is the unique solution pair of \eqref{semi1}, $v(x)$ is the solution to \eqref{1.3}, and  $h^\infty$ is some constant depending on $u_0$.
	
	However, in view of the length of the current paper, we have refrained from including this part here. Instead, we will prove \eqref{precise}
	in a forthcoming paper, together with an analysis of \eqref{1.1} for the case $\delta\geq 1$.
	
	\subsection{Fundamental differences of \eqref{1.1} from  \eqref{A} and \eqref{B}} Let us mention several fundamental differences of the model \eqref{1.1} from \eqref{A} and \eqref{B} in the biologically interesting situation described by Theorem  \ref{th1.2b} part (b).

	Firstly the solution of \eqref{A} and \eqref{B} with any admissible initial data always exists for all $t>0$, while the maximal existence time $T_{max}$ of \eqref{1.1} may be a finite number.
	
	Secondly, although \eqref{A} is capable to exhibit a trichotomy for its long-time dynamics (in the case of ${\bf (f_b)}$ or ${\bf (f_c)}$), the trichotomy in Theorem  \ref{th1.2b} part (b) has some fundamental differences:  vanishing for \eqref{1.1} always happens  in finite time ($T_{max}<\infty$), and it is caused by the population range shrinking to one point at $t=T_{max}$; in contrast, the vanishing of \eqref{A} always happens at infinity time, with the population range converging to a finite interval and the population density converging to 0.
	
	Thirdly, due to the order preserving property enjoyed by \eqref{A}, the trichotomy for \eqref{A} can be explained relatively easily by parameterizing the initial function $u_0(x)$ as in \eqref{u_sigma}, which gives a critical value $\sigma^*$ of the parameter $\sigma$ so that successful spreading happens when $\sigma>\sigma^*$, vanishing happens when $\sigma<\sigma^*$, and the transition case happens exactly when $\sigma=\sigma^*$. The lack of the order-preserving property of \eqref{1.1} renders such a simple description of the trichotomy  impossible.
	
	\subsection{Organization of the paper}

	The remainder of this paper is organized as follows. In Section 2, we prove Theorem \ref{th1.1}, where we first prove the local existence and uniqueness in Subsection 2.1, by adapting the approaches of \cite{Du2023a} and \cite{DLNS} to the current situation. The proof of the maximal existence and uniqueness result, completed in Subsection 2.3, relies on the preparations in Subsection 2.2, which include  some crucial bounds for \( h'(t) \), where the following  functional is first used:
	\begin{equation}\label{V}
		V(t) := \int_{0}^{h(t)} x u(t, x)  dx.
	\end{equation} 
	Several proofs in this section are postponed to Section 4, because they are based on existing ideas and the techniques are not used in later sections of this paper.
	
	Section 3 is the main part of this paper, where Theorems  \ref{th1.2b} and \ref{th1.3} are proved. Here, apart from the  functional $V(t)$ in \eqref{V}, our analysis relies on several other new ingredients. The first is the stationary solution $(w_*(x), l_*)$ of \eqref{1.1}, which is  given by
	\begin{equation*}
		\begin{cases}		\displaystyle l_* = \int_{\delta_{0}}^{\delta} \frac{dx}{\sqrt{\frac{2}{d}\int_{x}^{\delta} f(s)}ds}  \in (0, \infty),\\
			\displaystyle \int_{w_*(x)}^\delta \frac{du}{\sqrt{\frac{2}{d}\int_{u}^{\delta} f(s) \, ds}}du =l_*-x \mbox{ for } x\in [0, l_*],
		\end{cases}
	\end{equation*}
	and already appeared in Theorem \ref{th1.2b}.
	
	The second ingredient is a family of ``slowly moving traveling wave pieces" obtained in the following way:
	Using the trajectory generated by $w_*(x)$\footnote{More precisely the symmetric extension of $w_*(x)$ to $[0, 2l_*]$.}  in the $q$-$p$ plane through $ q(x):=w_*(x),\ p(x):=w_*'(x)$,
	for each {small $c>0$} we obtain a perturbed trajectory, which gives a  triple $(l_c, l^c,  q_c(x))$ that solves the following equations:
	\begin{equation*}\label{0q1}
		\left\{
		\begin{array}{ll}
			d  q'' - c  q' + f( q) = 0, & x \in (-l_c,l^c), \\
			q'(x) > 0, & x \in [-l_c, 0),\\
			q'(x) < 0, & x \in (0, l^c],\\
			q( -l_c)= q(l^c) = \delta_0, \quad  q(0) = \delta, \quad  q'(0)=0.
		\end{array}
		\right.
	\end{equation*}
	For each  \(L\le -l_{c}\), $
	v_{L}(t,x):
	=  q_{c}(-x+ct+L)$ satisfies
	\[
	\partial_tv_{L}=d\partial_{xx}v_{L}+f(v_{L}) \mbox{ for } \ t\in\mathbb R,\	ct+L-l^{c}< x< ct+L+l_{c}.
	\]
	So $v_{L}(t,x)$ is a family (with parameter $L$) of traveling wave pieces of speed $c$, whose
	left boundary, peak location and right boundary are given respectively by
	{\small \[
		G(t):=ct+L-l^{c},\qquad
		I(t):=ct+L,\qquad
		H(t):=ct+L+l_{c},
		\]}
		and so
		\[
		v_L(t, G(t))=v_L(t, H(t))=\delta_0,
		\ v_L(t, I(t))=\delta \mbox{ for } t\in \mathbb R.
		\]

	We will make use of the traveling wave pieces
	$v_L(t,x)$  to design several subtle settings where the maximum principle and a zero number result  of Angenent \cite{Angenent} can be combined to
	deduce various desired conclusions leading to the completion of the proof of Theorems \ref{th1.2b} and  \ref{th1.3}. Subsection 3.2 is perhaps the most  challenging part of this paper, where it is shown that if $T_{max}=\infty$ and $h(t)\in [a,b]$ for some $0<a<b<\infty$ and all $t>0$, then $h(t)\to l_*$ as $t\to\infty$.

	\section{Existence and uniqueness}  
	\subsection{Local existence and uniqueness}
	In this subsection, we prove the following local existence and uniqueness result, using strategies similar to those in \cite{Du2023a} and \cite{DLNS}. Note that here we allow more general $f(u)$ and only require $\delta>0$, $\delta_0\geq 0$.
	\begin{theorem}\label{th2.1}
		Assume that $f$ is $C^{1}$ and satisfies $f(0) = 0$. Then, for any given $u_{0} \in \mathcal{X}(h_{0})$ and $\alpha \in (0, 1)$, there exists a constant $T > 0$ such that problem \eqref{1.1} admits a unique solution
		\[
		(u, h) \in C^{\frac{1+\alpha}{2},1+\alpha}\left(\bar{\Omega}_{T}\right) \times C^{1+\frac{\alpha}{2}}\left([0, T]\right),
		\]
		where $\Omega_{T} = \left\{(t, x) \in \mathbb{R}^{2} : t \in (0, T],  x \in [0, h(t)]\right\}$. Moreover,  
		\[\begin{cases}
		\Vert u \Vert_{C^{\frac{1+\alpha}{2},1+\alpha}\left(\bar{\Omega}_{T}\right)} + \Vert h \Vert_{C^{1+\frac{\alpha}{2}}\left([0, T]\right)} \leq C,
		\\
		h \in C^{1+\frac{1+\alpha}{2}}\left((0, T]\right), \quad u \in C^{1+\frac{\alpha}{2},2+\alpha}\left(\Omega_{T}\right), \quad u > 0 \text{ in } \Omega_{T},
		\end{cases}
		\]
		where the constants $C$ and $T$ depend only on $h_{0}$, $\alpha$ and $\Vert u_{0} \Vert_{C^{2}\left([0,h_{0}]\right)}$.
	\end{theorem}
	The  proof of this theorem is given in Section 4.
	
	\subsection{A prior bounds} In this subsection, we obtain some a priori bounds to guarantee that the local solution $(u, h)$ obtained in Theorem \ref{th2.1} can be extended until some maximal time $T_{max}$. As we will see below, these rely on the construction of delicate upper and lower solutions, and on various estimates obtained through the auxiliary function $V(t)$ given in \eqref{V}.
	
	Recall that we will say  $f$ satisfies ${\bf (F)}$ if it is a $C^1$ function satisfying ${\bf (f_m)}$, ${\bf (f_b)}$ or ${\bf (f_c)}$.	
	We first show that the density function $u(t,x)$ remains positive and bounded  as long as it exists.
	\begin{lemma}\label{le2.5}
		Suppose that $(\bf{F})$ holds, $(u,h)$ is a solution to \eqref{1.1} defined for $t\in[0,T)$ for some $T\in (0,\infty)$. Then 		\[0<u(t,x)\leq C_{1}:=\max _{x \in\left[0, h_{0}\right]} u_{0} (x)+1  \quad\text{for } t\in[0,T) \text{ and } x\in(0,h(t)].\] 
		Moreover,	if $l:=\inf_{[0,T)}h(t)>0$, then there exists $C^*(T)>0$ depending on $T$ such that 
		\[u(t,x)\geq C^*(T)\quad\text{for } t\in[0,T) \text{ and } x\in[l/2,h(t)].\]
	\end{lemma}
		
	The proof of this result is based on some simple comparison arguments, and is given in Section 4.
	
	\begin{lemma}\label{th2.6}  	
		Suppose $(\bf{F})$ holds, $(u, h)$ solves \eqref{1.1} for $t\in [0,T)$ with some $T \in (0,\infty]$, and  $V(t)$ is defined by \eqref{V}.
		\begin{itemize}
			\item[{\rm (i)}] If $\delta_{0} \geq \delta$, then 
			\[
			h(t) \geq \sqrt{\frac{2V(0)}{C_{1}}} e^{-\frac{C_{0}t}{2} }\quad \text{for } t \in [0,T),
			\]
			where $C_{0} := \sup\limits_{u \in (0,\, C_1]} \frac{|f(u)|}{u}$, and $C_1$ is given in Lemma \ref{le2.5}.
			
			\item[{\rm (ii)}] If $0 \leq \delta_{0} < \delta$ and there exists $t_{0} \in [0,T)$ such that $V(t_{0}) < d(\delta - \delta_{0})/C_{f}$, then 
			\[
			V'(t) < 0  \text{ and }  V(t) < \frac{d(\delta - \delta_{0})}{C_{f}} \quad \text{for all } t \in [ t_{0}, T),
			\]
			where $C_{f}:= \sup_{u \in (0,1]} \frac{f(u)}{u}$.
		\end{itemize}		
	\end{lemma}
	\begin{proof}
		By Lemma \ref{le2.5}, we have 
		\[
		0 < u(t,x) \leq C_{1} \quad \text{for } t \in [0, T) \text{ and } x \in (0,h(t)].
		\]
		It then follows from the definition of $C_0$ that
		\begin{equation}\label{C0}
			|f(u(t,x))| \leq C_{0}u(t,x) \quad \text{for } t \in [0,T) \text{ and } x \in [0,h(t)].
		\end{equation}
		By the definition of $V(t)$, it is clear that $V(t)>0$ as long as $h(t)>0$.
		Direct computation yields
		\[\begin{aligned}
			V^{\prime}(t) & =h^{\prime}(t)h(t)u(t, h(t))+\int_{0}^{h(t)}x u_{t}(t, x) d x \\
			& =-dh(t)u_{x}(t,h(t)) +\int_{0}^{h(t)}x\left[d u_{x x}+f(u)\right] d x \\
			& =-dh(t)u_{x}(t,h(t))+d \int_{0}^{h(t)}x du_{x} +\int_{0}^{h(t)}xf(u)dx\\
			& =-du(t,h(t))+du(t,0)+\int_{0}^{h(t)}xf(u)dx\\
			& =d(\delta_{0}-\delta)+\int_{0}^{h(t)}xf(u)dx.
		\end{aligned}
		\]
		
		If $\delta_{0} \geq \delta$, then we obtain
		\[
		V^{\prime}(t) \geq \int_{0}^{h(t)} xf(u)  dx \geq -C_{0} \int_{0}^{h(t)} xu  dx = -C_{0} V(t),
		\]
		and hence
		\[
		V(t) \geq e^{-C_{0}t} V(0) > 0,
		\]
		which implies $h(t) > 0$ for $t \in [0,T)$. Furthermore, the inequalities
		\[
		e^{-C_{0}t} V(0) \leq V(t) = \int_{0}^{h(t)} xu  dx \leq \frac{C_{1}}{2} h^{2}(t)
		\]
		lead to
		\[
		h(t) \geq \sqrt{\frac{2V(0)}{C_{1}}} e^{-\frac{C_{0}t}{2}} \quad \text{for } t \in [0,T).
		\]
		
		If $0 \leq \delta_{0} < \delta$, then we have
		\begin{equation}\label{V'}
			V^{\prime}(t) = d(\delta_{0} - \delta) + \int_{0}^{h(t)} xf(u)  dx \leq d(\delta_{0} - \delta) + C_{f} V(t),
		\end{equation}
		where $C_{f}= \sup_{u \in (0,1]} \frac{f(u)}{u}= \sup_{u >0} \frac{f(u)}{u}$. Since $V(t_{0}) < d(\delta - \delta_{0})/C_{f}$ for some $t_{0} \in [0,T)$, we have $V'(t_{0}) < 0$. It then follows easily from \eqref{V'} that $V'(t) < 0$ for all $t \in [ t_{0}, T)$, and thus
		\[
		V(t) \leq V(t_{0}) < \frac{d(\delta - \delta_{0})}{C_{f}} \quad \text{for } t \in [ t_{0}, T).
		\]
		The proof is now complete.		
	\end{proof}

	The next lemma shows that $h'(t)$ is bounded  as long as it exists.
	Since the free boundary may move backward, the existing method to bound the free boundaries in \cite{DL} and similar works do not work anymore.  	\begin{lemma}\label{le2.7}
		Suppose that $(\bf{F})$ holds, $(u(t,x),h(t))$ is a solution to \eqref{1.1} defined on $t\in [0,T)$ for some $T\in (0,\infty]$, and $\inf_{t\in[0,T)}h(t)=:l>0$. Then there exists $C_{2}(T) >0$ dependent on $T$, such that
	\[	|h'(t)|\leq C_{2}(T)\ \mbox{ for }\ t\in[0, T).\]
	
	\end{lemma}
	
	\begin{proof}
We proceed in several steps.	
		
		{\bf Step 1.} {We  construct a lower solution to  bound  $h'(t)$ from below.}
		
		For some constants $0<c<1$, $k>0$ and $m>k/l$ to be specified later,
		define
		\[
		\begin{cases}
			\theta(s) := 	ce^{-s/c} - c & \text{for } s \geq 0, \\[2mm]
			\underline{u}(t,x) := \delta \theta((h(t)-x)m) + \delta &  \text{for } t \geq 0,\; x \in [h(t)-k/m,h(t)].
		\end{cases}
		\]
		We will show eventually that the following  holds:
		\begin{equation}\label{1}
			\begin{cases}
				\underline u_t< d\underline u_{xx}+f(\underline u),\ \ &t\in (0,T),\, x\in [h(t)-k/m,h(t)],\\
				\underline u(t,h(t)-k/m)\leq u(t,h(t)-k/m),\ \underline u(t, h(t))=\delta,\ & t\in [0,T),\\
				\underline u(0,x)\leq  u_0(x),\ &x\in [h_0-k/m,h_0].
			\end{cases}
		\end{equation}
		Assuming \eqref{1}, we obtain by the comparison principle that $u(t,x)\geq \underline u(t,x)$ for $t\in (0,T),\, x\in [h(t)-k/m,h(t)]$,
		and hence, due to $u(t,h(t))=\underline u(t, h(t))=\delta$, we deduce
				\begin{align}\label{2}
			h'(t)=-\frac{d}{\delta}  u_x(t, h(t))\geq -\frac{d}{\delta} \underline u_x(t, h(t))=-dm \quad\text{for } t\in [0,T).
		\end{align}
		
		Some of the inequalities in \eqref{1} can be shown to hold for the entire range of $t\in (0, T)$, some of them and \eqref{2} can only be shown  to hold for $t$ in some possibly smaller interval $(0, T_1)$ initially. We will show in Step 3 (after proving an upper bound for $h'(t)$ in Step 2) that all the desired inequalities hold with $T_1=T$.
		
		Define $c_0 := \min \left\{ C^*(T), \, \delta/2, \, 1/2 \right\} > 0$, where $C^*(T)$ is given in Lemma \ref{le2.5} satisfying
		\begin{align}\label{ne2.26}
				u(t,x)\geq C^*(T)\quad\text{for } t\in[0,T) \text{ and } x\in[l/2,h(t)].
		\end{align}
		  Select $c \in (0,1)$ sufficiently close to 1 such that
		\begin{equation*}
			k := -c \ln \left( \frac{c_0}{c\delta} + 1 - \frac{1}{c} \right) > 0.
		\end{equation*}
		Direct computation yields $\underline u(t, h(t) - k/m) = c_0$.
		We assume from now on that 
		\begin{align*}
			m>2k/l, \mbox{ and so } h(t)-k/m>l/2>0 \mbox{ for } t\in [0, T).
		\end{align*}
		By \eqref{ne2.26}, we have
		\begin{align*}
		u(t,x)\geq c_0=	\underline u(t,h(t)-k/m)\quad \text{ for } t\in [0, T), \ x\in [l/2, h(t)]\supset [h(t)-k/m,h(t)],
		\end{align*}
		and by the definition of $\underline u$, it satisfies $\underline u(t,\underline h(t))=\delta$; thus the desired conclusions in the second line of  \eqref{1} are verified. 
		
		Next we observe that for $x \in [h_0 - k/m, h_0]$, 
		\begin{equation*}
			\underline u_x(0, x) =-\delta m \theta'(h_0-x)\geq \delta m e^{-k/c} > \max_{x \in [0, h_0]} |u_0'(x)|
		\end{equation*}
		provided $m$ is sufficiently large. Furthermore, since $\underline u(0, h_0) = u_0(h_0) = \delta$, the above inequality  implies that for sufficiently large $m>0$,
		\begin{equation*}
			\underline u(0, x) < u_0(x) \quad \text{for }  x \in [h_0 - k/m, h_0),
		\end{equation*}
		which is the desired inequality on the third line of \eqref{1}.

		Finally we prove the  inequality in the first line of \eqref{1} for $t\in [0, T_{1})$, where $T_{1}\leq T$ is  defined by
		\begin{align*}
			T_1:=\sup\{	t\in (0, T]: h'(s)>-\frac{d}{\delta} \underline u_x(s, h(s))=-dm\quad \text{ for all } s\in [0,t)\}.
		\end{align*}
		Note that for $m\gg 1$,
		\begin{align*}
			-\frac{d}{\delta} \underline u_x(0, h_{0})=-dm<h'(0),
		\end{align*}
		which ensures that $T_{1}$ is well defined. (In Step 3, we will show that actually $T_{1}=T$.)

		A direct computation yields
		\begin{align}\label{u_xx}
			\underline u_t=-\delta m e^{-(h(t)-x)m/c} h'(t),\ \ \underline u_x=\delta m e^{-(h(t)-x)m/c},\ \  \underline u_{xx}=\frac{\delta m^2}{c} e^{-(h(t)-x)m/c}.
		\end{align}
		Then by the definition of $T_1$ we obtain
		\begin{align*}
			\underline u_t\leq d\delta m^2 e^{-(h(t)-x)m/c} \mbox{ for } t\in (0, T_1), \ x\in [h(t)-k/m,h(t)].
		\end{align*}
		Thus the first inequality in  \eqref{1} with $t\in (0, T_1)$ will follow if $K:=\max_{s\in [0, \delta]}|f(u)|$ and
		\[ d\delta m^2 e^{-(h(t)-x)m/c} <d \underline u_{xx}-K\leq d \underline u_{xx}  +f(\underline u) \mbox{ for } t\in (0, T_1), \ x\in [h(t)-k/m,h(t)],\]
		 which is satisfied, due to \eqref{u_xx},  when
		\begin{align}\label{m}
			m>\sqrt{\frac{Kce^{k/c}}{d\delta(1-c)}}.
		\end{align}

		{\bf Step 2.} We prove that
		\begin{align}\label{2.26}
			-dm \leq h'(t) \leq \frac{2M(2C_{1}-\delta)d}{\delta} \quad \text{for }  t \in [0,T_{1}).
		\end{align}

		By the definition of $T_{1}$, the first inequality in \eqref{2.26} holds automatically. It remains to verify the second. Recall from Lemma \ref{le2.5} that
		\[
		0 < u(t, x) \leq C_{1} \quad \text{for }  t \in [0, T) \text{ and } 0 < x \leq h(t).
		\]
		Since $f \in C^{1}$ with $f(0) = 0$, there exists a constant $L > 0$ depending on $C_{1}$ such that $f(u) \leq Lu$ for $u \in [0, 2C_{1}]$. Define
		\[\begin{cases}
			\Omega := \left\{ (t, x) : 0 \leq t < T_{1},\  h(t) - M^{-1} < x < h(t) \right\}
			\\
			\bar{u}(t, x) := \left(2C_{1} - \delta\right) \left[ 2M(h(t)-x) - M^{2}(h(t)-x)^{2} \right] + \delta,
		\end{cases}
		\]
		with $M > 1/l$ to be specified later. 
		
		Note that $\Omega \subset \{(t, x) : 0 \leq t < T_{1}, 0 < x < h(t)\}$, and for $  (t, x) \in \Omega,$
				\begin{align*}
			&\bar{u}(t, x) = \left(2C_{1} - \delta\right) \left( 1 - \left[ 1 - M(h(t)-x) \right]^{2} \right) + \delta \leq 2C_{1}  \\
			&\bar{u}(t, h(t)) = \delta \quad\text{and}\quad
			\bar{u}\left(t, h(t) - M^{-1}\right) = 2C_{1} > u\left(t, h(t) - M^{-1}\right).
		\end{align*}
		Moreover,  using $h'(t) \geq -m d$ on $[0,T_{1})$, we obtain, for $(t, x) \in \Omega$, 
		\begin{align*}
			\bar{u}_{t} - d \bar{u}_{xx} - f(\bar{u})
			&= \left(2C_{1} - \delta\right) h'(t) \left[ 2M - 2M^{2}(h(t)-x) \right] + 2d M^{2} \left(2C_{1} - \delta\right) - f(\bar{u}) \\
			&\geq -2d M m (2C_{1}-\delta) [1 - M(h(t)-x)] + 2d M^{2} (2C_{1}-\delta) - 2L C_{1} \\
			&\geq -2d M m (2C_{1}-\delta) + 2d M^{2} (2C_{1}-\delta) - 2L C_{1} \\
			&= 2d M (M - m) (2C_{1} - \delta) - 2L C_{1} \\
			&\geq 0
		\end{align*}
		provided that
		\[
		M \geq \sqrt{ \frac{L C_{1}}{d (2C_{1}-\delta)} } + m.
		\]
		
		Next, we show that 
		\[
		\bar{u}(0, x) \geq u_{0}(x) \quad \text{for }  x \in \left[h_{0} - M^{-1}, h_{0}\right]
		\]
		provided $M$ is sufficiently large. Select $M$ large enough to satisfy  
		\[
		M(2C_{1}-\delta) > \max_{x \in [0, h_{0}]} |u_{0}'(x)|;
		\]  
		then
		\[
		\bar{u}_{x}(0, x) \leq -M(2C_{1}-\delta) < u_{0}'(x) \quad \text{for }  x \in \left[h_{0} - (2M)^{-1}, h_{0}\right].
		\]  
		It then follows from  $\bar{u}(0, h_{0}) = u_{0}(h_{0}) = \delta$  that for such $M>0$,  
		\[
		\bar{u}(0, x) \geq u_{0}(x) \quad \text{for }  x \in \left[h_{0} - (2M)^{-1}, h_{0}\right].
		\]  
		Since $\bar{u}_{x}(0, x) \leq 0$ on $[h_{0} - M^{-1}, h_{0}]$, direct computation gives  
		\[
		\bar{u}(0, x) \geq \bar{u}\left(0, h_{0} - (2M)^{-1}\right) > \frac{2C_{1} - \delta}{2} + \delta > C_{1} \geq u_{0}(x) \quad \text{for }  x \in \left[h_{0} - M^{-1}, h_{0} - (2M)^{-1}\right].
		\]  
		Combining these inequalities, we conclude  
		\[
		\bar{u}(0, x) \geq u_{0}(x) \quad \text{for }  x \in \left[h_{0} - M^{-1}, h_{0}\right].
		\]  
		We may now apply the comparison principle over 
		$
		\left\{ (t, x) : 0 \leq t < T_{1}, \, h(t) - M^{-1} < x < h(t) \right\}
		$
		to obtain $u(t, x) \leq \bar{u}(t, x)$ in this region. Since $u(t, h(t)) = \bar{u}(t, h(t)) = \delta$, it follows that  
		\[
		u_{x}(t, h(t)) \geq \bar{u}_{x}(t, h(t)) = -2M(2C_{1} - \delta) \quad \text{for }  t \in [0, T_{1}).
		\]  
		Consequently,  
		\[
		-dm \leq h'(t) = -\frac{d}{\delta} u_{x}(t, h(t)) \leq \frac{2M(2C_{1} - \delta)d}{\delta} \quad \text{for } t \in (0, T_{1}),
		\]  
		where $M$ is chosen sufficiently large, independent of $T_{1}$.
		
		{\bf Step 3. } For sufficiently large fixed $m$, we show  that  \eqref{2}  holds for $t\in [0,T)$, and consequently, \eqref{2.26} is valid for all $t\in[0,T)$.

		By the definition of $T_1$, clearly either $T_1 = T$ or $T_1 < T$, and in the latter case, we necessarily have
		\begin{align}\label{2.27}
			h'(T_1) = -\frac{d}{\delta} \underline{u}_x(T_1, h(T_1)) = -dm.
		\end{align}
		Define $w := u - \underline{u}$. Then $w$ satisfies 
		\[
		\begin{cases}
			w_t > d w_{xx} + f(u) - f(\underline{u}) \geq d w_{xx} + c w, & t \in (0, T_1], \ x \in \left[h(t) - k/m, h(t)\right], \\
			w(t, h(t) - k/m) > 0, \quad w(t, h(t)) = 0, & t \in [0, T_1], \\
			w(0, x) \geq  0, & x \in \left[h_0 - k/m, h_0\right],
		\end{cases}
		\]
		for some constant $c\in\mathbb R$. The comparison principle yields
		\[
		w(t, x) > 0 \quad \text{for } t \in (0, T_1], \ x \in [h(t) - k/m, h(t)).
		\]
		Moreover, we can use the Hopf lemma to obtain $w_x(T_1, h(T_1)) < 0$, which implies $u_x(T_1, h(T_1)) < \underline{u}_x(T_1, h(T_1))$. Consequently,
		\[
		h'(T_1) = -\frac{d}{\delta} u_x(T_1, h(T_1)) > -\frac{d}{\delta} \underline{u}_x(T_1, h(T_1)) = -dm,
		\]
		contradicting \eqref{2.27}. Thus the desired inequality holds, and we have proved what are wanted.
	\end{proof}

	\subsection{Maximal existence}
	\begin{proof}[Proof of Theorem \ref{th1.1}.] Lemma \ref{le2.7} implies that as long as $h(T)>0$, the local solution obtained in
		Theorem \ref{th2.1} can be extended to some bigger time interval containing $[0, T]$. Let  $\left[0, \,T_{\mathrm{max}}\right) $ be the maximal interval of existence obtained via this extension, where $T_{\mathrm{max}}=\infty$ or $T_{\mathrm{max}}<\infty$. Then, by Step 4 in the proof of Theorem \ref{th2.1} we know that
		\[h \in C^{1+\frac{1+\alpha}{2}}\left(\left(0, T_{\mathrm{max}}\right)\right), \, u \in C^{1+\frac{\alpha}{2}, 2+\alpha}\left(\Omega_{T_{\mathrm{max}}}\right),\]
		where
		\[\Omega_{T_{\mathrm{max}}}:=\left\{(t, x): t \in\left(0,T_{\mathrm{max}}\right), \,x \in[0, h(t)]\right\}.\]
		This completes the proof.
	\end{proof}

	\begin{lemma}\label{lemma2.8} 	
		Suppose that $(\bf{F})$ holds, $\delta>0$ and $\delta_0 \geq 0$. Let $(u,h)$ be a solution to \eqref{1.1} defined for $t\in [0,T_{\mathrm{max}})$. If $\inf_{t\in (0,T_{\max})} h(t) > 0$, then necessarily $T_{\mathrm{max}} = \infty$.
	\end{lemma}
	\begin{proof}
		Arguing by contradiction, we assume \( T_{\mathrm{max}} < \infty \). 
		By Lemmas \ref{le2.5} and \ref{le2.7}, there exist constants \( C_{1}, C_{2} > 0 \) such that 
		\begin{equation}\label{Bound1}
			\begin{cases}
				0 \leq u(x, t) \leq C_{1} & \text{for } t \in [0, T_{\mathrm{max}}) \text{ and } x \in [0, h(t)], \\
				\left|h^{\prime}(t)\right| \leq C_{2} \quad \text{and} \quad |h(t)| \leq C_{2} t + h_{0} & \text{for } t \in [0, T_{\mathrm{max}}).
			\end{cases}
		\end{equation}	
		
		For any small constant  $\varepsilon>0$, by Lemmas   \ref{le2.5}, \ref{th2.6} and \ref{le2.7}, we see from the proof of Theorem \ref{th2.1}  that $ u \in C^{\frac{1+\alpha}{2}, 1+\alpha}\left(\bar{\Omega}_{T_{\mathrm{max}}-\varepsilon}\right)$. Therefore, as in Step 4 of the proof of Theorem \ref{th2.1}, we can use Schauder's estimates to obtain that for any fixed  $0<T<T_{\mathrm{max}}-\varepsilon$, 
		\[\Arrowvert u\Arrowvert_{C^{1+\frac{\alpha}{2}, 2+\alpha}\left(\Omega_{T_{\mathrm{max}}-\varepsilon} \backslash \Omega_{T}\right)} \leq M,\]
		where  $M$  depends on  $T$, $T_{\mathrm{max}}$,    $C_1,\ C_2$, but is independent of $ \varepsilon$. Since $ \varepsilon>0 $ can be arbitrarily small, it follows that for any $ t \in\left[T,\, T_{\mathrm{max}}\right)$,
		\begin{equation}\label{B2}
			\Arrowvert u(t, \cdot)\Arrowvert_{C^{2+\alpha}([0, h(t)])} \leq M.
		\end{equation}
		Now, repeating the proof of Theorem \ref{th2.1} with any $t_0\in (T, T_{\rm max})$ as the initial time, we conclude that there exists  $\varepsilon_{0}>0 $, depending only on $M$, $ C_1,\ C_2$  such that \eqref{1.1} with initial time  $t_0$ has a unique solution for $t\in [t_0, t_0+\varepsilon_0]$. Taking $t_0=T_{\mathrm{max}}-\frac{\varepsilon_{0}}{2}$  we see that  $(u,  h)$ is defined for $ t $ at least up to  $T_{\mathrm{max}}-\frac{\varepsilon_{0}}{2}+\varepsilon_{0}$, which is a contradiction to the definition of  $T_{\mathrm{max}}$. Thus, we conclude that $T_{\mathrm{max}}=\infty$.
	\end{proof}
	
	\begin{remark}\label{remark2.9}
		If $\delta_{0}\geq \delta$, then by Lemma \ref{th2.6} {\rm (i)} we have 
		\[
		h(t) \geq \sqrt{\tfrac{2V(0)}{C_{1}}}\, e^{-\tfrac{C_{0}T_{\mathrm{max}}}{2}} \quad \text{for } t \in [0, T_{\mathrm{max}}).
		\]  
		This leads to $T_{\mathrm{max}} = \infty$ by Lemma \ref{lemma2.8}.
	\end{remark}
	
	\section{Long time dynamics}
	
	Throughout this section, we assume $\delta\in (\hat\theta_f,1)$, and $f$ satisfies ${\bf (f_m)}$, ${\bf (f_b)}$ or ${\bf (f_c)}$, which we abbreviate to ${\bf (F)}$ as before. We will first classify the behaviour of $h(t)$, and then determine the behaviour of $u(t,x)$ accordingly.

	\subsection{Behaviour of $h(t)$.} 	
		
	We start by proving a simple lemma which will be used several times later.
	Since $(\bf{F})$ holds and $\delta\in(\hat\theta_f,1)$, we can choose 
	$\tilde\theta\in (\hat\theta_f, \delta)$ and
	$\tilde{f}\in C^1$ satisfying
		\[
  \begin{cases}\tilde{f}(u)\leq f(u) \ \mbox{ for }\ u\geq0,\ \tilde{f}(0)=\tilde{f}(\tilde \theta)=\tilde{f}(\delta)=0,\ \tilde{f}'(0)<0,\ \tilde{f}'(\delta)<0, \  \\
		\tilde{f}(u)<0 \mbox{ for } u\in (0,\tilde\theta)\cup (\delta,\infty),\ \tilde{f}(u)>0 \mbox{ for } u\in (\tilde\theta, \delta),\
		\int_0^\delta\tilde{f}(u)du>0.
	\end{cases}
	\]
		By \cite{DL},  there exists $\tilde{c}_0>0$ such that for each $c\in(0,\tilde{c}_0)$, there exists a unique solution $q=\tilde{q}_c$ satisfying
\begin{equation}\label{nn3.1}
	\begin{cases}
		dq^{\prime \prime} -cq^{\prime}+ \tilde{f}(q) = 0,\ q'>0 \mbox{ in }  [0, \infty), \\
		q(0) = 0,\  q(\infty)=\delta.
	\end{cases}
\end{equation}
Fix $c\in(0,\tilde{c}_0)$, $\theta^{**}\in(\hat\theta_f,\delta)$, and denote
\[m^* = \begin{cases}
	\delta_0 & \mbox{ if } \delta_0 < \delta,\\
	\theta^{**} & \mbox{ if } \delta_0 \geq \delta.
\end{cases}\]
Then, there exists $x_0\geq0$ such that $\tilde{q}_c(x_0)=m^*$.
	\begin{lemma}\label{nlower}
			Suppose $\delta \in (\hat\theta_f,1)$, $\delta_0 \geq 0$ and $(\bf{F})$ holds.  If $(u, h)$ is a solution of \eqref{1.1} with $T_{max}=\infty$, and $\tilde q_c$ and $x_0$ are given as above, then there exists a constant $t_0 \geq 0$,  such that 
			\[u(t,x)\geq \tilde{q}_c(x+x_0)\ \mbox{ for }\ t\geq t_0 \mbox{ and } x\in[0, h(t)].\]
	\end{lemma}

	\begin{proof}
		Regarding $L$ as a parameter, we define a family of lower solutions by
			\[
		w_L(t,x): = 
		\begin{cases} 
			0 & \text{if } x < -c t + L, \\[4pt]
			\tilde{q}_c(x + c t - L) & \text{if } -c t + L \leq x <\infty.
		\end{cases}
		\]
		Take $L\geq h_0$ and denote $t_L:=(L+x_0)/c$.
		Then 
		\[
		w_L(t_L, 0)=\tilde q_c(x_0)=m^*\leq\delta_0.
		\]
	Due to \eqref{nn3.1}, the function \( w_{L} \) satisfies (in the weak sense)
\begin{equation}\label{sxj}
	\left\{
	\begin{array}{ll}
		\partial_t w_L - d\, \partial_{xx} w_L - f(w_L) \leq 0 = u_t - d u_{xx} - f(u), & 0 < t \leq t_L,\ 0 < x < h(t), \\
		w_L(t, 0) \leq \delta_0 = u(t, 0), & 0 < t \leq t_L, \\
		w_L(t, h(t)) < \delta = u(t, h(t)), & 0 < t < t_L, \\
	w_L(0, x)=0 \leq u(0, x), & x \in [0, h(0)].
	\end{array}
	\right.
\end{equation}
 By the comparison principle we have
\[
w_L(t,x) \leq u(t,x) \quad \text{for } t \in [0, t_L],\ x \in [0, h(t)].
\]
In particular, 
\[ u(t_L,x)\geq w_L(t_L,x)=\tilde{q}_c(x + c t_L - L)=\tilde{q}_c(x + x_0) \quad \text{for }  x \in [0, h(t_L)].\]
For any given $t\geq t_0:=\frac{h_0+x_0}{c}$ we can find $L\geq h_0$ such that $t_L=t$, and hence we have
\[u(t,x)\geq \tilde{q}_c(x + x_0) \quad \text{for } t\geq t_0, \ x \in [0, h(t)]. \]
	This completes the proof of the lemma.	
		\end{proof}
	
\begin{corollary}\label{c3.2}
	Under the assumptions of Lemma \ref{nlower}, suppose $l: = \inf_{[0,\infty)} h(t) > 0$. Then there exists a finite constant $C_2$ such that
	\[
	|h'(t)| \le C_2 \qquad \text{for all } t \ge 0.
	\]
	Moreover,
	\[\begin{cases}
	\delta_0 < \delta & \implies u(t,x) > \delta_0,\  u_x(t,0) > 0 \mbox{ for $t \ge t_0$ and $x \in (0, h(t)]$},\\
	\delta_0 \ge \delta & \implies 	u(t,x) \ge \theta^{**} > \hat\theta_f \mbox{ for $t \ge t_0$ and $x \in (0, h(t)]$}.
	\end{cases}
	\]
\end{corollary}
	
	\begin{proof} 
	Applying Lemma~\ref{le2.7} with $T=\infty$, we see that there exists a constant $C_2$ such that $|h'(t)|\leq C_2$ for all $t\geq 0$.
	
		By Lemma~\ref{nlower}, we have $u(t,x)\geq \tilde{q}_c(x+x_0)\geq \tilde{q}_c(l/2+x_0)>0$ for $t\geq t_0$ and $x\in[l/2,h(t)]$.
		Define 
		\[
		c_0:=\min \bigl\{ \tilde{q}_c(l/2+x_0),\; \min_{0\leq t\leq t_0,\ l/2\leq x\leq h(t)} u(t,x),\; \delta/2,\; 1/2 \bigr\}.
		\]
		Since $u(t,x)>0$ for $t\in[0,t_0]$ and $x\in[l/2,h(t)]$, we see that 
		 $u(t,x)\geq c_0>0$ for such $(t,x)$.	
		
		If $\delta_0<\delta$,  then Lemma~\ref{nlower} implies  $u(t,x)\geq \tilde{q}_c(x+x_0)>\tilde q_c(x_0)=\delta_0$ for $t\geq t_0$ and $x\in(0,h(t)]$. Hence,
		due to $u(t,0)=\tilde q_c(x_0)=m^*=\delta_0$  for $t\geq t_0$, we obtain $u_x(t,0)\geq \tilde{q}_c'(x_0)>0$ for $t\geq t_0$.
		
		If $\delta_0\geq \delta$, Lemma~\ref{nlower} implies that $u(t,x)\geq q(x+x_0)\geq q(x_0)=\theta^{**}$ for $t\geq t_0$ and $x\in[0,h(t)]$.
	\end{proof}

	Since the case $\delta_0 < \delta$ leads to more complex dynamics of \eqref{1.1}, we  treat the case $\delta_0 \geq \delta$ first, which is done in the next two lemmas below. The subsequent lemmas address the case $0\leq \delta_0 < \delta$.

	\begin{lemma}\label{le2.8}
		Suppose $\delta \in (\hat\theta_f,1)$, $\delta_0 \geq \delta$ and $(\bf{F})$ holds.  Let $(u, h)$ be the solution of \eqref{1.1}. Then $T_{max}=\infty$ and there exists a constant $T_0 \geq 0$ such that 
		\[u(t, x) \geq \delta\quad \text{for all } x \in [0, h(t)] \text{ and } t \in [T_{0}, \infty).\] 
		Consequently, 
		\[h'(t) \geq 0 \quad \text{for } t \in [T_{0}, \infty).\]  
	\end{lemma}
	\begin{proof}
		By Remark ~\ref{remark2.9}, we have $T_{\max} = \infty$. Let $t_0\geq 0$ be given by Corollary \ref{c3.2} and 
		\[m_0 := \min_{x \in [0, h(t_0)]} u(t_0,x)\in [\theta^{**}, \delta].\]
		Then consider the auxiliary ODE problem
		\[
		v' = f(v)\quad \text{for } t>0, \quad v(0) = m_0.
		\]
		It follows from condition $(\bf{F})$ that $v'(t) > 0$ for all $t > 0$ and $v(t) \to 1$ as $t \to \infty$.  Hence there exists a unique $t_1 \geq 0$ satisfying 
		\[
		v(t_1) = \delta \quad \text{and} \quad m_0 \leq v(t) \leq \delta \quad \text{for} \quad t \in [0, t_1].
		\]
		This allows us to apply the standard comparison principle over the region 
		\[
		\mathcal{D}_1 := \{(t, x) : 0 \leq t \leq t_1,\  0 \leq x \leq h(t+t_0)\}
		\]
		to obtain $u(t+t_0, x) \geq v(t)$ in $\mathcal{D}_1$. In particular, 
		\[
		u(t_1+t_0, x) \geq v(t_1)=\delta \quad \text{for} \quad x \in [0, h(t_1+t_0)].
		\]
		Now consider the region
		\[
		\mathcal{D}_2 := \{(t, x) : t \geq T_0:=t_0+t_1,\  0 \leq x \leq h(t)\}.
		\]
		Comparing $u(t, x)$ with the constant subsolution $\underline{u} \equiv \delta$ gives $u(t, x) \geq \delta$  in $\mathcal{D}_2$, which implies $u_x(t, h(t))\leq 0$. Hence 
		\[h'(t)=-\frac{d}{\delta}u_x(t, h(t))\geq 0\quad\text{for } t\geq T_0.\]
		This completes the proof.	
	\end{proof}

	\begin{lemma}\label{le3.1}
		Under the conditions of Lemma \ref{le2.8}, we have
		\( h_{\infty}:=\lim _{t \rightarrow \infty} h(t)=\infty.\)
	\end{lemma}
	\begin{proof}
		By Lemma~\ref{le2.8},   the limit
		\[
		h_{\infty} := \lim_{t \to \infty} h(t) \ \text{exists and}\ h_{\infty} \in [ h(T_{0}),\infty].
		\]

		Suppose now that $  h_{\infty}$  is finite. Define the transformation
		\[y=\frac{x}{h(t)}, \quad U(t, y)=u(t, x) .\]
		Then $U$ satisfies 
		\begin{equation}\label{3.1}
			\left\{\begin{array}{ll}
				U_{t} -  \frac{d}{h^2(t)} U_{yy} - \frac{h'(t)}{h(t)}y U_{y} = f(U), & t > 0,\  y \in [0, 1], \\
				U(t, 0) = \delta_{0},\quad U(t, 1) = \delta, & t > 0, \\
				h'(t) = -\frac{d}{\delta} \frac{1}{h(t)} U_{y}(t, 1), & t > 0, \\
				U(0, y) = u_{0}(h_{0} y), & y \in [0, 1].
			\end{array}\right.
		\end{equation}
By Lemma \ref{le2.5} and Corollary \ref{c3.2}, we have $\| U(t,\cdot)\|_\infty\leq C_1$ and $|h'(t)|\leq C_2$ for $t> 0$. Applying standard $L^{p}$ theory to \eqref{3.1} (away from the initial time), we conclude that for any $p > 1$ and every integer $n\geq 1$,
		\[
		\| U \|_{W_{p}^{1,2}([n, n+2] \times [0, 1])} \leq C_{p},
		\]
		where $C_{p} > 0$ is independent of $n$. Selecting $p$ sufficiently large and invoking the Sobolev embedding theorem, we conclude from the third equation in \eqref{3.1} that $h'(t)$ is uniformly continuous for $t \geq 1$. This uniform continuity and $h_\infty<\infty$ imply that $ h^{\prime}(t) \rightarrow 0$   as $ t \rightarrow \infty$.
		Let $ \left\{t_{n}\right\}$  be an arbitrary sequence increasing to  $\infty $ as $ n \rightarrow \infty$, and define
		\[U_{n}(t, y):=U\left(t_{n}+t, y\right).\]
		Then, it is easy to verify that
		\begin{equation}\label{Un}
			\left\{
			\begin{array}{ll}
				(U_{n})_{t} - \frac{d}{h^2(t+t_n)}  (U_{n})_{yy} - \frac{h'(t+t_n)}{h(t+t_n)} y(U_{n})_{y} = f(U_{n}), & \! t > -t_n, \, y \in [0, 1], \\
				U_{n}(t, 0) = \delta_{0},\quad U_{n}(t, 1) = \delta, & \! t > -t_n, \\
				h'(t+t_n) = - \frac{d}{\delta h(t+t_n)} (U_{n})_{y}(t, 1), & \! t > -t_n.
			\end{array}
			\right.
		\end{equation}
		Applying $ L^{p} $ estimates to \eqref{Un}  and using the Sobolev embedding theorem, we find that, subject to a subsequence, $ U_{n} \rightarrow \tilde{U}$  in  $C_{l o c}^{(1+\alpha) / 2,1+\alpha}\left(\mathbb{R} \times\left[0, 1\right]\right)$  for some $ \alpha \in(0,1)$. The limit function  $\tilde{U} $ satisfies (in the weak sense and then in the classical sense)
		\begin{equation}\label{02.31}
			\left\{\begin{array}{ll}
				\tilde{U}_{t} -  \frac{d}{h_\infty^2}\tilde{U}_{yy} = f(\tilde{U}), & t \in \mathbb{R}, \ y \in [0, 1], \\
				\tilde{U}(t, 0) = \delta_{0},\quad \tilde{U}(t, 1) = \delta, & t \in \mathbb{R}, \\
				\tilde{U}_{y}(t, 1) = 0, & t \in \mathbb{R}.
			\end{array}\right.
		\end{equation}
		From	 Lemma \ref{le2.8}, we know that  $\tilde{U}(t, y) \geq \delta$. Since $ \underline{U} \equiv \delta $ is a strict lower solution to \eqref{02.31}, the strong maximum principle implies that  $\tilde{U}(t, y)>\delta $ for  $t \in \mathbb{R}$  and $ y \in\left(0, 1\right)$. The Hopf boundary lemma then infers $\tilde{U}_{y}\left(t, 1\right)<0 $, which contradicts the third equation in \eqref{02.31}. Therefore we must have $h_\infty=\infty$. This completes the proof of the lemma.
	\end{proof}

	We now consider the case  $0 \leq \delta_0 < \delta$, where $T_{\text{max}}$ might  be finite, and it is not clear that $h(t)$ is monotonic after some finite time. Recall that we always assume  $f$ satisfies ${\bf (F)}$. 
	\begin{lemma}\label{lemma2.10}
		If  $0\leq\delta_{0}<\delta$, then the following conclusions are equivalent:
		\begin{itemize}
			\item[\rm (i)]\ \ 
			$T_{\text{max}}<\yy$,
			\item[\rm (ii)]\ \  $ \lim_{t\to T_{\text{max}}^{-}} h(t)=0$,
			\item[\rm (iii)]\ \  $ \liminf_{t\to T_{\text{max}}^{-}} h(t)=0$,
			\item[\rm (iv)] \ \ $
			\displaystyle \int_{0}^{h(t_{0})}\!\!\!x u(t_{0}, x) d x<d(\delta-\delta_{0})/C_{f}\ \ {\rm for\ some}\ t_{0}\in[0,T_{\text{max}}),$
			where $C_{f}= \sup_{u \in (0,1]} \frac{f(u)}{u}$.
		\end{itemize}
	\end{lemma}
	\begin{proof}
		{\bf Step 1.} We prove that $T_{\text{max}}<\yy\implies \lim_{t\to T_{\text{max}}^{-}} h(t)=0$.

		Recall  from Lemma \ref{lemma2.8} that  
		\begin{align*}
			\inf_{t \in [0,T_{\text{max}})} h(t) > 0 \;\Longrightarrow\; T_{\text{max}} = \infty.
		\end{align*}
		Therefore,  $T_{\text{max}} < \infty$ implies  $\inf_{t \in [0,T_{\text{max}})} h(t) = 0$, and hence $\liminf_{t\to T_{\text{max}}^{-}} h(t)=0$.
		We now further prove $\lim_{t \to T_{\text{max}}^{-}} h(t) = 0$. Suppose by way of contradiction that $\limsup_{t \to T_{\max}^{-}} h(t) = \bar{h} \in (0, \infty]$. 
		Then there exist $\hat{h} \in (0, \bar{h})$ and a sequence $\{t_n\}$ satisfying $t_n \to T_{\max}$ such that $h(t_n) \geq \hat{h}/2$ for all sufficiently large $n$, say for all $n \geq N> 0$.

		Since $f$ satisfies  $(\bf{F})$, it is possible to choose  $\tilde{f}$ such that
		\[
	 \tilde{f} \in C^{1}, \quad \tilde{f}<0 \ \text{in}\ (0, \infty), \quad \tilde{f}(0)  = 0, \quad
			 \tilde{f} \leq f \ \text{in}\ [0,\infty).
		\]
		Consider the solution $\tilde w(t,x)$ of
		\begin{equation}\label{new3.5}
			\begin{cases}
				\tilde{w}_{t} = d \tilde{w}_{xx} + \tilde{f}(\tilde{w}), & t > 0,\ 0 < x < L, \\
				\tilde{w}(t, 0) = \tilde{w}(t, L) = 0, & t > 0, \\
				\tilde{w}(0, x) = \tilde{w}_0(x)>0, & 0 < x < L,
			\end{cases}
		\end{equation}
		where $\tilde{w}_0\in C^2([0,L])$ satisfies \[\tilde{w}_0(0)=\tilde{w}_0(L)=0,\, \max_{x\in[0, L]}\tilde{w}_0(x)<\delta \quad\text{and}\quad 0<\tilde{w}_0(x)<u_0(x) \quad\text{for}\quad x\in(0, h_0).\] 
		Fix $L>\max\{\hat{h}, h_0\}$ and define $\eta(t) = \min \{h(t), L\}$.  
	Since $\bar{w}\equiv \delta$ is an upper solution of \eqref{new3.5}, we easily obtain $\tilde{w}(t,x)\leq \delta$ for $t\geq0 $ and $x\in[0, L]$.
	Hence, due to $\tilde w(t, L)=0$, we always have $u(t,\eta(t))\geq \tilde{w}(t,\eta(t))$ and $u(t,0)=\delta_0\geq \tilde{w}(t,0)$ for $t\in[0, T_{\text{max}})$.
		Using the comparison principle to compare $\tilde{w}$ and $u$ over $\{(t,x): t\in[0,T_{\text{max}}),\ x\in [0, \eta(t)]\}$, we obtain
		$
		u(t, x) \geq \tilde{w}(t, x)$ in this region.
				Since
				\[
				m_T := \min\limits_{(t,x) \in [0, T_{\max}] \times [\hat{h}/4, \hat{h}/2]} \tilde{w}(t,x) > 0,
				\]
				it follows that
		\[
		u(t_n, x) \geq m_T \quad \text{for} \quad n \geq N \text{ and } x \in [\hat{h}/4, \hat{h}/2]\subset [0, h(t_n)].
		\]
		
		As $\liminf_{t \to T_{\max}^{-}} h(t) = 0$, there exists a sequence $\{s_n\}$ satisfying $s_n < t_n$ and $s_n \to T_{\max}$ such that $h(s_n) \to 0$ as $n \to \infty$. By Lemma \ref{le2.5}, we have
		\[
		0 \leq u(t,x) \leq C_1 \quad \text{for} \quad x \in [0, h(t)] \ \text{and} \ t \in [0, T_{\max}).
		\] 
		It follows that
		\[\begin{aligned}
			&V(t_{n})=\int_{0}^{h(t_{n})}xu(t_{n},x)dx\geq\int_{\hat{h}/4}^{\hat{h}/2}xm_Tdx\geq\frac{3m_T\hat{h}^{2}}{32} \text{ for } n\geq N,\\
			&V(s_{n})=\int_{0}^{h(s_{n})}xu(s_{n},x)dx\leq\int_{0}^{h(s_{n})}xC_{1}dx=\frac{C_{1}}{2}h^{2}(s_{n})\to 0 \,\,\text{as} \,\,n\to \infty.
		\end{aligned}\]
		Choose $n_0 > N$ such that $V(s_{n_0}) < d(\delta - \delta_0)/C_f$. By Lemma \ref{th2.6}, we then have $V'(t) < 0$ for all $t \in ( s_{n_0}, T_{\rm max})$. Hence  $0<V(t_n) < V(s_n)$ for $n > n_0$, which implies $V(t_n)\to 0$ as $n\to\infty$,  a contradiction.	Therefore, $\lim_{t \to T_{\text{max}}^{-}} h(t) = 0$, as desired.

		{\bf Step 2.} {If $\liminf_{t\to T_{\text{max}}^{-}} h(t)=0$, we prove $	T_{\text{max}}<\yy$.}

		We suppose that $T_{\text{max}}=\infty$. Then there exist $t_{n}\to \infty$ such that  $h(t_{n})\to 0$ as $n\to \infty$. By the boundedness of $u(t,x)$, we can choose $n$ big enough such that 
		\[
		V(t_n) = \int_0^{h(t_n)} x u(t_n, x)  dx <\frac 12 \frac{d(\delta - \delta_0)}{C_f}.
		\] 
		From the proof of Lemma \ref{th2.6} we know
		\[
		V'(t)\leq -d(\delta-\delta_{0})+C_{f}V(t) \mbox{ for } t>0.
		\]
		It follows that
		\[
		V(t) - \frac{d(\delta - \delta_0)}{C_f}\leq e^{C_f (t - t_n)} \left[ V(t_n) - \frac{d(\delta - \delta_0)}{C_f} \right]  \to -\infty \quad \text{as} \quad t \to \infty,
		\]
		which contradicts  the fact that $V(t)\geq 0$ for all $t\geq 0$. Hence we must have  $T_{\text{max}}<\infty$.
		\medskip
		
		{\bf Step 3.} {If $T_{\text{max}}<\yy$, we prove $\displaystyle\int_{0}^{h(t_{0})}x u(t_{0}, x) d x<d(\delta-\delta_{0})/C_{f}\ \ {\rm for\ some}\ t_{0}\in[0,T_{\text{max}})$.}\medskip
		
		By Step 1, $T_{\mathrm{max}} < \infty$ implies 
		$\lim_{t \to T_{\mathrm{max}}^{-}} h(t) = 0$. Direct computation yields
		\[
		V(t) = \int_{0}^{h(t)} \! x u(t,x)  \mathrm{d}x \leq \int_{0}^{h(t)} \! x C_{1}  \mathrm{d}x 
		= \frac{C_{1}}{2} h^{2}(t) \to 0 \quad \text{as} \quad t \to T_{\mathrm{max}},
		\]
		where $C_{1}$ is given by Lemma \ref{le2.5}. Consequently, 
		there exists $t_{0} \in (0, T_{\mathrm{max}})$ satisfying
		\[
		\int_{0}^{h(t_{0})} \! x u(t_{0}, x)  \mathrm{d}x < \frac{d(\delta - \delta_{0})}{C_{f}}.
		\]
		
		{\bf Step 4.} {If $\displaystyle \int_{0}^{h(t_{0})}x u(t_{0}, x) d x<d(\delta-\delta_{0})/C_{f}\ \ {\rm for\ some}\ t_{0}\in[0,T_{\text{max}})$, we prove $T_{\text{max}}<\yy$.}
		
		Assume $T_{\mathrm{max}} = \infty$, and recall from the proof of Lemma \ref{th2.6} that
		\[
		V'(t) \leq -d(\delta - \delta_{0}) + C_{f} V(t) \quad\text{for}\quad t>0.
		\]
		Then as in Step 2 above we obtain
		\[
		V(t) \leq e^{C_{f}(t-t_{0})} \left[ V(t_{0}) - \frac{d(\delta - \delta_{0})}{C_{f}} \right] + \frac{d(\delta - \delta_{0})}{C_{f}} <0 \quad \text{for all large} \quad t,
		\]
		which is a contradiction to $V(t)>0$.
		This completes the proof.
	\end{proof}
	
	The following two lemmas consider the behaviour of $h(t)$ when $T_{max}=\yy$.
	\begin{lemma}\label{lemma3.4}
		Suppose $0\leq\delta_{0}<\delta\in (\hat\theta_f, 1)$, and $(u, h)$ is the unique solution of \eqref{1.1} with maximal existence time    $T_{max}=\yy$. 
		Then  $h(t)$ exhibits one of the following asymptotic behaviours:
		\begin{itemize}\item[{\rm (i)}] 
			$h(t)$ is bounded away from 0 and $\infty$, namely
			\[
			0 < \liminf_{t \to \infty} h(t) \leq \limsup_{t \to \infty} h(t) < \infty
			\]
			
			\item[{\rm (ii)}] 
			
			$h(t)$ goes to infinity, i.e.,
			\[
			\lim_{t \to \infty} h(t) = \infty.
			\]
		\end{itemize}

	\end{lemma}
	
	\begin{proof} Since $f$ satisfies $(\bf{F})$, the function $w_*(x)$ satisfying the following identity is well-defined:
		\begin{equation}\label{w}
			\int_{w_*(x)}^\delta \frac{du}{\sqrt{\frac{2}{d}\int_{u}^{\delta} f(s) \, ds}}du =l_*-x \mbox{ for } x\in [0, l_*],
		\end{equation}
		where $l_*$ is given by \eqref{l*}, namely,
		\begin{equation*}
			l_* = \int_{\delta_{0}}^{\delta} \frac{dx}{\sqrt{\frac{2}{d}\int_{x}^{\delta} f(s)}ds}  \in (0, \infty).
		\end{equation*}
		
		Moreover, it is easy to check that $w=w_*$ satisfies
		\begin{equation}\label{2.32}
			\begin{cases}
				d w'' + f(w) = 0 \mbox{ for } x \in (0, l_*), \\
				w'(x) > 0 \mbox{ for }  x \in [0, l_*),\\
				w(0) = \delta_{0}, \quad w(l_*) = \delta, \quad w'(l_*) = 0.
			\end{cases}
		\end{equation}

		Clearly $(u(t,x), h(t)) \equiv (w_*(x), l_*)$ constitutes a steady-state solution of \eqref{1.1}.
		
		Now let $(u, h)$ be a solution of \eqref{1.1} with  $T_{\mathrm{max}} = \infty$. By Lemma \ref{lemma2.10} we have
		\begin{equation}
			V(t) \geq \frac{d(\delta - \delta_{0})}{C_{f}} \quad \text{for } t \in [0, \infty).
		\end{equation}
		Combining this with the estimate in Lemma \ref{le2.5}, we derive
		\begin{equation}
			\frac{C_{1}h^{2}(t)}{2} \geq \int_{0}^{h(t)} x u(t,x) \, dx \geq \frac{d(\delta - \delta_{0})}{C_{f}},
		\end{equation}
		and hence
		\begin{equation}\label{n3.7}
			h(t) \geq \sqrt{\frac{2d(\delta - \delta_{0})}{C_{f}C_{1}}}>0\quad \forall\, t > 0.
		\end{equation}
		
		Therefore we have the following three possibilities for $h(t)$:
		\begin{enumerate}[label=(\arabic*)]
			\item Bounded: $\displaystyle 0 < \liminf_{t \to \infty} h(t) \leq \limsup_{t \to \infty} h(t) < \infty$.
			\item Semi-unbounded: $\displaystyle 0 < \liminf_{t \to \infty} h(t) < \limsup_{t \to \infty} h(t) = \infty$.
			\item Fully unbounded: $\displaystyle \lim_{t \to \infty} h(t) = \infty$.
		\end{enumerate}
		
		To complete the proof of the lemma, it suffices to eliminate possibility (2). Suppose for contradiction that (2) happens; we are going to derive a contradiction in two steps.

	By Lemma \ref{nlower}, there exists a constant $t_0 \geq 0$,  such that 
		\[u(t,x)\geq \tilde{q}_c(x+x_0)\ \mbox{ for }\ t\geq t_0 \mbox{ and } x\in[0, h(t)].\]
		Since $\tilde{q}_c(x)\to \delta$ as $x\to \infty$,  there exists $L^*>0$ such that $\tilde{q}_c(x+x_0)\geq \theta^{**}>\hat\theta_f$ for $x\geq L^*$. Hence 
		\begin{equation}\label{uxj}
			u(t,x)\geq \tilde{q}_c(L^*+x_0)=\theta^{**} \ \mbox{ for }\ t\geq t_0,\ x\in[L^*,h(t)] \ \mbox{ if }\ h(t)\geq L^*.
		\end{equation}
		We will need the solution $W = W_{L}(t,x)$ of the  problem
		\begin{equation}\label{2.33}
			\begin{cases}
				W_{t} = d W_{xx} + f(W), & t > 0,\ L^* < x < L+L^*, \\
				W(t, L^*) = W(t, L+L^*) = \theta^{**}, & t > 0, \\
				W(0, x) = \theta^{**}, & L^* \leq x \leq  L+L^*,
			\end{cases}
		\end{equation}
		with a sufficiently large $L$. 
		
		{\bf Step 1}. We prove the existence of positive constants $L$ and $T$ such that $W_{L}(t,L^*+L/2)>\delta$ for $t\geq T$. 
		
		Since $\underline{W} \equiv \theta^{**}$ is a lower solution of \eqref{2.33} and  $\bar{W} \equiv 1$ is an upper solution of \eqref{2.33}, we deduce that  $W_{L}(t,x)$ is nondecreasing in $t$, and $\theta^{**}\leq W_{L}(t,x)\leq1$ for $t>0$, $x\in [L^*,L+L^*]$. It follows that 
		\[W_{L}^{\infty}(x):=\lim _{t \rightarrow \infty} W_{L}(t,x) \ \text{ exists, and}\ \theta^{**}\leq W_{L}^{\infty}(x)\leq1 \quad\text{for } x\in [L^*,L+L^*].\]
		Moreover, it is easy to see, by standard parabolic regularity, that the above limit holds in   $C^{2}([L^*,L+L^*])$,  and  $W_{L}^{\infty} $ satisfies
		\[-d\left(W_{L}^{\infty}\right)^{\prime \prime}(x)=f\left(W_{L}^{\infty}\right)\quad\text{in}\quad (L^*,L+L^*), \quad W_{L}^{\infty} (L^*)= W_{L}^{\infty}( L^*+ L)=\theta^{**} .\]
		Since $\theta^{**}\in (\hat\theta_f, 1)$ and $f$ satisfies ${\bf (F)}$, a simple shooting argument indicates that $W_L^\infty(x)$ is symmetric about $x=L^*+\frac L2$, and
		\[(W_{L}^{\infty})^{\prime}( L^*+ L/2)=0,\,(W_{L}^{\infty})^{\prime}( x)> 0 \,\,\text{in}\,\, [L^*,L^*+L/2)\, .\]
		Additionally, by the upper and lower solution methods, $W_{L} $ is  increasing with respect to $ L$.  Therefore
		\[W_{\infty}(x):=\lim _{L \rightarrow \infty} W_{L}^{\infty}(x) \text { exists for every } x \in [L^*, \infty), \, \,\text{with}\,\,(W_{\infty})^{\prime}(x)\geq0.\]
		By standard elliptic regularity, the above limit holds in $ C_{\text {loc }}^{2}([L^*, \infty))$, and
		\begin{equation}
			-d W_{\infty}^{\prime \prime}=f\left(W_{\infty}\right),\,W_{\infty}(L^*)=\theta^{**}, \,\theta^{**}< W_{\infty} \leq 1, \,(W_{\infty})^{\prime}(x)\geq0\,\,  \text { for } \,\,x \in (L^*, \infty).
		\end{equation}
		It then follows easily that $W_\infty(\infty)=1$, and $v=W_\infty$ is the unique solution of \eqref{1.3} with $\delta_0$ replaced by $\theta^{**}$.
		Consequently, for sufficiently large $L$, we have  $W_{L}^{\infty}(L^*+L/2)>\delta$. We may then take $T$ sufficiently large such that $W_{L}(t,L^*+L/2)>\delta$ for $t\geq T$.

		{\bf Step 2}: We prove that \( h'(t) > 0 \) for all large $t>0$. Clearly this is a contradiction to the assumption that (2) happens.
		
		Since by assumption $\limsup_{t \to \infty} h(t) = \infty$, there exists a sequence $\{t_{n}\}$ such that $t_{n} \to \infty$ and $h(t_{n}) \to \infty$ as $n \to \infty$. From Corollary \ref{c3.2}, we know that $|h'(t)|\leq C_{2}$. Hence, for $t>t_n$ we have
		\[
		h(t)\geq h(t_n)-C_2(t-t_n).
		\]
		In particular, for any given $T_1>0$ and all large $n$, say $n\geq n(T_1)$,
		\[
		h(t)\geq h(t_n)-C_2(T+2T_1)>L^*+L \mbox{ for } t\in[t_{n},\,t_{n}+T+2T_{1}].
		\]
		
		Consider the following ODE problem:
		\begin{equation*}
			\left\{\begin{array}{ll}
				V'(t)=f(V),\quad\text{for } t>0,\\
				V(0)=\theta^{**}.
			\end{array}\right.
		\end{equation*}
		It follows from $(\bf{F})$ that  $V(t)$ increases to $1$ as $t\to \infty$. Thus, there exists $T_1=T_1(\theta^{**},\delta)>0$ such that $V(T_1)=\delta$.
		
	By \eqref{uxj}, for $n\geq n(T_1)$, we have
	$	u(t,x)\geq \theta^{**}$ for $t\in[t_{n},\,t_{n}+T+2T_{1}]$ and $ x\in[L^*,h(t)]$.
		Therefore, for fixed $n\geq n(T_1)$ such that $t_{n}> \max \{ T,T_1\}$, we can easily verify that $v(t,x):=u(t_{n}+t,x)$ is an upper solution of \eqref{2.33} for $t\in [0, T+2T_1]$ and $x\in[L^*,L+L^*]$. Consequently, by the comparison principle and Step 1,
		\begin{equation}
			\label{L/2}
			v(t,L^*+ L/2)\geq W_L(t, L^*+L/2)\geq \delta  \quad\text{for } t\in [T,\,T+2T_1],
		\end{equation}
		and $v$ satisfies 
		\begin{equation*}
			\left\{\begin{array}{ll}
				v_{t}=d v_{x x}+f(v), & T<t<T+2T_{1},\, L^*+L/2<x<h(t+t_n), \\
				v(t, L^*+L/2)\geq \delta,\,v(t, h(t+t_n))=\delta, & T<t<T+2T_{1}, \\
				v(T, x)\geq\theta^{**}, & L^*+L/2<x<h(t_{n}+T).
			\end{array}\right.
		\end{equation*}
		This allows us to use the  comparison principle
		to conclude that
		\[u(t+t_{n},x)=v(t,x)\geq V(t-T) \quad\text{for } t \in[T,T+T_{1}] \text{ and } x\in[L^*+L/2,\,h(t_{n}+t)] .\] In particular, we have \[u(t_{n}+T+T_{1},x)\geq V(T_{1})=\delta \quad \text{for } x\in[L^*+L/2,\,h(t_{n}+T+T_{1})].\]
		Note that by \eqref{L/2} we have
		\[u(t,L^*+L/2)\geq \delta \quad\text{for } t\in [t_{n}+T+T_{1},\,t_{n}+T+2T_{1}].\] 
		Hence we can use the strong maximum principle  to deduce 
		\[
		u(t,x)> \delta \quad\text{for } t\in (t_{n}+T+T_{1},\,t_{n}+T+2T_{1}] \text{ and } x\in( L^*+L/2,\,h(t)).\]
		It then follows from the Hopt boundary lemma that 
		\begin{equation}\label{h'>0}
			h'(t)=-\frac{d}{\delta}u_{x}(t,h(t))> 0 \quad\text{for } t\in(t_{n}+T+T_{1},t_{n}+T+2T_{1}].
		\end{equation}

		For the above fixed \( t_n \), define  
		\begin{align*}
			T_2 := \sup \{ s : h'(t) > 0 \ \text{for all} \ t \in (t_n + T + T_1, s] \}.
		\end{align*}
		By \eqref{h'>0} we see that \( T_2 \) is well defined, and  \( T_2 \geq t_n + T + 2T_1 \).
		
		If \( T_2 < \infty \), then we must have \( h'(T_2) = 0 \). Since \( h'(t) > 0 \) for \( t \in [t_n + T + T_1, T_2)\), it follows that  
		\[h(t) >h(t_n+T+T_1) > L^*+ L \quad\text{for}\quad t \in [t_n + T + T_1, T_2] .\] 
		This allows us to repeat the argument leading to \eqref{L/2} to obtain 
		\( u(t, L^*+L/2) \geq \delta \) for \( t \in [t_n + T + T_1, T_2] \).  
		This in turn allows us to repeat the argument leading to \eqref{h'>0} to deduce
		\[
		h'(t) > 0 \mbox{ for } t\in (t_n+T+T_1, T_2],
		\]  
		which contradicts \( h'(T_2) = 0 \). Therefore, \( T_2 = \infty \) and \( h'(t) > 0 \) for all \( t > t_n + T + T_1 \), as desired.
		
		The above analysis shows that case (2) never happens, completing the proof of the lemma.
	\end{proof}

	If \(0\leq\delta_0<\delta<1\), $T_{\rm max}=\infty$ and \( h(t) \) is bounded away from 0 and $\infty$, we will show that actually \( \lim_{t \to \infty} h(t) \) exists. The proof of this fact is rather involved and requires  new ideas  well beyond those in \cite{DM, DL} for similar questions for \eqref{A} and \eqref{eq-kpp}, where the order-preserving property \eqref{order} holds.

	\subsection{Proof of \( \lim_{t \to \infty} h(t)=l_* \).}	The main purpose of this subsection is to sharpen the conclusion in part (i) of Lemma \ref{lemma3.4}, namely we prove the following result.
	
	\begin{theorem}\label{key1}
		If  $0\leq \delta_0 < \delta\in (\hat\theta_f, 1) $ and $0 < \displaystyle\liminf_{t \to \infty} h(t) \leq \limsup_{t \to \infty} h(t) < \infty$, then $\displaystyle\lim_{t \to \infty} h(t)=l_*$, where $l_*\in (0,\infty)$ is given by \eqref{l*}.
	\end{theorem}
	
	The proof of Theorem \ref{key1} will make use of some zero-number arguments based on the following result of Angenent \cite{Angenent} for  solutions to the parabolic equation	
	\begin{equation}\label{0001}
		\eta_{t}=a(t, x) \eta_{x x}+b(t, x) \eta_{x}+c(t, x) \eta \quad \text { in } Q:=\left\{(t, x) \mid \xi_{1}(t)<x<\xi_{2}(t), t \in\left(t_{1}, t_{2}\right)\right\}
	\end{equation}
	where  $\xi_{1}(t)$  and  $\xi_{2}(t)$  are continuous functions over $\left[t_{1}, t_{2}\right]$. For each  $t \in\left(t_{1}, t_{2}\right) $, denote by $	\mathcal{Z}(t)$
	the number of zeroes of  $\eta(t, \cdot)$  in the interval  ${\Omega}(t):=\left[\xi_{1}(t), \xi_{2}(t)\right] $. A point  $x_{0} \in {\Omega}(t)$  is called a  degenerate zero of  $\eta(t, \cdot) $ if  $\eta\left(t, x_{0}\right)=\eta_{x}\left(t, x_{0}\right)=0 $.
	
	\begin{lemma}\label{lemma0}{\rm (Angenent \cite{Angenent})}
		Assume the coefficients in \eqref{0001} satisfy
		\[	a, \frac 1a, a_{t}, a_{x}, b, c \in L^{\infty},
		\]
		and  $\eta\in W_{p}^{2,1}(Q)\cap C(\overline Q)$  for some $p>3$ is a solution of \eqref{0001}  and
		\begin{align*}
			\eta(t,  \xi_i(t)) \neq 0 \quad  {\rm for\ every }\ t \in (t_1, t_2),\ i=1,2.
		\end{align*}
		Then
		\begin{itemize}
			\item[{\rm (1)}] $ \mathcal{Z}(t)$  is finite and nonincreasing in  $t \in\left(t_{1}, t_{2}\right)$,
			\item[{\rm (2)}]  if  $s \in\left(t_{1}, t_{2}\right)$  and  $x_{0} \in {\Omega}(s)$  is a degenerate zero of  $\eta(s, \cdot) $, then  $\mathcal{Z}\left(s_{1}\right)>\mathcal{Z}\left(s_{2}\right)$  for all  $s_{1}$, $s_{2}$  satisfying  $t_{1}<s_{1}<s<s_{2}<t_{2} $.
		\end{itemize}
	\end{lemma}
	In the original statement of Angenent \cite{Angenent}, the functions $\xi_1(t)$ and $\xi_2(t)$ are constants, but it is easily seen that the result remains valid when they are continuous functions; for example, one can  break the interval $(t_1, t_2)$ into the union of small overlapping intervals $I_j:=(t_1^j, t_2^j)\subset [t_1, t_2]$, and replace $(\xi_1(t), \xi_2(t))$ by  $(\sup_{I_j}\xi_1(t),\inf_{I_j}\xi_2(t))$ for $t$ in each $I_j$, and then apply Angenent's result for $t$ in every $I_j$, which will lead to the result as stated in Lemma \ref{lemma0} above.

	\begin{remark}\label{rk-nondege}	Lemma \ref{lemma0} implies that in the interval $(t_1, t_2)$ there can be at most finitely many time moments $s^j$ such that $\eta(s^j,\cdot)$ has a degenerate zero in $(\xi_1(s^j),\xi_2(s^j))$, since otherwise $\mathcal Z(t)$ would become negative for some $t$ close to $t_1$. This implies in particular that for some $\epsilon>0$ sufficiently small,  and every $t\in (t_1, t_1+\epsilon)\cup (t_2-\epsilon, t_2)$, $\eta(t,\cdot)$ has only nondegenerate zeros in the interval $(\xi_1(t), \xi_2(t))$. This observation will be used later.
	\end{remark}
	
	\begin{lemma}\label{l8}  
		If 
		$0 \leq \delta_0 < \delta \in (\hat\theta_f, 1)$ and $\lim_{t \to \infty} h(t)=m\in(0,\infty)$, then 
		\[
		m =l_* \mbox{ and}
		\qquad 
		\lim_{t \to \infty} \sup_{x \in [0, h(t)]} \left| u(t, x) - w_*(x) \right| = 0,
		\]
		where  $w_*$ and $l_*$ are given by \eqref{w} and \eqref{l*}, respectively.
	\end{lemma}
	
	\begin{proof}  
		Define the transformation
		\[
		y := \frac{m x}{h(t)}, \quad U(t, y) := u(t, x).
		\]
		A direct computation shows that $ U $ satisfies
		\begin{equation}\label{3.9}
			\begin{cases}
				U_t - d \big[\frac{m}{h(t)}\big]^2 U_{yy} - \frac{h'(t)}{h(t)} yU_y = f(U), & t > 0, \, y \in [0, m], \\
				U(t, 0) = \delta_0, \, U(t, m) = \delta, & t > 0, \\
				h'(t) = -\frac{d}{\delta} \frac{m}{h(t)} U_y(t, m), & t > 0, \\
				U(0, y) = u_0\left(\frac{h_0 y}{m}\right), & y \in [0, m].
			\end{cases}
		\end{equation}
	By Corollary \ref{c3.2} we have $|h'(t)|\leq C_2$ for $t\geq 0$. Then by standard $ L^p $ theory, we deduce from \eqref{3.9} that for any $ p > 1 $,
		\[
		\|U\|_{W_p^{1,2}([n, n+2] \times [0, m])} \leq C_p,
		\]
		for all integers $ n \geq 1 $ and some $ C_p > 0 $ independent of $ n $. Taking $ p $ sufficiently large and applying the Sobolev embedding theorem, we find from the  equation on the third line of \eqref{3.9} that $ h'(t) $ is uniformly continuous for $ t \geq 1 $. This and $h(t)\to m$ imply $ h'(t) \to 0 $ as $ t \to \infty $.
		
		Let $ \{t_n\} $ be an arbitrary sequence such that $ t_n \to \infty $ as $ n \to \infty $, and define
		\[
		U_n(t, y) := U(t_n + t, y).
		\]
		Applying $ L^p $ estimates to the equation satisfied by $U_n$, and then the Sobolev embedding theorem, we find that, up to a subsequence, $ U_n \to \tilde{U} $ in $ C_{\text{loc}}^{\frac{1+\alpha}{2}, 1+\alpha}(\mathbb{R} \times [0, m]) $ for some $ \alpha \in (0, 1) $. The limit $ \tilde{U} $ satisfies (in the $W^{1,2}_p$ sense and then classical sense)
		\begin{equation}\label{3.2}
			\begin{cases}
				\tilde{U}_t - d \tilde{U}_{yy} = f(\tilde{U}), & t \in \mathbb{R}, \, y \in [0, m], \\
				\tilde{U}(t, 0) = \delta_0, \, \tilde{U}(t, m) = \delta, & t \in \mathbb{R}, \\
				\tilde{U}_y(t, m) = 0, & t \in \mathbb{R}.
			\end{cases}
		\end{equation}
		
		The last identity allows us to extend \( \tilde{U} \) to \( y\in [m, 2m] \) such that  the extended $\tilde U$  is  symmetric about \( y = m \) and still satisfies the same differential equation for $0<y<2m$. Similarly the function $V(y):=w_*(y-m+l_*)$ can be extended symmetrically about $y=m$ and the extended $V$  satisfies \eqref{3.2} for $y\in (m-l_*, m+l^*)$. Define  
		\(\eta(t, x) := \tilde{U}(t, x) -  w_*(x - m + l_*)\). Then,  
		\begin{equation}\label{3.111}  
			\eta(t, m) = \eta_x(t, m) = 0 \quad \text{for } t \in \mathbb{R}. 
		\end{equation}  
		The function \( \eta(t, x) \) satisfies, for some $L^\infty$ function $c(t,x)$,
		\[
		\eta_t=d\eta_{xx}+c(t,x)\eta \mbox{ for } t\in\mathbb R, \ \max\{0, m-l_*\}<x< \min\{2m, m+l_*\}. 		\]
		The degeneracy of \(x= m \) in \eqref{3.111}  would lead to a contradiction to Lemma \ref{lemma0} unless $\eta\equiv 0$.
		Indeed, if $\eta(t_0,x_0)\not=0$ for some $t_0\in\mathbb R$ and $\max\{0, m-l_*\}<x_0< \min\{2m, m+l_*\}$, then necessarily $x_0\not=m$.
		By symmetry and the continuity of $\eta$, there exists $\epsilon>0$ small so that
		\[
		\eta(t,x)\not= 0 \mbox{ for } t\in [t_0-\epsilon, t_0+\epsilon],\ x\in\{x_0, \ \tilde x_0\} \mbox{ with } \tilde x_0:=2m-x_0.
		\]
		Then \eqref{3.111} contradicts the statement in Remark \ref{rk-nondege}.
		Therefore, we must have $\eta\equiv 0$, which implies \( m=l_* \) and \( {U} \equiv w_* \), completing the proof.
	\end{proof}
	
	We next obtain a function $q_c(x)$ which will be used to construct a lower solution to estimate $h(t)$.
	
	\begin{lemma}\label{lem-lc}	Suppose  $0\leq \delta_0 < \delta\in (\hat\theta_f, 1) $. Then there exists $\epsilon_0>0$ sufficiently small such that for every $c\in [0, \epsilon_0]$,	
		the  problem
		\begin{equation}\label{qc1}
			\left\{
			\begin{array}{ll}
				d q'' - c q' + f(q) = 0 \mbox{ for } x \in (0, l), \\
				q'(x)>0 \mbox{ for } x\in [0, l),\\
				q(0) = \delta_0, \quad q(l) = \delta, \quad q'(l) = 0
			\end{array}
			\right.
		\end{equation}
		has a unique solution pair $(l,q(x))=(l_c, q_c(x))$. Moreover, 
		\[
		\mbox{\( l_c > l_0 = l_* \) for \( c \in (0, \epsilon_0] \), and  \( \lim_{c \to 0} l_c = l_* \).}
		\]
	\end{lemma}
	\begin{proof} Let
		\begin{equation}\label{l_0}
			\tilde l_0:= \int_0^{\delta} \frac{dx}{\sqrt{\frac{2}{d}\int_{x}^{\delta} f(s)}ds}  \in (0, \infty).
		\end{equation}
		As in the proof of Lemma \ref{lemma3.4}, the identity
		\begin{equation}\label{w_0}
			\int_{w(x)}^\delta \frac{du}{\sqrt{\frac{2}{d}\int_{u}^{\delta} f(s) \, ds}}du =\tilde l_0-x \mbox{ for } x\in [0, \tilde l_0],
		\end{equation}
		uniquely defines a function $w(x)$, which  satisfies
		\begin{equation}\label{l0}
			\begin{cases}
				d w'' + f(w) = 0 \mbox{ for } x \in (0, \tilde l_0), \\
				w'(x) > 0 \mbox{ for }  x \in [0, \tilde l_0),\\
				w(0) = {0}, \quad w(\tilde l_0) = \delta, \quad w'(\tilde l_0) = 0.
			\end{cases}
		\end{equation}
		
		To prove our conclusions in \eqref{qc1}, we rewrite the first equation in \eqref{qc1} as a first order ODE system and view it as a perturbation from the case $c=0$, which corresponds to \eqref{l0}, for which the solution is already known.
		
		More precisely, a solution pair $(l, q(x))$ to
		\eqref{qc1} gives a trajectory of the first-order system
		\begin{equation}\label{s2}
			q'(z) = p, \quad p' (z)= \frac{1}{d}(c p - f(q)),
		\end{equation}
		which connects the two points $(q,p)=(0, q'(0))$ and $(q,p)=(\delta, 0)$ in the $q$-$p$ plane, and lies above the $q$-axis. Moreover,
		if $(q(z_1), p(z_1))=(\delta_0, q'(0))$ and $(q(z_2), p(z_2))=(\delta, 0)$, then $l=z_2-z_1$.
		
		With $\tilde l_0$ and $w(x)$ given by \eqref{l_0} and \eqref{l0}, respectively, we easily see that
		\[
		q_0(z):=w(z+\tilde l_0),\ p_0(z):=w'(z+\tilde l_0)
		\]
		satisfies \eqref{s2} with $c=0$, and $(q_0(-\tilde l_0), p_0(-\tilde l_0))=(0, w'(0))$, $(q_0(0), p_0(0))=(\delta, 0)$. Clearly $(q_0(z), p_0(z))$ with $z\in [-\tilde l_0, 0]$ represents a $C^1$ curve in the $q$-$p$ plane that lies above the $q$-axis and connects the two points $(0,w'(0))$ and $(\delta, 0)$. Moreover, this curve contains no stationary point of \eqref{s2} with $c=0$.	Therefore, by the continuous dependence of ODE solutions on the parameter in the system, there exists $\epsilon_0>0$ small such that for every $c\in (0,\epsilon_0]$, the system \eqref{s2} has a unique solution $(q_c(z), p_c(z))$ satisfying 
		$(q_c(0), p_c(0))=(\delta, 0)$, which is defined for $z$ in some maximal interval containing $[-\tilde l_c, 0]$, with $\tilde l_c$ determined by $(q_c(-\tilde l_c), p_c(-\tilde l_c))=(0, \tilde p_0)$ and $\tilde p_0>0$ close to $w'(0)$.	As $p_c'(0)=\frac 1d(cp_c(0)-f(q_c(0)))=-\frac{f(\delta)}d<0$, we further see that $(q_c(z), p_c(z))$ with $z\in [-\tilde l_c, 0]$ gives a $C^1$ curve in the $q$-$p$ plane that connects the points $(0, \tilde p_c)$ and $(\delta, 0)$, and lies above the $q$-axis.
		Since $q_c(z)$ is strictly increasing in $[-\tilde l_c, 0]$ and $\delta_0\in [0, \delta)$, there exists a unique $z_c\in [\tilde l_c, 0)$ such that $q_c(z_c)=\delta_0$.
		Now define $l_c:=-z_c$ and $\hat q_c(x):=q_c(z_c+x)$; then it is straightforward to verify that $(l,q(x)):=(l_c, \hat q_c(x))$ satisfies \eqref{qc1}. Since $(q_c(z), p_c(z))\to (q_0(z), p_0(z))$ as $c\to 0$, it is easy to see that $l_c\to l_0$ when $c\to 0$. 
		
		To complete the proof, it remains to show that $l_c>l_0$ for $c\in (0, \epsilon_0]$. 
		Since \( q_c'(z)=p_c(z) > 0 \) for all \( z \in [0, l_c) \), the trajectory $(q_c(z), p_c(z))$ in the $q$-$p$ plane can be represented as a function \( p=P_c(q) \) with \( q \in [\delta_0, \delta) \), which satisfies
		\begin{equation}\label{s3a}
			P_c' = \frac{c}{d} - \frac{f(q)}{d P_c} =: F_c(q, P_c), \quad 
			P_c(\delta_0) = \omega_c:=p_c(z_c)>0, \quad \lim_{q \to \delta} P_c(q) = 0.
		\end{equation}
		Note that, if we define 
		\[
		\theta_f:=\begin{cases} 0 &\mbox{ if  $f$ satisfies } {\bf (f_m)},\\
		\theta &\mbox{ if $f$ satisfies ${\bf (f_b)}$ or ${\bf (f_c)}$},\end{cases},
		\]
		 then  for each fixed \( q \in (\theta_f,1) \), the function \( F_c(q, P) \) is strictly increasing in both \( c \) and \( P \). 		
		Let \( 0 \leq c_1 < c_2\leq\epsilon_0 \). We claim that
		\begin{equation}\label{A1.7}
			P_{c_1}(q) > P_{c_2}(q) \mbox{ for } q\in [\delta_0, \delta).
		\end{equation}
		
	If $\delta_0 < \theta_f$ (which can happen only if ${\bf (f_b)}$ or ${\bf (f_c)}$ holds), we first prove that the conclusion holds for all $q \in [\theta, \delta]$. Suppose, for contradiction, that there exists $z_0 \in [\theta, \delta)$ such that $0 < P_{c_1}(z_0) \leq P_{c_2}(z_0)$.
		Then it follows from \eqref{s3a} and the monotonicity of \( F_c(q, P) \) in both variables $c$ and $P$  that
		\[
		P_{c_1}'(z_0) = F_{c_1}(z_0, P_{c_1}(z_0)) < F_{c_2}(z_0, P_{c_2}(z_0)) = P_{c_2}'(z_0).
		\]
		We claim that \( P_{c_1}(q) < P_{c_2}(q) \) for all \( q \in (z_0, \delta) \). Otherwise due to $P_{c_1}'(z_0)<P_{c_2}'(z_0)$ and $ P_{c_1}(z_0) \leq P_{c_2}(z_0)$ there exists \( q_0 \in (z_0, \delta) \) such that $P_{c_1}(q) < P_{c_2}(q)$ for all $q \in (z_0, q_0)$, and $P_{c_1}(q_0) = P_{c_2}(q_0)$.
		Then we must have $P_{c_1}'(q_0) \geq P_{c_2}'(q_0)$,
		which contradicts the fact that $P_{c_1}'(q_0) = F_{c_1}(q_0, P_{c_1}(q_0)) < F_{c_2}(q_0, P_{c_2}(q_0)) = P_{c_2}'(q_0)$.
		Hence, \( P_{c_1}(q) < P_{c_2}(q) \) for all \( q \in (z_0, \delta) \).
		Since \( F_c(q, P) \) is strictly increasing in both \( c \) and \( P \), it then follows that $P_{c_1}'(q) < P_{c_2}'(q)$ for all $q \in (z_0, \delta)$,
		and thus
		\[
		P_{c_1}(\delta) = \int_{z_0}^{\delta} P_{c_1}'(q)\, dq + P_{c_1}(z_0)
		< \int_{z_0}^{\delta} P_{c_2}'(q)\, dq + P_{c_2}(z_0) = P_{c_2}(\delta),
		\]
		which contradicts the boundary condition \( P_c(\delta) = 0 \) for all $c$. Hence \eqref{A1.7} holds for $q\in[\theta, \delta]$.
		Since $P_{c_1}(\theta) > P_{c_2}(\theta)$, if \eqref{A1.7} does not hold in the remaining range $q\in [0,\theta)$, then there exists $z_1\in[0,\theta)$ such that
		\[P_{c_1}(q) > P_{c_2}(q) \mbox{ for } q\in(z_1,\theta) \mbox{ and } P_{c_1}(z_1) = P_{c_2}(z_1).\]
		Hence $P_{c_1}'(z_1) \geq P_{c_2}'(z_1)$, but by \eqref{s3a}, 
		$P_{c_1}'(z_1)=\frac{c_1}{d} - \frac{f(z_1)}{d P_{c_1}(z_1)}<\frac{c_2}{d} - \frac{f(z_1)}{d P_{c_2}(z_1)}=P_{c_2}'(z_1)$, which gives a contradiction.
		We have now proved \eqref{A1.7} for the case $\delta_0<\theta_f$.
		
	If $\delta_0 \geq \theta_f$, then the proof is simpler as the argument in the first part of the above analysis leads to the conclusion.
		
		From $q_c'(z)=p_c(z)=P_{c}(q_c(z))  $ we obtain
		\[ \frac {q_c'(z)}{P_c(q_c(z)) }=1,\ \mbox{ which implies } \int_{\delta_0}^\delta\frac {dq}{P_c(q)}=l_c.
		\]
		We may now use \eqref{A1.7} to obtain  $l_{c_1} < l_{c_2}$ for  $0 \leq c_1 < c_2\leq \epsilon_0$. In particular, $l_c>l_0$ for 
		\( c \in (0, 
		\epsilon_0] \).
	\end{proof}
	
	\begin{lemma}\label{lemma1.2}
		Under the assumptions of Theorem~\ref{key1}, we have $\limsup_{t \to \infty} h(t) \leq l_*$.
	\end{lemma}
	\begin{proof}
		The proof proceeds in three steps.
			
		\textbf{Step 1.} Construction of a lower solution.
	
		Fix \( c \in (0,\epsilon_0] \)  and   let $(k,l,q(x))=(k_c, l_c, q_c(x))$ be the solution pair solving  
			\begin{equation}\label{nqc1}
			\left\{
			\begin{array}{ll}
				d q'' - c q' + f(q) = 0 \mbox{ for } x \in (-k, l), \\
				q'(x)>0 \mbox{ for } x\in [-k, l),\\
				q(-k)=0, \quad	q(0) = \delta_0, \quad q(l) = \delta, \quad q'(l) = 0
			\end{array}
			\right.
		\end{equation}
obtained in	Lemma \ref{lem-lc}.  For \( L > h(0)+k_c \),	 the following function will serve as a lower solution for our analysis:
		\[
		\underline{u}_L(t,x): = 
		\begin{cases} 
			0 & \text{if } x < -c t + L-k_c, \\[4pt]
			q_c(x + c t - L) & \text{if } -c t + L-k_c \leq x \leq -c t + L + l_c.
		\end{cases}
		\]
		It is clear that \( \underline{u}_L(t,x) \) represents a profile moving leftward with speed \( c > 0 \). Denote 
		\[
		h_L^1(t) := -c t + L,\  h_L^2(t) := -c t + L + l_c,\ t_L^1 := \frac{L}{c}.		   \]
		Clearly  \( h_L^1(t_L^1) = 0 \). We now consider two possible cases regarding the position of the free boundary \(x= h(t) \) relative to 
		\(x= h_L^2(t) \) for \( t\in [0, t_L^1] \):
		\begin{align*}
			&\text{Case 1:} \quad h(t) < h_L^2(t) \quad \text{for all } t \in [0, t_L^1), \\
			&\text{Case 2:} \quad \text{there exists } t_1 \in (0, t_L^1) \text{ such that } h(t_1) = h_L^2(t_1), \
			h(t) < h_L^2(t) \ \text{for  } t \in (0, t_1).
		\end{align*}

		\textbf{Step 2.} We show that Case~2 cannot occur.
		
		Suppose, for contradiction, that Case~2 happens. We shall derive a contradiction by first proving that
		\begin{align*}
			h'(t) > 0 \quad \text{for all } t \geq t_1.
		\end{align*}
		This, together with the boundedness of \( h(t) \), i.e., \( \limsup_{t \to \infty} h(t) < \infty \), ensures that the limit \( \lim_{t \to \infty} h(t) \) exists. Lemma~\ref{l8} then yields $\lim_{t \to \infty} h(t) = l_*$.
		However, the assumption \( t_1 < t_L^1 \), combined with the identity \( h(t_1) = h_L^2(t_1) \), implies that 		
		\[
		h(t) > h(t_1) = h_L^2(t_1) > l_c \mbox{ for all \( t > t_1 \)},
		\]
		which contradicts the fact that \( l_* < l_c \). Thus, to complete Step~2, it suffices to show that \( h'(t) > 0 \) for all \( t \geq t_1 \).
		
		Due to \eqref{nqc1}, the function \( \underline{u}_L \) satisfies (in the weak sense)
		\begin{equation}\label{sxj}
			\left\{
			\begin{array}{ll}
				\partial_t \underline{u}_L - d\, \partial_{xx} \underline{u}_L - f(\underline{u}_L) \leq 0 = u_t - d u_{xx} - f(u), & 0 < t \leq t_1,\ 0 < x < h(t), \\
				\underline{u}_L(t, 0) \leq \delta_0 = u(t, 0), & 0 < t \leq t_1, \\
				\underline{u}_L(t, h(t)) < \delta = u(t, h(t)), & 0 < t < t_1, \\
				\underline{u}_L(0, x)=0 \leq u(0, x), & x \in [0, h(0)].
			\end{array}
			\right.
		\end{equation}
		Note that \( \partial_x \underline{u}_L(t, h_L^2(t)) = 0 \) for all \( t \geq 0 \). By the comparison principle for weak solutions, we have
		\[
		\underline{u}_L(t,x) < u(t,x) \quad \text{for } t \in [0, t_1],\ x \in (0, h(t)).
		\]
		Since \( h(t_1) = h_L^2(t_1) \), we have \( \underline{u}_L(t_1, h(t_1)) = u(t_1, h(t_1)) = \delta \). Then, by the Hopf boundary lemma,
		\[
		0 = \partial_x \underline{u}_L (t_1, h(t_1)) > \partial_x u(t_1, h(t_1)),
		\]
		which implies
		\[
		h'(t_1) = -\frac{d}{\delta} \partial_x u(t_1, h(t_1)) >0.
		\]
		If \( h'(t) \) is not strictly positive for all \( t \geq t_1 \), then there exists \( t_2 > t_1 \) such that
		\begin{align}\label{A1.3}
			h'(t) > 0 \quad \text{for } t \in [t_1, t_2),  \quad h'(t_2) = 0.
		\end{align}
		Define
		\[
		L_1 := h(t_2) + c t_2 - l_c,
		\]
		so that \( h_{L_1}^2(t_2) = h(t_2) \). We claim that \( t_2 \) is the first time at which \( h_{L_1}^2(t) \) touches \( h(t) \); that is,
		\begin{align}\label{A1.4}
			h_{L_1}^2(t) > h(t) \quad \text{and} \quad h_{L_1}^1(t) > 0 \quad \text{for all } t \in [0, t_2).
		\end{align}
		Indeed, for \( t \in [t_1, t_2] \), we have $\left[h_{L_1}^2(t) - h(t)\right]' = -c - h'(t) \leq -c< 0$, which implies, due to
		\( h_{L_1}^2(t_2) - h(t_2)=0 \),  that \( h_{L_1}^2(t) - h(t)>0 \) for all \( t \in [t_1, t_2) \). In particular,  \( h_{L_1}^2(t_1) > h(t_1) \). As 
		\( h_L^2(t_1) = h(t_1) \), the monotonicity of \( h_L^2(t) \) with respect to \( L \) leads to \( L_1 > L \). Consequently,
		\[
		h_{L_1}^2(t) > h_L^2(t) > h(t) \quad \text{for } t \in [0, t_1),
		\]
		proving the first inequality in \eqref{A1.4}. For the second inequality, observe that for \( t \in [0, t_2] \),
		\begin{align*}
			h_{L_1}^1(t) 
			&= -c t + L_1 
			= -c t + h(t_2) + c t_2 - l_c 
			\geq -c t_2 + h(t_2) + c t_2 - l_c 
			= h(t_2) - l_c \\
			&> h(t_1) - l_c 
			= h_L^2(t_1) - l_c 
			> 0,
		\end{align*}
		where the last inequality follows from the assumption of Case~2.
		
		From \eqref{A1.4}, we deduce that the function \( \underline{u}_L \), with \( L \) replaced by \(  L_1 \), satisfies the differential inequalities in \eqref{sxj} with $t_1$ replaced by $t_2$. Since \( h_{L_1}^2(t_2) = h(t_2) \),  we can use the Hopf boundary lemma as before to conclude \( h'(t_2) > 0 \), contradicting \eqref{A1.3}.
		
		Therefore, \( h'(t) > 0 \) for all \( t \geq t_1 \), and  Step~2 is completed.

		\textbf{Step 3.} Proof of \( \limsup_{t \to \infty} h(t) \leq l_* \).
		
		By the conclusion of Step~2, only Case~1 can occur. Therefore, for any \( L > h(0) \), we have
		\[
		h(t) \leq h_L^2(t_L^1) = l_c \mbox{ for } 0\leq t\leq t_L^1 := \frac{L}{c}.
		\]
		Letting $c\to 0$ we obtain $h(t)\leq \lim_{c\to 0}l_c=l_*$ for every $t\geq 0$. 
		This implies $\limsup_{t \to \infty} h(t)\leq {l}_{*}$.
	\end{proof}

	\begin{lemma}\label{liminf-h}
		Under the conditions of Theorem \ref{key1}, we have $\liminf_{t \to \infty} h(t)\geq l_*$.
	\end{lemma}
	
	\begin{proof}
	Corollary \ref{c3.2} implies that there exists $t_0>0$ such that 
	\begin{equation}\label{>delta0}
		u(t,x)>\delta_0,\ u_x(t,0)>0  \mbox{ for } t\geq t_0 \mbox{ and } x\in (0, h(t)].
	\end{equation}
		For convenience of notation and without loss of generality,  in the analysis below we assume $t_0=0$.

		\textbf{Step 1.} We introduce two auxiliary functions.
		
		In this step we obtain two auxiliary functions $\hat q_c(x)$ and $q_{\epsilon, c}(x)$ for small $c>0$ and $\epsilon>0$. They will be used to construct functions that can be easily compared with the density function $u(t,x)$ from \eqref{1.1} via some zero number arguments based on Lemma \ref{lemma0}.
		
		It is easily seen that the solution $w(x)$ of \eqref{l0} can be extended from $x\in [0, \tilde l_0]$ to $x\in [0, 2\tilde l_0]$ to obtain a function $\hat w(x)$ which is symmetric about $x=\tilde l_0$, and the extended function $\hat w$ satisfies
		\begin{equation}\label{2l0}
			\begin{cases}
				d \hat w'' + f(\hat w) = 0 \mbox{ for } x \in (0, 2\tilde l_0), \\
				-\hat w'(2\tilde l_0-x)=\hat w'(x) > 0 \mbox{ for }  x \in [0, \tilde l_0),\ \\
				\hat w(0) =\hat w(2\tilde l_0)= {0}, \quad \hat w(\tilde l_0) = \delta, \quad \hat w'(\tilde l_0) = 0.
			\end{cases}
		\end{equation}

		Through the formulas $q(z)=\hat w(z),\ p(z)=q'(z)$, the function $\hat w(x)$  gives a trajectory of the first-order system
		\begin{equation}\label{c=0}
			q'(z) = p, \quad p' (z)= -\frac{1}df(q),
		\end{equation}
		in the $q$-$p$ plane,
		which connects the points $(q,p)=(0, q'(0))$ and $(q,p)=(\delta, 0)$  in the upper half plane, and then symmetrically continues from $(q,p)=(\delta, 0)$ to 
		$(q,p)=(0, -q'(0))$ in the lower half plane. 
		
		As in the proof of Lemma \ref{lem-lc}, we now consider the system \eqref{s2}  with $c>0$ small and view it as a  perturbation  of \eqref{c=0}. Similar to the arguments in the proof of Lemma \ref{lem-lc}, there exists $\hat\epsilon_0>0$ small so that for each $c\in (0, \hat\epsilon_0]$, \eqref{s2} has a trajectory close to the one generated by $\hat w(x)$, which gives rise to a triple $(l_c, l^c, \hat q_c(x))$ that solves the following system:
		\begin{equation}\label{0q1}
			\left\{
			\begin{array}{ll}
				d \hat q'' - c \hat q' + f(\hat q) = 0, & x \in (-l_c,l^c), \\
				\hat q'(x) > 0, & x \in [-l_c, 0),\\
				\hat q'(x) < 0, & x \in (0, l^c],\\
				\hat q( -l_c)=\hat q(l^c) = \delta_0, \quad \hat q(0) = \delta, \quad \hat q'(0)=0.
			\end{array}
			\right.
		\end{equation}
		
		If $\delta$ is replaced by $\delta-\epsilon$ with $\epsilon\in (0, \delta-\delta_0)$ sufficiently small, then the unique solution of \eqref{s2} with $(q(0), p(0))=(\delta-\epsilon, 0)$ similarly gives rise to a triple 	
		\((l_{c,\epsilon},l^{c,\epsilon}, q_{c, \epsilon}(x))\), with $q=q_{c,\epsilon}(x)$ 
		satisfying
		\[
		\left\{
		\begin{array}{ll}
			d q'' - c q' + f(q) = 0, & x \in (-l_{c,\epsilon},l^{c,\epsilon}), \\
			q'(x) < 0, & x \in (0,l^{c,\epsilon}],\\
			q'(x) > 0, & x \in [-l_{c,\epsilon},0),\\
			q(-l_{c,\epsilon}) = q(l^{c,\epsilon})=\delta_0,\quad q(0) = \delta-\epsilon,\quad \quad q'(0)=0.	
		\end{array}
		\right.
		\]
		Moreover,
		\[ \lim_{\epsilon\to 0}(l_{c,\epsilon}, l^{c,\epsilon})=(l_c, l^c),\ 
		\lim_{c,\epsilon\to 0} (l_{c,\epsilon}, l^{c,\epsilon})=(l_*, l_*),
		\]
		and
		\begin{align}\label{3.69e}
			\max_{x\in[-l_{c,\epsilon},\,l^{c,\epsilon}]\cap[-l_c,\,l^c]}
			|q_{c,\epsilon}(x)-q_c(x)| \longrightarrow 0
			\quad \text{as } \epsilon\to 0.
		\end{align}

		\medskip

		\textbf{Step 2.} We prove a zero number conclusion (see \eqref{Z=1}).
		
		With \((l_{c,\epsilon},l^{c,\epsilon}, q_{c,\epsilon}(x))\) given in Step 1, we fix  \(L\le -l_{c,\epsilon}\) and define
		\[
		v_{\epsilon,L}(t,x):
		= q_{c,\epsilon}(-x+ct+L) \mbox{ for } \ t\in\mathbb R,\			ct+L-l^{c,\epsilon}\le x\le ct+L+l_{c,\epsilon}.
		\]
		Clearly $v_{\epsilon,L}$ satisfies 
		\[
		\partial_tv_{\epsilon,L}=d\partial_{xx}v_{\epsilon,L}+f(v_{\epsilon,L}) \mbox{ for } \ t\in\mathbb R,\	ct+L-l^{c,\epsilon}< x< ct+L+l_{c,\epsilon}.
		\]
		We will call $v_{\epsilon,L}(t,x)$ a traveling wave piece of speed $c$, whose
		left boundary, peak location and right boundary are given respectively by
		\[
		G_\epsilon(t):=ct+L-l^{c,\epsilon},\qquad
		I_\epsilon(t):=ct+L,\qquad
		H_\epsilon(t):=ct+L+l_{c,\epsilon}.
		\]
		
		Define
		\[\begin{cases}
			\eta_\epsilon(t,x):=u(t,x)-v_{\epsilon,L}(t,x),\\
			\xi_\epsilon(t):=\min\{H_\epsilon(t),\,h(t)\},\\
			t_{\epsilon,L}^0:=\frac{-L-l_{c,\epsilon}}{c},\ 
			T_{\epsilon,L}:=\frac{l^{c, \epsilon}-L}{c}.
		\end{cases}
		\]
		Then clearly
		\begin{equation}\label{t-T}
			\begin{cases}
				0\leq t_{\epsilon,L}^0<T_{\epsilon,L},\ G_\epsilon(T_{\epsilon,L})=H_\epsilon(t_{\epsilon,L}^0)=0,\ \ 	\\
				G_\epsilon(t)< 0< H_\epsilon(t) \mbox{ and }	
				\eta_\epsilon(t, 0)<0<\eta_\epsilon\big(t,\xi_\epsilon(t)\big)
				\quad \text{for } t\in(t_{\epsilon,L}^0,\,T_{\epsilon,L}).
			\end{cases}
		\end{equation}
		Moreover, $\eta=\eta_\epsilon$ satisfies
		\[
		\eta_t=d\eta_{xx}+c(t,x)\eta \mbox{ for } t\in(t_{\epsilon,L}^0,\,T_{\epsilon,L}),\  0< x< \xi_\epsilon(t),
		\]
		with $c(t,x)$ some $L^\infty$ function.
		Let \(\mathcal Z_\epsilon(t)\) denote the number of zeros of
		\(\eta_\epsilon(t,\cdot)\) in the interval
		\(\Omega(t):=(0,\xi_\epsilon(t))\).
		By \eqref{t-T} and Lemma \ref{lemma0}, $\mathcal Z_\epsilon(t)$ is finite and nonincreasing for $t\in(t_{\epsilon,L}^0,\,T_{\epsilon,L})$. We claim  that
		\begin{align}\label{Z=1}
			\mathcal{Z}_\epsilon(t) \equiv 1 \mbox{ for } t\in (t_{\epsilon,L}^0,T_{\epsilon,L}).
		\end{align}

		To simplify notations, in the following discussion, we will drop the subscript $\epsilon$ and denote
		\[
		(t_{\epsilon,L}^0,\,T_{\epsilon,L},\ \xi_\epsilon, \ \eta_\epsilon,\ G_\epsilon, H_\epsilon) \mbox{ by } (t_{L}^0,\,T_{L},\ \xi, \ \eta,\ G, H).
		\]
		It is clear that $\eta(t_L^0,0)=0$ and by \eqref{>delta0} (with the assumption $t_0=0$),
		\[
		\eta_x(t_L^0,0)=u_x(t_L^0  ,0)-(v_L)_x(t_L^0,0)=u_x(t_L^0  ,0)+q_{c,\epsilon}'(-l_{c,\epsilon})>0.
		\]
		Hence, for $0< \epsilon_0\ll 1$ we have
		\[
		0<\xi(t)=H(t)\ll 1,\  \eta_x(t,x)>0 \mbox{ for } t\in(t_L^0,t_L^0+\epsilon_0),\ x\in [0, \xi(t)].
		\]
		This implies, due to
		\[
		\eta(t,0)<0<\eta(t,\xi(t)) \mbox{ for } t\in(t_L^0,t_L^0+\epsilon_0),
		\]
		 that for every
		$t\in(t_L^0,t_L^0+\epsilon_0]$, $\eta(t,\cdot)$ has a unique nondegenerate zero  $x_t\in (0, \xi(t))=(0,H(t))$, and 
		\[
		\eta(t,x)<0 \text{ for } x\in[0,x_t),\qquad
		\eta(t,x_t)=0,\qquad
		\eta(t,x)>0 \text{ for } x\in(x_t,\xi(t)].
		\]
		Hence, $\mathcal Z_\epsilon(t)\equiv1$ for
		$t\in(t_L^0,t_L^0+\epsilon_0]$. On the other hand, by \eqref{t-T} we see that $\mathcal Z_\epsilon(t)\geq 1$ for $t\in(t_L^0,T_L)$. We may now apply the nonincreasing property of $\mathcal Z_\epsilon(t)$ to conclude that \eqref{Z=1} holds. This completes Step 2.

		\textbf{Step 3.} We prove a monotonicity property of $h(t)$ (see \eqref{h-prop}).
		
		Parallel to Step 2, using \((l_{c},l^{c}, \hat q_{c}(x))\) given in Step 1, we fix  \(L\le -l_{c}\) and define
		\begin{equation}\label{v_L}
			v_{L}(t,x):
			= \hat q_{c}(-x+ct+L) \mbox{ for } \ t\in\mathbb R,\			ct+L-l^{c}\le x\le ct+L+l_{c}.
		\end{equation}
		Clearly $v_{L}$ satisfies 
		\[
		\partial_tv_{L}=d\partial_{xx}v_{L}+f(v_{L}) \mbox{ for } \ t\in\mathbb R,\	ct+L-l^{c}< x< ct+L+l_{c}.
		\]
		So $v_{L}(t,x)$ is a traveling wave piece of speed $c$, whose
		left boundary, peak location and right boundary are given respectively by
		\[
		G(t):=ct+L-l^{c},\qquad
		I(t):=ct+L,\qquad
		H(t):=ct+L+l_{c}.
		\]
		Define
		\[\begin{cases}
			\eta(t,x):=u(t,x)-v_{L}(t,x),\\
			\xi(t):=\min\{H(t),\,h(t)\},\\
			t_{L}^0:=\frac{-L-l_{c}}{c},\ 
			T_{L}:=\frac{l^{c}-L}{c}.
		\end{cases}
		\]
		Then 
		\begin{equation}\label{t-T-a}
			\begin{cases}
				0\leq t_{L}^0<T_{L},\ G(T_{L})=H(t_{L}^0)=0,\ \ 	\\
				G(t)< 0< H(t),\
				\eta(t, 0)<0
				\quad \text{for } t\in(t_{L}^0,\,T_{L}),
			\end{cases}
		\end{equation}
		and 
		\[
		\eta_t=d\eta_{xx}+c(t,x)\eta \mbox{ for } t\in(t_{L}^0,\,T_{L}),\  0< x< \xi(t),
		\]
		with $c(t,x)$ some $L^\infty$ function.
		
		We  show that $h(t)$ has the following monotonicity property:
		\begin{equation}\label{h-prop}
			\begin{cases} \mbox{ if there exists $\widetilde{t_L}\in (0,T_L)$ such that
					$h(\widetilde{t_L})=I(\widetilde{t_L})<l^c$ and $
					h'(\widetilde{t_L})<0$,}\\
				\mbox{ then
					$
					h'(t)<0 \quad \text{for all } t\geq \widetilde{t_L}.$}\end{cases}
		\end{equation}
		
		Due to the free boundary condition, the assumption $h'(\widetilde{t_L})< 0$ leads to
		\[
		u_x(\widetilde{t_L},h(\widetilde{t_L}))=-\frac d\delta h'(\widetilde{t_L})=:\tilde\sigma>0.
		\]
		In view of $I(\tilde t_L)=h(\tilde t_L)$, this implies the existence of  some $\epsilon_1>0$ sufficiently small such that
		\[
		u_x(t,x)=\tilde\sigma+o_{\epsilon_1}(1)>0
		\quad \mbox{for }\
		t\in [\widetilde{t_L}-\epsilon_1,\widetilde{t_L}],
		\quad
		x\in [I(t)-\epsilon_1,\,h(t)],
		\]
		where $o_{\epsilon_1}(1)$ denotes a generic quantity which goes to 0 as $\epsilon_1\to 0$.
		Moreover, by shrinking $\epsilon_1>0$ we may also assume that
		\begin{equation}\label{I<h}
			I(t)<I(\tilde t_L)=h(\tilde t_L)<h(t) \mbox{ for } t\in [\widetilde{t_L}-\epsilon_1,\widetilde{t_L}).
		\end{equation}
		Since $u(t,h(t))=\delta=v_L(t,I(t))$ and $(v_L)_x(t,I(t))\equiv 0$, it follows that
		\begin{equation}\label{3.71a}\begin{aligned}
			\eta(t,x)&=u(t,x)-v_L(t,x)=[\delta-(\tilde\sigma+o_{\epsilon_1}(1))(h(t)-x)]-[\delta+o_{\epsilon_1}(1)(I(t)-x)]\\
			&\leq  [\delta-(\tilde\sigma+o_{\epsilon_1}(1))(I(t)-x)]-[\delta+o_{\epsilon_1}(1)(I(t)-x)]\\
			&= -(\tilde\sigma+o_{\epsilon_1}(1))(I(t)-x)\\
			&<0
			\ \mbox{ for }
			t\in [\widetilde{t_L}-\epsilon_1,\widetilde{t_L}],
			\ 
			x\in [I(t)-\epsilon_1,I(t)).
			\end{aligned}
		\end{equation}
		For 
		$t\in [\widetilde{t_L}-\epsilon_1,\widetilde{t_L})$, let $\widetilde{\mathcal{Z}}(t)$ denote the number of zeros of $\eta(t,\cdot)$ in $(0,I(t)-\epsilon_1)$. 
		By \eqref{t-T-a} and \eqref{3.71a}, we see that
		\begin{equation}\label{<0}
			\eta(t,0)<0,
			\qquad
			\eta(t,I(t)-\epsilon_1)<0 \ \mbox{ for }
			t\in [\widetilde{t_L}-\epsilon_1,\widetilde{t_L}),
		\end{equation}
		and then by Remark \ref{rk-nondege} there exists $\epsilon_2\in (0,\epsilon_1)$ such that $\eta(t,\cdot)$ has no degenerate zeros
		in $(0, I(t)-\epsilon_1)$ for $t\in [\widetilde{t_L}-\epsilon_2,\widetilde{t_L})$.	Hence 		
		\[
		\widetilde{\mathcal{Z}}(t)\equiv \text{constant} \mbox{ for }
		t\in [\widetilde{t_L}-\epsilon_2,\widetilde{t_L}).
		\]
		
		{\bf Claim 1:}
		$\widetilde{\mathcal{Z}}(t)\equiv 0$ for $t\in [\widetilde{t_L}-\epsilon_2,\widetilde{t_L})$.
		
		Otherwise, due to \eqref{<0}, there are at least two nondegenerate zeros $x_t,y_t\in (0,I(t)-\epsilon_1)$ satisfying
		\begin{equation}\label{x_t-y_t}
			\eta(t,x_t)=\eta(t,y_t)=0,
			\qquad
			\eta_x(t,x_t)\neq 0,\ \eta_x(t,y_t)\neq 0 \ \mbox{ for every }
			\
			t\in [\widetilde{t_L}-\epsilon_2,\widetilde{t_L}).
		\end{equation}
		We show next that this leads to a contradiction to \eqref{Z=1}. 
		
		Firstly we note that $H(\tilde t_L)>I(\tilde t_L)=h(\tilde t_L)>0=H(t^0_L)$, and hence $t^0_L<\tilde t_L<T_L$. As $t^0_{\epsilon, L}\to t^0_L$ and $T_{\epsilon, L}\to T_L$ as $\epsilon\to 0$, we thus have
		\begin{equation}\label{ep-1}
			t^0_{\epsilon, L}<\tilde t_L-\epsilon_2<T_{\epsilon, L}
		\end{equation}
		provided that both $\epsilon$ and $\epsilon_2$ are sufficiently small. 
		
		Secondly, by \eqref{3.69e} we see that $\eta_\epsilon(t,x):=u(t,x)-v_{\epsilon,L}(t,x)$ is close to $\eta(t,x)$ for sufficiently small $\epsilon>0$ and therefore,
		due to \eqref{x_t-y_t}, for such $\epsilon$, 
		\begin{equation}\label{ep-2}
			\mbox{$\eta_\epsilon(\tilde t_L-\epsilon_2,\cdot)$ has at least two zeros in	$(0,I(\tilde t_L-\epsilon_2)-\epsilon_1)$.}
		\end{equation}
		
		Thirdly due to \eqref{I<h} and $I(t)<H(t)$ we obtain $I(t)<\xi(t)$ for $t\in [\widetilde{t_L}-\epsilon_1,\widetilde{t_L})$. In particular,
		$I(\tilde t_L-\epsilon_2)<\xi(\tilde t_L-\epsilon_2)$. Since $\xi_\epsilon(t)=\min\{H_\epsilon(t),\,h(t)\}\to\xi(t)$ as $\epsilon\to 0$, we can conclude that
		\[
		I(\tilde t_L-\epsilon_2)<\xi_\epsilon(\tilde t_L-\epsilon_2) \mbox{ for all  small } \epsilon.
		\]
		Therefore \eqref{ep-2} implies that 	$\eta_\epsilon(\tilde t_L-\epsilon_2,\cdot)$ has at least two zeros in $(0,\xi_\epsilon(\tilde t_L-\epsilon_2))$.
		Recalling the definition of 	$\mathcal Z_\epsilon(t)$ in Step 2, we thus have 
		\[
		\mathcal Z_\epsilon(\tilde t_L-\epsilon_2)\geq 2.
		\]
		In view of \eqref{ep-1}, this contradicts \eqref{Z=1}.
		We have thus proved  Claim 1.

		{\bf Claim 2.} 	Claim 1 and \eqref{3.71a} imply $h'(t)<0$ for all $t>\tilde t_L$.
		
		Clearly Claim 1  and \eqref{3.71a} imply
		\begin{equation}\label{t_L}
			\eta(t,x)=u(t,x)-v_L(t,x)<0 \ \mbox{ for } \
			t\in [\widetilde{t_L}-\epsilon_2,\widetilde{t_L}],
			\quad
			x\in [0,I(t)).
		\end{equation}
		Therefore to prove Claim 2, it suffices to show that 
		\[
		\mbox{\eqref{t_L} implies $h'(t)<0$ for all $t>\tilde t_L$.}
		\]
		Otherwise, due to our assumption in \eqref{h-prop} that $h'(\tilde t_L)<0$,  there exists $T_0>\tilde t_L$ such that
		\[
		h'(t)<0 \quad \text{for } t\in[\tilde t_{L},T_0),
		\qquad
		h'(T_0)=0.
		\]
		
		We first show that
		\begin{align}\label{3.68b}
			u(t,x)<\delta
			\quad \text{for all } t\in[\tilde t_{L},T_0],
			\ \ x\in[0,h(t)).
		\end{align}
		Since $h'(\tilde t_L)<0$ and hence $u_x(\tilde t_L, h(\tilde t_L))>0$, we obtain $h(t)<h(\tilde t_L)=I(\tilde t_L)<I(t)$ for $t>\tilde t_L$ but close to $\tilde t_L$, and
		\[
		u(t,x)<u(t, h(t))=\delta \mbox{ for } t\in [\tilde t_L, \tilde t_L+\tilde\epsilon_1],\ x\in [h(\tilde t_L)-\tilde\epsilon_2, h(t))
		\]
		provided that $\tilde\epsilon_2>\tilde\epsilon_1>0$ are sufficiently small. Moreover, by \eqref{t_L} we see that
		\[
		u(\tilde t_L, x)<\delta \mbox{ for } x\in [0, h(\tilde t_L)-\tilde\epsilon_2]=[0, I(\tilde t_L)-\tilde\epsilon_2].
		\]
		Therefore, we can find $\tilde\epsilon_0\in (0, \tilde\epsilon_1)$ such that
		\[
		u(t, x)<\delta \mbox{ for } t\in[\tilde t_L, \tilde t_L+\tilde\epsilon_0],\ x\in [0, h(t)).
		\]
		
		Assume to the contrary that \eqref{3.68b} is false.
		Then there exist $T_1\in(\tilde t_{L}+\tilde\epsilon_0,T_0]$ and
		$x_0\in(0,h(T_1))$ such that
		\begin{equation}\label{T1}
			u(T_1,x_0)=\delta \mbox{ and } u(t,x)<\delta
			\quad \text{for } t\in[\tilde t_{L},T_1),
			\ \ x\in[0,h(t)).
		\end{equation}

		For $c>0$ small, define
		\[
		L^*:=h(T_1)-cT_1,
		\]
		and let $v_{L^*}(t,x)$ be the traveling wave piece of speed $c$ given by \eqref{v_L} but with $L$ replaced by $L^*$, whose  left boundary, peak location 
		and right
		boundary are respectively given by
		\[
		G^*(t):=ct+L^*-l^{c},\qquad
		I^*(t):=ct+L^*,\qquad
		H^*(t):=ct+L^*+l_{c}.
		\]
		Then
		\[\begin{cases}
			I^*(T_1)=h(T_1)<h(\tilde t_L)<l^c,\\
			G^*(T_1)=I^*(T_1)-l^c<0,\\
			L^*=h(T_1)-cT_1<
			h(\tilde t_{L})-cT_1 =c(\tilde t_{L}-T_1)+L
			< L,
		\end{cases}
		\]
		and
		\[
		h(t)>h(T_1)=I^*(T_1)>I^*(t) ,\ 0>G^*(T_1)>G^*(t) \mbox{ for } t\in (\tilde t_L, T_1).
		\]
		Let 
		\[
		t_1:=-L^*/c=T_1-h(T_1)/c, \ \hat t_1:=\max\{t_1, \tilde t_L\}.
		\]
		Then $0<t_1\leq \hat t_1<T_1$ and $I^*(\hat t_1)\geq I^*(t_1)=0$. 
		
		If $t_1<\tilde t_L$ and hence $\hat t_1=\tilde t_L$, then
		since $\hat q_c'(x)<0$ for $x\in (0, l^c]$ and $G(\tilde t_L)=I(\tilde t_L)-l^c=h(\tilde t_L)-l^c<0$, it follows from $L^*<L$ that
		\[
		v_L(\tilde t_L,x)<v_{L^*}(\tilde t_L,x) \mbox{ for } 0\leq x\leq I^*(\tilde t_L)=I^*(\hat t_1).
		\]
		Hence by \eqref{t_L} we have
		\begin{equation}\label{hat-t1}
		u(\hat t_1,x)<v_L(\hat t_1, x)<v_{L^*}(\hat t_1,x) \mbox{ for } x\in (0, I^*(\hat t_1)]\subset (0,I(\hat t_1)).
		\end{equation}
		
		If $t_1\geq \tilde t_L$ and hence $\hat t_1=t_1$,  we have $I^*(\hat t_1)=0$ and
		\begin{equation}\label{hat-t1-2}
		u(\hat t_1, I^*(\hat t_1))=u( \hat t_1, 0)=\delta_0<\delta=v_{L^*}(\hat t_1, I^*(\hat t_1)).
		\end{equation}
		
		By the differential equations satisfied by $u$ and $v_{L^*}$, \eqref{hat-t1}, \eqref{hat-t1-2} and \eqref{T1}, we see that
		$\eta^*:=u-v_{L^*}$ satisfies
		\[
		\begin{cases}
			\eta^*_t - d\eta^*_{xx} - C(t,x)\eta^*= 0,
			& \hat t_1< t \leq T_1,\ 0 < x < I^*(t), \\[2pt]
			\eta^*(t,0) < 0,\quad
			\eta^*\bigl(t, I^*(t)\bigr) \leq 0,
			& \hat t_1 < t \leq T_1, \\[2pt]
			\eta^*(\hat t_1, x)< 0,
			& 0\leq x\leq I^*(\hat t_1),
		\end{cases}
		\]
		where $C(t,x)$ is some $L^\infty$ function. Since $\eta^*(T_1,I^*(T_1))=0$, applying the strong maximum principle and the
		Hopf boundary lemma we obtain
		\[
		\eta^*(T_1, x) < 0 \quad \text{for } x \in [0, I^*(T_1)),
		\qquad
		\eta^*_x\bigl(T_1, I^*(T_1)\bigr) > 0.
		\]
		Consequently,
		\begin{align*}
			u(T_1, x)
			&< v_{L^*}(T_1, x)
			= q_c(-x +I^*(T_1))
			< \delta
			\quad \text{for } x \in [0, h(T_1))=[0,I^*(T_1)),
		\end{align*}
		and
		\[
		u_x\bigl(T_1, h(T_1)\bigr) =\eta^*_x\bigl(T_1, I^*(T_1)\bigr)> 0.
		\]
		This contradicts the assumption that
		$u(T_1, x_0) = \delta$ with $x_0 \in (0, h(T_1))$.
		Hence, \eqref{3.68b} holds.

		The validity of \eqref{3.68b} allows us to repeat the above argument with $T_1$ replaced by $T_0$ to obtain
		$u_x\bigl(T_0, h(T_0)\bigr) > 0$. Hence
		\[
		h'(T_0)
		=-\frac{d}{\delta}\,
		u_x\bigl(T_0, h(T_0)\bigr)
		< 0,
		\]
		which contradicts the assumption $h'(T_0)=0$.
		Therefore   $h'(t)<0$ for all $t\ge \tilde t_L$. This proves the desired property \eqref{h-prop} and
		Step 3 is now completed.
		\medskip
		
		\textbf{Step 4.} We prove that $\liminf_{t\to\infty} h(t)\geq l_*$.
		
		Suppose, by contradiction, that $\liminf_{t\to\infty} h(t)< l_*$. Then, due to
		\( \lim_{c \to 0} l^c = l_* \), for $c>0$ sufficiently small we have
		\[
		h_*:=\liminf_{t\to\infty} h(t)< l^c.
		\]
		We are going to show that this  leads to  a contradiction.
		
		Let $M>0$ be such that $h(t)\leq M$ for all $t\geq 0$, and define
		$t^*:=\frac{M+l_c+l^c}{c}$.
		Then,
		\[
		h(t)-ct\leq M-ct^*=-l_c-l^c<    -	l_c \ \mbox{  for all $t\geq t^*$}.
		\]
		Consequently, for any $s\geq t^*$, if $h(s)\in (0,l^c)$, then 
		\[
		L_s:=h(s)-cs<l^c-ct^*<-l_c. 
		\]
		We now apply \eqref{h-prop} with $L=L_s$ and see that 
		\begin{equation}\label{Ls}
			\begin{cases} \mbox{ if there exists $\widetilde t_{L_s}\in (0,T_{L_s})$ such that
					$h(\widetilde t_{L_s})=I(\widetilde t_{L_s})<l^c$ and $
					h'(\widetilde t_{L_s})<0$,}\\
				\mbox{ then
					$
					h'(t)<0 \quad \text{for all } t\geq \widetilde t_{L_s}.$}\end{cases}
		\end{equation}
		Our choice of $L_s$ and the assumption $h(\tilde t_{L_s})\in (0, l^c)$ imply 
		\[
		\mbox{$\widetilde t_{L_s}=s$ and $T_{L_s}=	\frac{l^c-L_s}c=s+\frac{l^c-h(s)}c>s.$}
		\]
		Therefore \eqref{Ls} 
		implies that for any $s\geq t^*$, the conditions
		\begin{equation}\label{3.71d}
			h(s)\in (0,l^c) \mbox{ and } h'(s)<0 
		\end{equation}
		imply $  h'(t)< 0 \mbox{ for all } t\geq s$, and therefor $m=\lim_{t\to\infty} h(t)=h_*<l_*$.  But by Lemma~\ref{l8} in such a case we must have $m=l_*$. This contradiction shows that \eqref{3.71d} cannot happen, in other words, the following conclusion holds:
		\[
		s\geq t^* \mbox{ and } h(s)\in (0,l^c) \implies h'(s)\geq 0.
		\]
		
		We now consider $h^*:=\limsup_{t\to\infty} h(t)$. If $h^*<l_*$ then $h^*<l^c$ for small $c>0$ and the above conclusion then implies  $h(t)$ is
		nondecreasing for all  large $t$, and hence
		$m=\lim_{t\to\infty} h(t)$ exists.
		By Lemma~\ref{l8}, it follows that $\lim_{t\to\infty} h(t)=l_*$,
		which contradicts the assumption $h_*< l^*$.
		
		If $h^*\geq l_*$, then our assumption $h_*=\liminf_{t\to\infty} h(t)< l_*$ implies the existence of some large $s>t^*$ such that $h(s)\in (0, l^c)$ and $h'(s)<0$, which contradicts our earlier conclusion that \eqref{3.71d} cannot occur. 
		
		Therefore $h_*<l_*$ always leads to a contradiction, as desired. It follows that $ h_*=\liminf_{t\to\infty} h(t)\geq l_*$,
		and the proof is now complete.	
	\end{proof}

	Summarising the conclusions on $h(t)$,  we have 
	
	\begin{proposition}\label{prop-h}
		Under the conditions of Theorem \ref{th1.2b}, we have the following conclusions:
		
		\begin{itemize}
			\item[\rm (i)] If the maximal existence time $T_{\max}<\infty$, then $\lim_{t\to T_{max}}h(t)=0$.
			\item[\rm (ii)] If $T_{max}=\infty$, then either $\lim_{t\to\infty}h(t)=\infty$, or $\lim_{t\to\infty} h(t)=l_*$, where $l_*$ is given by \eqref{l*}.
		\end{itemize}
		
	\end{proposition}
	
	\subsection{Behaviour of $u(t,x)$ when $T_{max}=\infty$.} Throughout this subsection, we assume $T_{max}=\infty$.
	By Lemma \ref{l8},
	\[
	\lim_{t \to \infty} h(t)=l_* \implies 	\lim_{t \to \infty} \sup_{x \in [0, h(t)]} \left| u(t, x) - w_*(x) \right| = 0,
	\]
	where  $w_*$ is given by \eqref{w}.
	
	Next we consider the case  $\lim_{t\to\infty} h(t)=\infty$ in Proposition \ref{prop-h} (ii), and prove the following conclusion.
	
	\begin{lemma}
		In Proposition \ref{prop-h} (ii),
		\[
		\lim_{t\to\infty} h(t)=\infty\implies \lim_{t \to \infty} u(t, x) = v(x)
		\	\text{locally uniformly in} \ [0,\infty),\]
		where $v$ is the unique solution of \eqref{1.3}.
	\end{lemma}
	
	\begin{proof}	
		Suppose $ h(t) \to \infty $ as $ t \to \infty $. Then for any given $ L \gg 1 $, there exists $ T_L \gg 1 $ such that
		\[
		h(t) > L \quad \text{for } t \geq T_L.
		\]
	
Define	\[
	w(x) = 
	\begin{cases} 
		0, & \text{if } x \leq h(0), \\[4pt]
		q_0\bigl(x - h(0)\bigr), & \text{if } x \in [h(0),\, h(0) + l_0], \\[4pt]
		\delta, & \text{if } x > h(0) + l_0,
	\end{cases}
	\]
	where $(q_0, l_0)$ is the unique solution of \eqref{qc1} with $c = 0$ and $\delta_0 = 0$. 	By the definition of $q_0$, $-dw_{xx}\leq f(w)$ for $x\in \mathbb{R}$ in the weak sense. The comparison principle then yields
	\begin{equation}\label{u>w}
	u(t,x)\geq w(x) \ \mbox{ for }\ t\geq0, \ x\in[0,h(t)].
	\end{equation}

		{\bf Step 1}: We show $\liminf_{t \to \infty} u(t, x) \geq v(x) \ \text{locally uniformly for } x \in [0, \infty)$.
		
		By Lemma \ref{le2.5}, we have $ u(t, x) > 0 $ for $ t \geq 0 $ and $ x \in (0, h(t)] $. Let $ V = V_L(t, x) $ be the unique positive solution of the following auxiliary problem:
		\begin{equation}\label{3.4}
			\begin{cases}
				V_t = d V_{xx} + f(V), & t > T_L, \, 0 < x < L, \\
				V(t, 0) = \delta_0, \, V(t, L) = \delta, & t > T_L, \\
				V(T_L, x) = w(x), & 0 < x < L.
			\end{cases}
		\end{equation}
		
		Since $ w(x) $ is a lower stationary solution of \eqref{3.4} and $ \bar{V} \equiv \max\{1,\delta_0\} $ is an upper stationary solution of \eqref{3.4}, it follows that $ V_L(t, x) $ is nondecreasing in $ t $ and
		\[
	w(x) \leq V_L(t, x) \leq\bar V\quad \text{for } t > T_L \text{ and } x \in [0, L].
		\]
		Thus,
		\[
		V_L^\infty(x) := \lim_{t \to \infty} V_L(t, x) \quad \text{exists and satisfies}\quad
		w(x) \leq V_L^\infty(x) \leq \bar V \quad \text{for } x \in [0, L] 	.\]
		Moreover, by standard parabolic regularity, this limit holds in the $ C^2([0, L]) $ norm, and $ V_L^\infty(x) $ satisfies
		\[
		-d (V_L^\infty)''(x) = f(V_L^\infty), \quad V_L^\infty(0) = \delta_0, \quad V_L^\infty(L) = \delta.
		\]
		
		We next show that $ V_L^\infty(x) \to v(x) $ as $ L \to \infty $. By the comparison principle, we easily see that  $V_L(t,x)\leq V_{\tilde L}(t,x)$ when $\tilde L>L$. It follows that  $ V_L^\infty $ is nondecreasing in $L$ and hence 
		\[
		V_\infty(x) := \lim_{L \to \infty} V_L^\infty(x) \quad \text{exists for every } x \in [0, \infty).
		··		\]
		Furthermore, by standard elliptic regularity, this limit holds in $ C_{\text{loc}}^2([0, \infty)) $, and $ V_\infty(x) $ satisfies
		\[
		-d V_\infty'' = f(V_\infty), \quad V_\infty(0) = \delta_0, \quad w(x) \leq V_\infty(x) \leq \bar V=\max\{1,\delta_0\} \quad \text{for } x \in [0, \infty).
		\]
		By a simple phase plane analysis, we see that $ V_\infty(x) = v(x) $, where $ v(x) $ is the unique solution of \eqref{1.3}. Moreover, 
		\[
		\begin{cases} \mbox{$v(x)$ is strict increasing in $x$ when $\delta_0<1$},\\
			\mbox{$v(x)$ is strict decreasing in $x$ when $\delta_0>1$},\\
			\mbox{$v(x)\equiv 1$  when $\delta_0=1$}.
		\end{cases}
		\]
		Since $h(t)>L\gg1$ when $t\geq T_L$, it follows from \eqref{u>w} that $u(t,L)\geq\delta$ for such $t$, and so
		by the  comparison principle, we  obtain
		\[
		V_L(t, x) \leq u(t, x) \quad \text{for } t > T_L \text{ and } x \in [0, L].
		\]
		It follows that
		\begin{equation}\label{3.6a}
			\liminf_{t \to \infty} u(t, x) \geq V_L^\infty(x) \quad \text{for } x \in [0, L].
		\end{equation}
		Since $ V_L^\infty(x) \to v(x) $ as $ L \to \infty $ locally uniformly in $ x \in [0, \infty) $, letting $ L \to \infty $ in \eqref{3.6a} yields
		\[
		\liminf_{t \to \infty} u(t, x) \geq v(x) \quad \text{locally uniformly for } x \in [0, \infty).
		\]

		{\bf Step 3}: We show that $ \limsup_{t \to \infty} u(t, x) \leq v(x) \ \text{locally uniformly for } x \in [0, \infty)$.
		
		Let $ M := 1+ \|u_0\|_{C[0, h_0]}>\max\{w(x),\delta_0\} $, and $ W = W_L(t, x) $ be the unique solution of the following auxiliary problem:
		\begin{equation}\label{3.3}
			\begin{cases}
				W_t = d W_{xx} + f(W), & t > T_L, \, 0 < x < L, \\
				W(t, 0) = \delta_0, \, W(t, L) = M, & t > T_L, \\
				W(T_L, x) = M, & 0 < x < L.
			\end{cases}
		\end{equation}

		Since $ w(x) $ is a lower stationary solution of \eqref{3.3} and $ \bar{W} \equiv M $ is an upper stationary solution of \eqref{3.3}, we see that $w(x) \leq W_L(t,x) \leq M$, $ W_L(t, x) $ is monotone nonincreasing in $ t $ and 
		\[
		\lim_{t \to \infty} W_L(t, x) = \tilde{V}_L(x) \quad \text{uniformly in } [0, L].
		\]
		Moreover, by standard parabolic regularity, this limit holds in the $ C^2([0, L]) $ norm, and $ \tilde{V}_L(x) $ satisfies
		\[
		-d \tilde{V}_L'' = f(\tilde{V}_L), \quad \tilde{V}_L(0) = \delta_0, \quad \tilde{V}_L(L) = M.
		\]
		By the monotonicity of $ W_L $ with respect to $ L $ (easily seen by an upper and lower solution argument), we conclude that
		\[
		\tilde{V}_\infty(x) := \lim_{L \to \infty} \tilde{V}_L(x) \quad \text{exists for every } x \in [0, \infty).
		\]
		Furthermore, by standard elliptic regularity, this limit holds in $ C_{\text{loc}}^2([0, \infty)) $, and $ \tilde{V}_\infty(x) $ satisfies
		\[
		-d \tilde{V}_\infty'' = f(\tilde{V}_\infty), \quad \tilde{V}_\infty(0) = \delta_0, \quad w(x) \leq \tilde{V}_\infty \leq M \quad \text{for } x \in [0, \infty).
		\]
		By a phase plane consideration we easily see  that $ \tilde{V}_\infty(x) = v(x) $, where $ v(x) $ is the  unique solution of \eqref{1.3}.
		
		By the  comparison principle, we have $ u(t, x) \leq W(t, x) $ in $ [T_L, \infty) \times [0, L] $. It follows that
		\[
		\limsup_{t \to \infty} u(t, x) \leq \tilde{V}_L(x) \mbox{ uniformly for } x\in [0,L].
		\]
		Letting $L\to\infty$ we obtain
		\[
		\limsup_{t \to \infty} u(t, x) \leq \tilde{V}_\infty(x) = v(x) \quad \text{for all } x \geq 0.
		\]
		This completes Step 3.
		
		Combining the results in Steps 1, 2 and 3, we obtain
		\[
		\lim_{t \to \infty} u(t, x) = v(x) \quad \text{locally uniformly for } x \in [0, \infty).
		\]
		The proof is complete.
	\end{proof}
	
	\subsection{Proof of Theorem \ref{th1.3}.}\label{sec:3.5}
	The conclusions in part (i) of Theorem \ref{th1.3} have already been proved in Lemma \ref{lemma2.10}. It remains to prove part (ii), which we restate as a theorem.
	
	\begin{theorem}\label{thbwd}
		Let $0 \leq \delta_0 < \delta < 1$ and $h_0 = l_*$. Then,
		\begin{enumerate}[label=(\roman*), topsep=0pt, itemsep=0pt, leftmargin=*]
			\item[\rm (i)] if $u_0(x) \equiv w_*(x)$, then \textsl{transition} occurs;
			\item[\rm (ii)] if $u_0(x)\not\equiv w_*(x)$ and $u_0(x) \geq w_*(x)$, then \textsl{spreading} occurs;
			\item[\rm (iii)] if $u_0(x)\not\equiv w_*(x)$ and $u_0(x) \leq w_*(x)$, then \textsl{vanishing} occurs.
		\end{enumerate}
	\end{theorem}
	
	\begin{proof}
		By the uniqueness of the solution to \eqref{1.1}, it is obvious that 
		\[
		u_0(x)\equiv w_*(x)\implies u(t,x)\equiv w_*(x),\ h(t)\equiv l_*.
		\]
		 Hence (i) holds.
		
		To prove (ii) and (iii), we first consider the  case 
		\begin{equation}\label{not0}
			u_0'(l_*)\not= 0.
		\end{equation}
		The case $u_0'(l_*)=0$ will be treated afterward, by an approximation argument using a sequence of initial functions $u_0^n$ satisfying \eqref{not0} and converging to $u_0$ as $n\to\infty$.	
		
		{\bf Step 1.} We prove part (ii) when \eqref{not0} holds.
		
		The assumptions imply $ u_0'(h_0) < 0 $.  The regularity of $u(t,x)$ and $h(t)$ in Theorem \ref{th2.1} imply that $h'(t) = -\frac{d}{\delta} u_x(t, h(t))>0$ for $ t \in [0, \epsilon) $  with some small $ \epsilon > 0 $. 
		
		We are going to show that $h'(t)>0$ for all $t\in (0, T_{max})$. Otherwise there exists $ t_0 \in [\epsilon, T_{max}) $ such that $ h'(t) > 0 $ for $ t \in (0, t_0) $ and $ h'(t_0) = 0 $.
		Define $ \eta(t, x) = {u}(t, x) - \tilde{w}(x) $, where
		\begin{equation}\label{ext}
			\tilde{w}(x) =
			\begin{cases}
				w_*(x), & x \in [0, l_*], \\
				\delta, & x \in [l_*, \infty).
			\end{cases}
		\end{equation}
		Then $\tilde w$ satisfies, in the weak sense,
		\[
		-d\tilde w''\leq f(\tilde w) \mbox{ for } x\geq 0.
		\]
		Hence
		\begin{equation}\label{3.12}
			\begin{cases}
				\eta_t - d \eta_{xx} \geq f(u) - f(\tilde{w}) = c (t,x)\eta, & 0 < t \leq t_0, \, 0 < x < h(t), \\
				\eta(t, 0) \equiv 0, \, \eta(t, h(t)) \equiv 0, & 0 < t \leq t_0, \\
				\eta(0, x) \geq, \not\equiv 0, & 0 \leq x \leq l_*,
			\end{cases}
		\end{equation}
		where $c(t,x)$ is some $L^\infty$ function.
		By the strong maximum principle, $ \eta(t, x) > 0 $ for $ 0 < t \leq t_0 $ and $ 0 < x < h(t) $. Applying the Hopf lemma at $ t = t_0 $ where
		$\eta(t_0, h(t_0))=0$, we find $ \eta_x(t_0, h(t_0)) < 0 $, which implies $u_x(t_0, h(t_0))= -\frac\delta d h'(t_0) > 0 $,  a contradiction. Thus we must have $h'(t)>0$ for all $t\in (0, T_{max})$. We may now use Lemma \ref{lemma2.10} to conclude that $T_{max}=\infty$, and hence
		$ \lim_{t \to \infty} h(t) = h_\infty>l_*$.  By Theorem \ref{th1.2b}, we necessarily have $ h_\infty = \infty $, and so $ \lim_{t \to \infty} u(t, x) = v(x) $ locally uniformly in $ [0, \infty) $.
		
		{\bf Step 2.} We prove part (iii) when \eqref{not0} holds.
		
		The assumptions imply $ u_0'(h_0) > 0 $.  The regularity of $u(t,x)$ and $h(t)$ in Theorem \ref{th2.1} then imply that $h'(t) = -\frac{d}{\delta} u_x(t, h(t))<0$ for $ t \in [0, \epsilon] $  with some small $ \epsilon > 0 $. Define
		\[
		T_0:=\sup\{T\in (0, T_{max}): h'(t)<0 \mbox{ for } t\in [0, T)\}.
		\]
		Then $T_0\geq\epsilon>0$. We claim  that 
		\begin{equation}\label{h'<0}
			T_0=T_{\max}.
		\end{equation}
		Note that by Proposition \ref{prop-h}, the equality \eqref{h'<0} and $h_0=l_*$ imply vanishing must occur. So to complete Step 2, it suffices to prove \eqref{h'<0}.
		
		Now suppose, for contradiction, that  $T_0 \in [\epsilon, T_{\max})$ and so
		\begin{equation}\label{T0}
			h'(t) < 0 \quad \text{for} \ t \in [0, T_0) \ \text{and} \ h'(T_0) = 0.
		\end{equation}
		We are going to show that \eqref{T0} leads to a contradiction,  in four substeps below.
		
		{\bf Step 2.1.} We construct two auxiliary functions  $\tilde f(u)$ and $\bar u(x)$ depending on some $T_1\in (0, T_0]$, whose value will be specified later.
		
		Fix $Q\in [\delta,1)$ and define
		\begin{equation*}\label{l_Q}
			\tilde l_Q:= \int_0^{Q} \frac{dx}{\sqrt{\frac{2}{d}\int_{x}^{\delta} f(s)}ds}  \in (0, \infty).
		\end{equation*}
		As in the proof of Lemma \ref{lemma3.4}, the identity
		\begin{equation*}
			\int_{w_Q(x)}^Q \frac{du}{\sqrt{\frac{2}{d}\int_{u}^{\delta} f(s) \, ds}}du =\tilde l_Q-x \mbox{ for } x\in [0, \tilde l_Q],
		\end{equation*}
		uniquely defines a function $w_Q(x)$, and $(l, w(x))=(\tilde l_Q, w_Q)$ satisfies
		\begin{equation*}\label{lQ}
			\begin{cases}
				d w'' + f(w) = 0 \mbox{ for } x \in (0, l), \\
				w'(x) > 0 \mbox{ for }  x \in [0, l),\\
				w(0) = {0}, \quad w(l) = Q, \quad w'(l) = 0.
			\end{cases}
		\end{equation*}
		As $0\leq\delta_0<\delta\leq Q$, there exist $0\leq \tilde l_Q(\delta_0)<\tilde l_Q(\delta)\leq \tilde l_Q$ such that
		\[
		w_Q(\tilde l_Q(\delta_0))=\delta_0,\ w_Q(\tilde l_Q(\delta))=\delta.
		\]
		Define
		\[
		l_Q:=\tilde l_Q-\tilde l_Q(\delta_0),\ l_Q^\delta:=\tilde l_Q(\delta)-\tilde l_Q(\delta_0),\ q_Q(x):=w_Q(x+\tilde l_Q(\delta_0)).\]
		Then it is easily seen that $(l_Q, q_Q(x))$ satisfies	
		\begin{equation}\label{q_Q}
			\left\{
			\begin{array}{ll}
				d  q_Q'' + f(q_Q) = 0, & x \in (0, l_Q), \\[2mm]
				q_Q(0) = \delta_0, \quad q_Q(l^\delta_Q) = \delta, \quad q_Q(l_Q) = Q, \quad q_Q'(l_Q) = 0, \\[2mm]
				q_Q'(x) > 0, & x \in [0, l_Q).
			\end{array}
			\right.
		\end{equation}
		Since $l^\delta_Q$ depends continuously on $Q$ and $\delta$, and $l_* = l^\delta_\delta =h(0)> h(t)$ for $t\in (0, T_0]$, for any given $T_1\in (0, T_0]$ we can select $Q=Q(T_1) \in (\delta, 1)$ such that
		\[
		l^\delta_Q > h(T_1). \]
		
		With $Q\in (\delta, 1)$ chosen as above (depending on $T_1$), we can now find a function  \( \tilde{f} \in C^1([0, Q]) \) such that \( \tilde{f}(0) = \tilde{f}(Q) = 0 \) and
		\begin{align*}
			\tilde{f}(s) = f(s)  \quad \text{for } s \in \left[0, \delta \right], \qquad 	0 < \tilde{f}(s) \leq f(s)  \quad \text{for } s \in (0, Q).
		\end{align*}
		
		Next we define $\bar u(x)$ for $x \in [0, \infty)$ by
		\begin{equation*}
			\bar{u}(x) := 
			\begin{cases}
				q_Q(x+l_Q^\delta-h(T_1)), & x \in [0, h(T_1)-l^\delta_Q+l_Q],\\
				Q, & x \in [h(T_1)-l^\delta_Q+l_Q, +\infty),
			\end{cases}
		\end{equation*}		
		Then by \eqref{q_Q} and the above definition, it is easily seen that $\bar u(x)$ satisfies, in the weak sense,
		\begin{equation}\label{bar-u}
			d\bar u_{xx}+ \tilde f(\bar u) \leq 0 \mbox{ for } x>0.
		\end{equation}
		Moreover,	since $l_Q^\delta>h(T_1)$ and $h(t)\geq h(T_1)	$ for $0\leq t\leq T_1$, we further have
		\begin{align}\label{bdy}
			\begin{cases}
				\bar{u}(0)=q_Q(l_Q^\delta-h(T_1))>q_Q(0)=\delta_0=u(t,0),\\
				\bar{u}(h(t))\geq \bar{u}(h(T_1))=q_Q(l_Q^\delta)=\delta=u(t, h(t)).
			\end{cases}
		\end{align}

		{\bf Step 2.2.} 	
		We prove that 
		\begin{equation}\label{t=0}
			u_0(x)\leq \bar u(x) \mbox{ for } x\in [0, l_*]=[0, h(0)].
		\end{equation}
		
		As we already know by assumption that $u_0(x)\leq w_*(x)=q_\delta(x)$ for $x\in [0, l_*]=[0, h(0)]$, to prove \eqref{t=0}, it suffices to show
		\begin{equation}\label{t=0a}
			q_\delta(x)\leq \bar u(x) \mbox{ for } x\in [0, l_*].
		\end{equation}
		Define $p_Q(x):=q'_Q(x)$.
		Since \( q_Q'(x) > 0 \) for all \( x \in [0, l_Q^\delta] \), the trajectory $(q_Q(x), p_Q(x))$ with $x\in [0, l_Q^\delta]$ in the $q$-$p$ plane can be represented as a function \( p=P_Q(q) \) with \( q \in [\delta_0, \delta] \), and
		\begin{equation}\label{P_Q}
			P_Q'(q) =  - \frac{f(q)}{d P_Q(q)} \mbox{ for } q\in [\delta_0,\delta],
			\quad  P_Q(\delta) =p_Q(l_Q^\delta)> 0.
		\end{equation}
		Analogously, with $p_\delta(x):=q'_\delta(x)$,
		the trajectory $(q_\delta(x), p_\delta(x))$ with $x\in [0, l_*)$ in the $q$-$p$ plane can be represented as a function \( p=P_\delta(q) \) with \( q \in [\delta_0, \delta) \), and
		\begin{equation}\label{P_delta}
			P_\delta'(q) =  - \frac{f(q)}{d P_\delta(q)} \mbox{ for } q\in [\delta_0,\delta),
			\quad  \lim_{q\to \delta}P_\delta(q) = 0.
		\end{equation}
		Hence $P_Q(q)>P_\delta(q)$ for $q<\delta$ but close to $\delta$. By uniqueness of the solution to the initial value problem of ODEs, we see that the solution
		curve $p= P_Q(q)$ of \eqref{P_Q} stays above that of \eqref{P_delta} when $q$ is decreased from $\delta$ to $\delta_0$, i.e.,
		\[
		P_Q(q)>P_\delta(q) \mbox{ for } q\in [\delta_0, \delta).
		\]
		From $q_Q'(x)=p_Q(x)=P_Q(q_Q(x))$ for $x\in [0, l_Q^\delta]$ we deduce
		\[
		l^\delta_Q=\int_0^{l^\delta_Q}\frac {q_Q'(x)}{P_Q(q_Q(x))	}dx=\int_{\delta_0}^\delta\frac{dq}{P_Q(q)}.
		\]
		Since
		\[
		\int_{\delta_0}^\delta\frac{dq}{P_Q(q)}<\int_{\delta_0}^\delta\frac{dq}{P_\delta(q)}=l_*,
		\]
		we obtain $l^\delta_Q<l_*$.
		
		Since $q_Q(0)=q_\delta(0)=\delta_0$ and $q_Q'(0)=P_Q(\delta_0)>P_\delta(\delta_0)=q'_\delta(0)$, there exists $\beta > 0$ small such that  
		\[q_\delta(x) < q_Q(x) \quad \text{for} \quad x \in (0, \beta).\]  
		We claim that the above inequality holds for $x\in (0, l_Q^\delta)$. Otherwise
		there exists some $x_0 \in (0, l_Q^\delta]$ satisfying  
		\[q_\delta(x) < q_Q(x) \quad \text{for} \quad x \in (0, x_0) \quad \text{and} \quad q_\delta(x_0) = q_Q(x_0).\]  
		It follows that $q_Q'(x_0)\leq q'_\delta(x_0)$. On the other hand, we have
		\[
		q_Q'(x_0)=P_Q(q_Q(x_0))>P_\delta(q_Q(x_0))=P_\delta(q_\delta(x_0))=q_\delta'(x_0).
		\]
		This contradiction proves our claim that 
		\[q_\delta(x) < q_Q(x) \quad \text{for} \quad x \in (0, l_Q^\delta).\] 
		Since $l_Q^\delta>h(T_1)$, from the definition of $\bar u(x)$ we see that $\bar u(x)\geq q_Q(x)$ for $x\in [0, l_Q^\delta)$ and
		hence
		\[
		\bar u(x)\geq q_Q(x)>q_\delta(x) \ \text{for} \ x \in (0, l_Q^\delta).\]
		For $ x\in [l_Q^\delta, l_*]$, $\bar u(x)=Q>\delta\geq q_\delta(x)$. Thus \eqref{t=0a} holds, and Step 2.2 is completed.

		{\bf Step 2.3.}
		We show that
		\begin{equation}\label{T00}
			u(t,x) < \delta \quad \text{for} \ t \in [0, T_0] \ \text{and} \ x \in [0, h(t)).
		\end{equation}

		Since $u_0'(l_*)>0=w_*'(l_*)$, $u_0(x) \leq w_*(x)<\delta$ for $x\in [0, l_*)$, and $u_0(l_*)=w_*(l_*)=\delta$,  by	
		continuity, there exists $\epsilon_0 > 0$ such that 
		\begin{align}\label{3.66a}
			u(t,x) < \delta \quad \text{for} \ t \in [0, \epsilon_0] \ \text{and} \ x \in [0, h(t)).
		\end{align}
		
		Now assume, contrary to \eqref{T00}, that there exists $T_1 \in (\epsilon_0, T_0]$ with
		\[
		u(t,x) < \delta \quad \text{for} \ t \in [0, T_1) \ \text{and} \ x \in [0, h(t)),
		\]
		yet $u(T_1, x_0) = \delta$ for some $x_0 \in (0, h(T_1))$.
		
		We are going to show that this statement leads to a contradiction. 
		Since \( 0 < u(t, x) \leq \delta \) holds for \( t \in [0, T_1] \) and \( x \in [0, h(t)] \), it follows that
		\begin{equation}\label{u-tilde-f}
			u_t = d u_{xx} + \tilde{f}(u)  \ \mbox{ for } t\in (0, T_1],\ x \in (0, h(t)).
		\end{equation}
		
		We now consider $\bar u(x)$ defined in Step 2.1 with the above $T_1$, which satisfies \eqref{bar-u}, \eqref{bdy} and \eqref{t=0}.				
		Due to \eqref{u-tilde-f},   we can compare $u(t,x)$ with $\bar u(x)$ for $t\in [0, T_1]$ and $x\in [0, h(t)]$ to conclude, by the comparison principle, that
		\[
		\bar{u}(x)> u(t, x)\quad\text{for } t\in[0, T_1]\quad\text{and}\quad x\in[0, h(t)).
		\]		
		In particular,  
		\[\delta=\bar{u}(h(T_1))>\bar{u}(x_0)\geq u(x_0)=\delta.
		\]
		This contradiction shows that \eqref{T00} holds, and Step 2.3 is completed.
		
		\textbf{Step 2.4.} We derive the desired contradiction.
		
		Since \eqref{T00} holds, we may take $T_1=T_0$ in the definitions of $\bar u(x)$ and $\tilde f$, and
		repeat the argument in Step 2.3 above with $T_1$ replaced by $T_0$, to compare $u(t,x)$ with $\bar u(x)$ for $t\in [0, T_0]$ and $x\in [0, h(t)]$, which gives 
		\[\bar{u}(x) > u(t,x) \quad \text{for} \ t \in [0, T_0] \ \text{and} \ x \in [0, h(t)).\]   
		Since $\bar{u}(h(T_0)) = u(T_0,h(T_0)) = \delta$, the Hopf boundary lemma gives  
		\[0 < q_Q'(l_Q^\delta) = \bar{u}_x(h(T_0)) < u_x(T_0,h(T_0)).\]  
		Consequently,  
		\[h'(T_0) = -\frac{d}{\delta} u_x(T_0,h(T_0)) < 0,\]  
		contradicting $h'(T_0)=0$. This completes Step 2.

		\textbf{Step 3.} We prove (ii) and (iii) for the case $u_0'(l_*)=0$.
		
		Since $h'(0)=-\frac{d}{\delta}u_0'(l_*)=0$, there exists $t_0>0$ such that $h(t)>0$ for $t\in[0, t_0]$, which implies $T_{\text{max}}>t_0$.  
		For any sequence $\{u_0^n\}\subset \mathcal X(l_*)$ satisfying  
		\[
		\lim_{n \to \infty}u_0^n(x)=u_0(x) \quad\text{in } C^2([0,l_*]),
		\]  
		let $(u_n(t,x), h_n(t))$ be the unique solution of \eqref{1.1} with initial function $u_0^n(x)$, and maximal existence time $ T_{max}^n$. 
		By the continuous dependence of the solution of \eqref{1.1} on its initial data (which follows easily from  the uniqueness of the solution), we know that   $T^n_{\text{max}}\to T_{max}$ as $n\to\infty$ and for any $T\in (0, T_{max})$,
		\[\begin{cases}
			\lim_{n \to \infty}h_n(t)=h(t) \quad\text{in } C^1([0, T]),\\
			\lim_{n \to \infty} u_n(t,x h_n(t))=u(t,x h(t)) \quad\text{ in } C([0, T]\times [0, 1]).
		\end{cases}
		\]

		\textbf{Case 1:}
		$u_0(x) \geq, \not\equiv w_*(x)$ and $u_0'(l_*) = 0$. 
		
		Choose a sequence $\{u_0^n\}\subset \mathcal{X}(l_*)$ such that 
		\[
		u_0^n(x)\geq  w_*(x),\quad (u_0^n)'(l_*) < 0, \quad  \lim_{n \to \infty} u_0^n = u_0 \ \text{in } C^2([0,l_*]).
		\]
		By Step 1 we know that for each $n\geq 1$,
		$T^n_{\mathrm{max}} =\infty$, $h_n'(t)>0$ for $t>0$ and $h_n(t)\to\infty$ as $t\to\infty$.
		It follows that for any $t\in (0, T_{max})$, $h(t) = \lim_{n \to \infty} h_n(t)\geq l_*$ and $h'(t)=\lim_{n \to \infty} h'_n(t)\geq 0$.
		Therefore either (i) $h(t)>l_*$ for all $t\in (0, T_{max})$, or (ii) $h(t)\equiv l_*$ for $t\in [0,t_1]$ with some $t_1\in (0, T_{max})$.
		
		If alternative (ii) happens, 
		then $\tilde{u} := u - w$ satisfies, for some $L^\infty$ function $c(t,x)$,
		\begin{equation*}
			\begin{cases}
				\tilde{u}_{t} - d \tilde{u}_{xx} - c(t,x)\tilde{u} = 0, & 0 < t \leq t_1,\ 0 < x <  l_*, \\
				\tilde{u}(t, 0) = \tilde{u}(t, l_*) = 0, & 0 < t \leq t_1, \\
				\tilde{u}(0, x) \geq, \not\equiv 0, & 0 \leq x \leq l_*.
			\end{cases}
		\end{equation*}
		The strong maximum principle and the Hopf boundary lemma then yield
		\[\begin{cases}
			\tilde u(t, x)>0 &\mbox{ for } t\in (0, t_1],\ x\in (0, l_*),\\
			\tilde u_x(t, l_*)<0 &\mbox{ for } t\in (0, t_1].
		\end{cases}
		\]
		It follows that
		\[
		u_x(t,h(t))=u_x(t, l_*)=\tilde u_x(t, l_*)-w'(l_*)=\tilde u_x(t, l_*)<0 \mbox{ for } t\in (0, t_1].
		\]
		The free boundary condition in \eqref{1.1} then gives $h'(t)=-\frac d\delta u_x(t, h(t))>0$ for $t\in (0, t_1]$, a contradiction to $h(t)\equiv l_*$.
		
		Therefore only (i) can happen, which implies, by Proposition \ref{prop-h}, that $T_{max}=\infty$. We may now use $h'(t)\geq 0$ to see that $\lim_{t\to\infty} h(t)>l_*$, and therefore spreading happens.
		
		\textbf{Case 2:}
		$u_0(x) \leq, \not\equiv w_*(x)$ and $u_0'(l_*) = 0$. 
		
		Choose $\{u_0^n \}\subset \mathcal{X}(l_*)$ such that 
		\[
		u_0^n(x)\leq  w_*(x), \quad (u_0^n)'(l_*) > 0, \quad  \lim_{n \to \infty} u_0^n = u_0 \ \text{in } C^2([0,l_*]) .
		\]
		Fix $t_0\in (0, T_{max})$. Then
		$T^n_{\mathrm{max}} > t_0$ for all large $n$, and by Step 2, 
		\[
		\mbox{$h_n'(t)<0,\ u_n(t,x)<\delta$ for $t\in [0, T^n_{max}),\ x\in [0, h_n(t)).  $}
		\]
		Since we further have		
		\[
		\lim_{n \to \infty} h_n = h \ \text{in } C^1([0, t_0]),\ \mbox{ and } \lim_{n \to \infty} u_n(t, x h_n(t)) = u(t, x h(t)) \ \text{uniformly in } [0, t_0] \times [0, 1],
		\]
		it follows that
		\[
		h'(t) \leq 0 \quad \text{and} \quad u(t,x) \leq \delta \ \text{for } t \in [0, t_0] \text{ and } x \in [0,h(t)].
		\]
		
		We next show that 
		\begin{align}\label{3.73}
			h'(t)<0,\qquad u(t,x)<\delta \mbox{ for } t\in (0,t_0],\ x\in [0,h(t)).
		\end{align}
		In fact, for any given  $T_1 \in  (0,t_0]$, as in Step 2.3 above, we can compare $u(t,x)$ with  the function
		\[
		\bar{u}(x) =
		\begin{cases}
			q_Q(x + l_Q^\delta - h(T_1)), & x \in [0,\, h(T_1) - l_Q^\delta + l_Q],\\[2mm]
			Q, & x \in [h(T_1) - l_Q^\delta + l_Q,\, +\infty),
		\end{cases}
		\]
		for $t\in [0, T_1]$ and $x\in [0, h(t)]$, to deduce 
		\[
		h'(T_1)<0,\qquad u(T_1,x)<\delta \mbox{ for }  x\in [0,h(T_1)).
		\]
		Since $T_1\in (0, t_0]$ is arbitrary, this implies \eqref{3.73}. 
		
		Next, we show that $h'(t)<0$ for all $t\in[0,T_{\max})$. Arguing indirectly,  we assume  by contradiction
		that there exists $T_0 \in [0,T_{\max})$ such that  
		\[
		h'(t)<0 \quad \text{for } t\in[0,T_0), 
		\qquad h'(T_0)=0.
		\]
		We may then repeat the above reasoning with $T_1$ replaced by $T_0$ to deduce
		\[
		h'(T_0)<0,\qquad u(T_0,x)<\delta \mbox{ for }  x\in [0,h(T_0)).
		\]
		This contradiction proves 
		$h'(t)<0$ for all $t\in[0,T_{\max})$, which implies, by Theorem \ref{th1.2b}, that vanishing  happens. This completes Step 3, and the theorem is now fully proved.
	\end{proof}
	
	\section{Proofs of several basic results}
	
	\subsection{Proof of Theorem \ref{th2.1}} We complete the proof in four steps. 
		In Step 1, we straighten the free boundary, transforming the problem into an initial-boundary value problem with fixed boundaries. In Steps 2 and 3, the local existence and uniqueness of the solution are established via a contraction mapping argument combined with an extension technique. Finally, in Step 4, we employ Schauder theory to obtain additional smoothness of the solution away from $t=0$.

		\noindent	$\textbf{Step\,1}.$ We straighten the free boundary $x=h(t)$.

		Consider the transformation $(t, y) \mapsto (t, x)$ defined by
		\[
		x = \Psi(t, y) := y h(t), \quad y \in \mathbb{R}.
		\]
		For each fixed $t \geq 0$, this transformation constitutes a diffeomorphism from $\mathbb{R}$ to $\mathbb{R}$ provided $|h(t) - h_{0}| < \frac{h_{0}}{2}$. This condition holds for sufficiently small $t > 0$, specifically $t \in [0, T]$ with $T>0$ small,  provided that $h(t)$ is continuous. Furthermore, the inverse transformation maps the curve $x = h(t)$ to the line $y = h_{0}$.
		
		Define $U(t, y) := u(t, x)$ for $(t, x) \in \Omega_{T}$. Then for $t \in (0, T]$, system \eqref{1.1} is equivalent to
		the following initial-boundary value problem with fixed boundaries:
		\begin{equation}\label{2.1}
			\left\{
			\begin{array}{ll}
				U_{t} - \frac d{h^2(t)} U_{yy} - \frac{h'(t)}{h(t)} y U_{y} = f(U), & 0 < t \leq T, \, 0 < y < 1, \\
				U(t, 0) = \delta_{0}, \quad U(t, 1) = \delta, & 0 < t \leq T, \\
				h'(t) = -\frac{d}{\delta h(t)} U_{y}(t, 1), & 0 < t \leq T, \\
				U(0, y) = u_{0}\left(h_{0} y\right), & 0 < y < 1.
			\end{array}
			\right.
		\end{equation}
		
			\medskip
		
		\noindent	$\textbf{Step\,2}.$ An extension trick for the application of  $L^{p}$	theory and Sobolev embedding theorems.  
	\smallskip	
		
Set
		\[
		H := \max \{1, |h^{0}|\}, \quad 
		K := \max \left\{1, \| u_{0} \|_{C^{2}[0, h_{0}]} \right\}, \quad 
		T_{1} := \frac{h_{0}}{2(H + |h^{0}|)} 
		\mbox{
		with $h^{0} := -\frac{d}{\delta} \frac{u_{0}'(h_{0}) }{ h_{0}}$. }
		\]
		For $T \in (0, T_{1}]$ and the domain $D_T := [0, T] \times [0, 1]$, define the function spaces
		\[
		\begin{aligned}
			X_{1,T} &:= \left\{ U \in C(D_T) : U(0, y) = u_0(h_0 y),\ \| U - U(0, \cdot) \|_{C(D_T)} \leq K \right\}, \\
			X_{2,T} &:= \left\{ h \in C^{0,1}([0,T]) :  h(0) = h_0,\  h'(0) = h^{0}, \ \| h' - h^{0} \|_{L^{\infty}([0,T])} \leq H  \right\}.
		\end{aligned}
		\]
		The product space $X_T := X_{1,T} \times X_{2,T}$ forms a complete metric space under the metric
		\[
		d\left( (U_1, h_1), (U_2, h_2) \right) := \| U_1 - U_2 \|_{C(D_T)} + \| h'_1 - h'_2 \|_{L^{\infty}([0,T])}.
		\]
		For $0 < T < T_1$, define the subspace $X_{T_1}^T := X_{1,T_1}^T \times X_{2,T_1}^T \subset X_{T_1}$ by trivial extension:
		\[
		\begin{aligned}
			X_{1,T_1}^T &:= \left\{ U \in X_{1,T_1} : U(t,y) = U(T,y) \quad\text{for all}\quad t \in [T, T_1], \, y \in [0,1] \right\}, \\
			X_{2,T_1}^T &:= \left\{ h \in X_{2,T_1} : h(t) = h(T) \quad\text{for all}\quad t \in [T, T_1] \right\}.
		\end{aligned}
		\]
		Clearly any $(U, h) \in X_T$ admits a trivial extension to $X_{T_1}^T$ as above. For convenience, we will always identify $X_T$ with $X_{T_1}^T$. Moreover, we always assume that $f(u)=0$ for $u<0$.

		For each pair $(U, h) \in X_T = X_{T_1}^T \subseteq X_{T_1}$, 
		define
		\[
		\rho(t): = d/ h^{2}(t),\ \xi(t) =: h^{\prime}(t) / h(t), 
		\]
		and consider the initial-boundary value problem 
		\begin{equation}\label{2.2}
			\left\{
			\begin{array}{ll}
				\bar{U}_{t} - \rho(t) \bar{U}_{yy} - \xi(t) y \bar{U}_{y} = f(U), & 0 < t \leq T_1, \, 0 < y < 1, \\
				\bar{U}(t,0) = \delta_0, \, \bar{U}(t,1) = \delta, & 0 < t \leq T_1, \\
				\bar{U}(0,y) = u_0(h_0 y), & 0 < y < 1.
			\end{array}
			\right.
		\end{equation}
		Direct calculations yield 
		\begin{equation}\label{2.3}
			4d/(9h_0^2) \leq \rho(t) = d / h^{2}(t) \leq 4d/h_0^2.
		\end{equation}
		For $Y_1 = (t_1, y_1)$ and $Y_2 = (t_2, y_2)$ in $D_{T_1}$ with parabolic distance $\delta(Y_1, Y_2) = \sqrt{(y_1-y_2)^2 + |t_1 - t_2|}$, we have
		\begin{equation}\label{2.4}
			\begin{aligned}
				\omega(R) &:= d \sup_{\delta(Y_1,Y_2) \leq R} \left| \frac{1}{h^2(t_1)} - \frac{1}{h^2(t_2)} \right| \\
				&\leq d \sup_{\delta(Y_1,Y_2) \leq R} \frac{|h^2(t_2) - h^2(t_1)|}{h^2(t_1)h^2(t_2)} \\
				&\leq \frac{48d(|h^0| + H)}{h_0^3} R^2 \to 0 \quad \text{as } R \to 0.
			\end{aligned}
		\end{equation}
		Additionally, 
		\begin{equation}\label{2.5}
			|\xi(t) y| \leq |h'(t) / h(t)| \leq 2(|h^0| + H)/h_0.
		\end{equation}
		In light of \eqref{2.3}-\eqref{2.5} and the assumption $u_0 \in C^2([0, h_0])$ with $u_0(0)=\delta_0$ and $u_0(h_0)=\delta$, we can apply the $L^p$ theory to \eqref{2.2} and the Sobolev embedding theorem  to conclude that for $\alpha \in (0,1)$, \eqref{2.2} admits a unique solution $\bar{U}$ satisfying
		\begin{equation}\label{2.6}
			\| \bar{U} \|_{C^{\frac{1+\alpha}{2}, 1+\alpha}(D_{T_1})} \leq C_{T_1} \| \bar{U} \|_{W_p^{2,1}(D_{T_1})} \leq K_1,
		\end{equation}
		where $p > 3/(2-\alpha)$, and $K_1$ depends on $p$, $\| f(U) \|_{L^p(D_{T_1})}$, $\| u_0 \|_{C^2([0, h_0])}$, $h_0$, $h^0$, and $C_{T_1}$. The constant $C_{T_1}$ depends on $D_{T_1}$ and $\alpha$.\footnote{For the analysis in Step 3, it is important  that $D_{T_1}$ is independent of $T$, which is a consequence of the extension trick.}
		
		Define $\bar{h}(t)$ via 
		\begin{equation*}
			\bar{h}'(t) = -\frac{d}{\delta} \frac{\bar{U}_y(t, 1)}{h(t)}, \quad \bar{h}(0) = h_0.
		\end{equation*}
		Then $\bar{h}' \in C^{\alpha/2}([0, T_1])$ with
		\begin{equation}\label{2.7}
			\| \bar{h}' \|_{C^{\alpha/2}([0, T_1])} \leq K_2,
		\end{equation}
		where $K_2$ depends on $K_1$.
		Define the mapping $ \mathcal{F}: X_{T}=X_{T_{1}}^{T} \rightarrow C\left(D_{T_{1}}\right) \times C([0, T_{1}])$  by
		\[\mathcal{F}(U, h)=(\bar{U}, \bar{h}).\]
		Set
		\[\tilde{\mathcal{F}}(U, h):=\left.\mathcal{F}(U, h)\right|_{X_{T}}.\]
		Then we see that  $(U, h)$  is a fixed point of $ \tilde{\mathcal{F}}$  if and only if it solves \eqref{2.1}, which is equivalent to \eqref{1.1} for  $t \in[0, T]$.
		
		\medskip
		
		\noindent	$\textbf{Step\,3}.$ We show that $ \tilde{\mathcal{F}} $ is a contraction mapping for small enough  $T>0$.
		
		Given any fixed $ 0<T<\min \left\{T_{1},\,{(\frac{K_{1}}{K})}^{\frac{-2}{1+\alpha}},\, (\frac{K_{2}}{H})^{\frac{-2}{\alpha}}\right\}$, we have
		\[\begin{array}{l}
			\left\Arrowvert\bar{U}-u_{0}\right\Arrowvert_{C\left(D_{T}\right)} \leq T^{\frac{1+\alpha}{2}}\Arrowvert\bar{U}\Arrowvert_{C^{0, \frac{1+\alpha}{2}}\left(D_{T}\right)} \leq T^{\frac{1+\alpha}{2}}\Arrowvert\bar{U}\Arrowvert_{C^{0, \frac{1+\alpha}{2}}\left(D_{T_{1}}\right)} \leq K_{1} T^{\frac{1+\alpha}{2}} \leq K, \\
			\left\Arrowvert\bar{h}^{\prime}(t)-h^{0}\right\Arrowvert_{L^{\infty}([0, T])} \leq T^{\frac{\alpha}{2}}\left\Arrowvert\bar{h}^{\prime}\right\Arrowvert_{C^{\frac{\alpha}{2}}([0, T])} \leq T^{\frac{\alpha}{2}}\left\Arrowvert\bar{h}^{\prime}\right\Arrowvert_{C^{\frac{\alpha}{2}}\left(\left[0, T_{1}\right]\right)} \leq K_{2} T^{\frac{\alpha}{2}} \leq H, \\
		\end{array}\]
		which implies that  $\tilde{\mathcal{F}} $ maps $ X_{T} $ to itself.
		
		Next, we prove that $\tilde{\mathcal{F}}$ is a contraction mapping on $X_T$ for all sufficiently small $T > 0$. Let $(U_i, h_i) \in X_T = X_{T_1}^T$ for $i = 1, 2$, and define
		\[
		W := \bar{U}_1 - \bar{U}_2.
		\]
		Then, it is easy to verify that $W$ satisfies
		\begin{equation}\label{2.8}
			\left\{
			\begin{array}{ll}
				W_t - \rho_1(t) W_{yy} - \xi_1(t) y W_y = \Psi, & 0 < y < 1, \, 0 < t \leq T_1, \\
				W(t, 1) = W(t, 0) = 0, & 0 < t \leq T_1, \\
				W(0, y) = 0, & 0 \leq y \leq 1,
			\end{array}
			\right.
		\end{equation}
		where
		\[\begin{cases}\rho_i(t): = d /h_i^{2}(t), \ \xi_i(t) := h_i'(t) / h_i(t) \mbox{ for } i = 1, 2,\\[2mm]
			\Psi :=   \left[ \xi_1(t)  - \xi_2(t)  \right]y \bar{U}_{2,y} 
			+ \left( \rho_1(t) - \rho_2(t) \right) \bar{U}_{2,yy} + f(U_1) - f(U_2).
		\end{cases}
		\]

		Direct calculation gives
		\begin{equation}\label{2.9}
			\begin{aligned}
				\left\Vert [\xi_{1}(t) - \xi_{2}(t)]y \right\Vert_{C(D_{T_1})} 
				&= \sup_{(t, y) \in D_{T_1}} \left| \frac{h_1'(t)}{h_1(t)} - \frac{h_2'(t)}{h_2(t)} \right| \\
				&\leq S_1 \left\Vert h_1 - h_2 \right\Vert_{C^{0,1}([0, T])}
			\end{aligned}
		\end{equation}
		for some constant $S_1$ depending on $h_0$, $ h^0$ and $T_1$, but is independent of $T$.  Similarly,
		\begin{equation}\label{2.10}
			\begin{aligned}
				\left\Vert \rho_1(t) - \rho_2(t) \right\Vert_{C(D_{T_1})} 
				&= \sup_{(t, y) \in D_{T_1}} \left| \frac{d}{h_1^2(t)} - \frac{d}{h_2^2(t)} \right| \\
				&\leq S_2 \left\Vert h_1 - h_2 \right\Vert_{C^{0,1}([0, T])}
			\end{aligned}
		\end{equation}
		where the positive constant $S_2$ depends on $h_0$,  $d$ and $T_1$, but is independent of $T$.
		Combining \eqref{2.9} and \eqref{2.10}, for any $p > 1$ we obtain
		\begin{equation}\label{2.11}
			\begin{aligned}
				\left\Vert \Psi \right\Vert_{L^p(D_{T_1})} 
				&\leq \left\Vert \xi_1 y - \xi_2 y \right\Vert_{L^\infty(D_{T_1})} \left\Vert \bar{U}_{2,y} \right\Vert_{L^p(D_{T_1})} \\
				&\quad + \left\Vert \rho_1 - \rho_2 \right\Vert_{L^\infty(D_{T_1})} \left\Vert \bar{U}_{2,yy} \right\Vert_{L^p(D_{T_1})} \\
				&\quad + \left\Vert f(U_1) - f(U_2) \right\Vert_{L^p(D_{T_1})} \\
				&\leq S_3 \left( \left\Vert U_1 - U_2 \right\Vert_{C(D_{T_1})} + \left\Vert h_1 - h_2 \right\Vert_{C^{0,1}([0, T_1])} \right),
			\end{aligned}
		\end{equation}
		where $S_3$ depends only on $D_{T_1}$, $K_1$, $f$, $S_1$, and $S_2$.
		
		Applying $L^p$ estimates to \eqref{2.8} and the Sobolev embedding theorem we then obtain
		\begin{equation}\label{2.12}
			\begin{aligned}
				\left\Vert W \right\Vert_{C^{\frac{1+\alpha}{2},1+\alpha}(D_{T_1})} 
				&\leq C_{T_1} \left\Vert W \right\Vert_{W_p^{1,2}(D_{T_1})}\leq K_3 \left\Vert \Psi \right\Vert_{L^p(D_{T_1})} \\
				&\leq K_3 \left( \left\Vert U_1 - U_2 \right\Vert_{C(D_{T_1})} + \left\Vert h_1 - h_2 \right\Vert_{C^{0,1}([0,T_1])} \right),
			\end{aligned}
		\end{equation}
		with $K_3$ depending on $p$, $D_{T_1}$, $S_3$ and $C_{T_1}$. Since
		\[
		\bar h_1'(t)-\bar h_2'(t)=\frac d\delta\left[\frac{\bar U_{1,y}(t,1)}{h_1(t)}-\frac {\bar U_{2,y}(t,1)}{h_2(t)}	\right],
			\]
			it is easy to show that
		\begin{equation}\label{2.13}
			\left\Vert \bar{h}_1' - \bar{h}_2' \right\Vert_{C^{\alpha/2}([0,T_1])} \leq \tilde C_{T_1}\left( \left\Vert \bar{U}_{1,y} - \bar{U}_{2,y}\right\Vert_{C^{0,\alpha/2}(D_{T_1})}+\left\Vert h_1 - h_2 \right\Vert_{C^{0,1}([0,T_1])} \right)
		\end{equation}
for some constant $\tilde C_{T_1}$ independent of $T$.		
		The inequalities \eqref{2.12} and \eqref{2.13} imply
		\begin{align*}
			\left\Vert \bar{U}_1 - \bar{U}_2 \right\Vert_{C^{\frac{1+\alpha}{2},1+\alpha}(D_{T_1})} 
			+ \left\Vert \bar{h}_1' - \bar{h}_2' \right\Vert_{C^{\alpha/2}([0,T_1])} 
			\leq K_4 \left( \left\Vert U_1 - U_2 \right\Vert_{C(D_{T_1})} 
			+ \left\Vert h_1' - h_2' \right\Vert_{L^\infty([0,T_1])} \right),
		\end{align*}
		where $K_4$ depends on $K_3$ and $\tilde C_{T_1}$ but is independent of $T$.
		If we 	choose 
		\[
		T := \min \left\{ \frac{1}{2},\, T_1,\, (\frac{K_1}{K})^{-\frac{2}{1+\alpha}},\, (\frac{K_2}{H})^{-\frac{2}{\alpha}},\, K_4^{-\frac{2}{\alpha}} \right\}, 
		\]
		then
		\begin{align*}
			\left\Vert \bar{U}_1 - \bar{U}_2 \right\Vert_{C(D_T)} 
			+ \left\Vert \bar{h}_1' - \bar{h}_2' \right\Vert_{C([0,T])} 
			&  \leq T^{\frac{1+\alpha}{2}} \left\Vert \bar{U}_1 - \bar{U}_2 \right\Vert_{C^{\frac{1+\alpha}{2},1+\alpha}(D_{T_1})} 
			+ T^{\frac{\alpha}{2}} \left\Vert \bar{h}_1' - \bar{h}_2' \right\Vert_{C^{\alpha/2}([0,T_1])}, \\
			& \leq \frac{1}{2} \left( \left\Vert U_1 - U_2 \right\Vert_{C(D_{T_1})} 
			+ \left\Vert h_1' - h_2' \right\Vert_{L^\infty([0,T_1])} \right), \\
			&  = \frac{1}{2} \left( \left\Vert U_1 - U_2 \right\Vert_{C(D_T)} 
			+ \left\Vert h_1' - h_2' \right\Vert_{L^\infty([0,T])} \right).
		\end{align*}
		Thus $\tilde{\mathcal{F}}$  is a contraction mapping on $ X_{T}$. Therefore, it has a unique fixed point  $(U, h) $ in $ X_{T}$. The maximum principle implies that $U > 0$ in $[0,T] \times (0,1]$.
		\medskip

		\noindent	$\textbf{Step\,4}.$ We apply the Schauder theory to obtain additional regularity for the solution of \eqref{2.1}.
		
		From Step 3, it follows that $U \in C^{\frac{1+\alpha}{2}, 1+\alpha}(D_T)$ and $h \in C^{1+\frac{\alpha}{2}}([0, T])$. Consequently, 		\begin{equation}\label{2.14}
			\xi(t)y = \frac{h'(t)}{h(t)}y \mbox{ lies in } C^{\frac{\alpha}{2}, \alpha}(D_T), \quad \rho(t) = d h^{-2}(t) \mbox{ lies in } C^{\frac{\alpha}{2}, \alpha}(D_T).
		\end{equation}
		Since the initial function $u_0 \in C^{2}([0, h_0])$ lacks sufficient regularity, a direct application of the Schauder theory to \eqref{2.1} is precluded. To circumvent this issue, we use a standard cutoff technique. For an arbitrary small constant $\varepsilon > 0$, select a function $\xi^{*} \in C^{\infty}([0, T])$ satisfying 
		\[
		\xi^{*}(t) = \begin{cases} 
			1 & \text{for } t \in [2\varepsilon, T], \\
			0 & \text{for } t \in [0, \varepsilon].
		\end{cases}
		\]
		Define $V := U \xi^{*}$. Substitution into \eqref{2.1} yields the modified system  
		\begin{equation}\label{2.15}
			\begin{cases}
				V_t = \rho(t) V_{yy} + \xi(t) y V_y + F(t, y), & 0 < t \leq T, \, 0 < y < 1, \\
				V(t, 1) = \delta_0 \xi^{*}(t), & 0 < t \leq T, \\
				V(t, 0) = \delta \xi^{*}(t), & 0 < t \leq T, \\
				V(0, y) = 0, & 0 \leq y \leq 1,
			\end{cases}
		\end{equation}
		where $F(t, y) = U \xi_{t}^{*} - f(U) \xi^{*}$.  
		Given $F \in C^{\frac{\alpha}{2}, \alpha}([0, T] \times [0, 1])$ and \eqref{2.14}, the Schauder estimate applied to \eqref{2.15} yields  
		\[
		\begin{aligned}
			\|U\|_{C^{1+\frac{\alpha}{2}, 2+\alpha}([2\varepsilon, T] \times [0, 1])} &\leq \|V\|_{C^{1+\frac{\alpha}{2}, 2+\alpha}([0, T] \times [0, 1])}
			\leq M \|F\|_{C^{\frac{\alpha}{2}, \alpha}([0, T] \times [0, 1])}.
		\end{aligned}
		\]
		Here $M$ depends on $\varepsilon$, $h_0$, and $\alpha$.  
		As $\varepsilon > 0$ is arbitrary, we deduce that $h \in C^{1+\frac{1+\alpha}{2}}((0, T])$ and $u \in C^{1+\frac{\alpha}{2}, 2+\alpha}(\Omega_T)$. Therefore, $(u, h)$ constitutes a classical solution of \eqref{1.1} on $\Omega_T$.
	\qed

\subsection{Proof of Lemma \ref{le2.5}}
		By the standard theory of parabolic equations, the following initial-boundary value problem
		\[\left\{\begin{array}{ll}
			v_{t}-d v_{x x}=f(v), & t \in(0, T),\, x \in(0, h(t)), \\
			v(t, 0)=\delta_{0},\,v(t, h(t))=\delta, & t \in(0, T), \\
			v(0, x)=0, & x \in\left[0, h_{0}\right],
		\end{array}\right.\]
		admits a unique solution  $v(t, x)$, which satisfies  $v(t, x)>0 $ for $ t \in(0, T)$ and $x \in(0, h(t)] $. Applying the comparison principle, we deduce that $ u(t, x) \geq v(t, x)>0$  for  all $ t \in(0, T)$ and $x \in(0, h(t)] $. This establishes the positivity of $u(t,x)$.

		Next, with  $C_1=\max _{x \in\left[0, h_{0}\right]} u_{0} (x)+1 $,  let  $w(t,x)$  be the unique solution of the following problem
		\begin{equation}\label{sw}
			\left\{\begin{array}{ll}
				w_{t}-d w_{x x}=f(w), & t \in(0, T),\, x \in(0, h(t)), \\
				w(t, 0)=C_1,\,w(t, h(t))=C_1, & t \in(0, T), \\
				w(0, x)=C_1, & x \in\left[0, h_{0}\right],
			\end{array}\right.
		\end{equation}
		Our assumption on $f$ implies that $f(C_1)<0$. Hence $C_1$ is a strict upper (stationary) solution of \eqref{sw}, which implies that $w(t,x)\leq C_1$ for $t\in[0, T)$ and $x\in [0, h(t)]$. Using the comparison principle, we deduce that 
		\(u(t,x)\leq w(t,x)\leq C_1\) for $t\in[0, T)$ and $x\in [0, h(t)]$.
		This proves the first part of the lemma.
		
	For the second part, we consider the following problem	
		\begin{equation}\label{1w}
			\left\{\begin{array}{ll}
				\hat{w}_{t}-d \hat{w}_{x x}=f(\hat{w}), & t>0,\, x \in(0, l), \\
				\hat{w}(t, 0)=0,\,	\hat{w}(t, l)=0, & t >0, \\
				\hat{w}(0, x)=	\hat{w}_0(x), & x \in\left[0, l\right],
			\end{array}\right.
		\end{equation}
	with $	\hat{w}_0\in C^1([0,l])$ satisfying 
	\[	\hat{w}_0(0)=	\hat{w}_0(l)=0, \ \mbox{ and }\ 0<	\hat{w}_0(x)<u_0(x) \ \mbox{ for }\ x\in(0,l).\]
The strong maximum principle gives that $	\hat{w}(t,x)>0$ for $t\geq0$ and $x\in(0,l)$. We can apply the comparison principle to deduce that $u(t,x)\geq 	\hat{w}(t,x)$ for $t\in[0,T)$ and $x\in[0,l]$. Hence
 \[u(t,l/2)	\geq m^*(T):= \min_{t \in [0, T]}		\hat{w}(t,l/2)>0 \mbox{ for } t\in[0,T).\]

Since $f\in C^1$ and $f(0)=0$, there exists $L>0$ such that $f(u)\geq-Lu$ for $u\in[0, C_1]$. Then consider the auxiliary ODE problem
\[
\hat{v}' = -L\hat{v}\quad \text{for } t>0, \quad \hat{v}(0) =\min \{m^*(T), \min_{[l/2,h_0]}u_0(x)\}\in(0,\delta].
\]
Direct calculation gives $\hat{v}(t)=\hat{v}(0)e^{-Lt}\in(0,v(0))$.
Comparing $u$ with $\hat{v}$, we obtain
\[u(t,x)\geq v(t)\geq v(T)=\hat{v}(0)e^{-LT}=:C^*(T)>0 \mbox{ for } t\in[0,T) \mbox{ and } x\in[l/2,h(t)].\]
This completes the proof.		
	\qed

\end{document}